\documentclass[11pt,a4paper,leqno]{article}
\usepackage{a4wide}
\usepackage{latexsym}
\usepackage{amsmath}
\usepackage{amssymb}
\usepackage{stackrel} 
\usepackage{color}
\usepackage{graphicx}

\newtheorem{defin}{Definition}
\newtheorem{lemma}{Lemma}
\newtheorem{prop}{Proposition}
\newtheorem{theo}{Theorem}
\newtheorem{corol}{Corollary}
\newenvironment{proof}{\medskip\par\noindent{\bf Proof}}{\hfill $\Box$
\medskip\par}

\newcommand{\C}{\mathbb{C}}
\newcommand{\N}{\mathbb{N}}
\newcommand{\R}{\mathbb{R}}

\newcommand{\bnu}{\boldsymbol{\nu}}
\newcommand{\bk}{\boldsymbol{k}}
\newcommand{\bd}{\boldsymbol{d}}
\newcommand{\bT}{\boldsymbol{T}}
\newcommand{\bsigma}{\boldsymbol{\sigma}}
\newcommand{\bt}{\boldsymbol{t}}
\newcommand{\bu}{\boldsymbol{u}}
\newcommand{\btau}{\boldsymbol{\tau}}

\begin{document}
\title{On parametric Gevrey asymptotics for some initial value problems in two asymmetric complex time variables}
\author{{\bf A. Lastra\footnote{The author is partially supported by the project MTM2016-77642-C2-1-P of Ministerio de Econom\'ia y Competitividad, Spain}, S. Malek\footnote{The author is partially supported by the project MTM2016-77642-C2-1-P of Ministerio de Econom\'ia y Competitividad, Spain.}}\\
University of Alcal\'{a}, Departamento de F\'{i}sica y Matem\'{a}ticas,\\
Ap. de Correos 20, E-28871 Alcal\'{a} de Henares (Madrid), Spain,\\
University of Lille 1, Laboratoire Paul Painlev\'e,\\
59655 Villeneuve d'Ascq cedex, France,\\
{\tt alberto.lastra@uah.es}\\
{\tt Stephane.Malek@math.univ-lille1.fr }}
\date{}
\maketitle
\thispagestyle{empty}
{ \small \begin{center}
{\bf Abstract}
\end{center}

We study a family of nonlinear initial value partial differential equations in the complex domain under the action of two asymmetric time variables. Different Gevrey bounds and multisummability results are obtain depending on each element of the family, providing a more complete picture on the asymptotic behavior of the solutions of PDEs in the complex domain in several complex variables.

The main results lean on a fixed point argument in certain Banach space in the Borel plane, together with a Borel summability procedure and the action of different Ramis-Sibuya type theorems.

\medskip

}
\bigskip \bigskip

\section{Introduction}

This work is framed into the study of multisummable formal solutions of certain family of PDEs. Multisummability of formal solutions of functional equations is observed in recent studies made by some research groups in different directions, and a growing interest has been observed in the scientific community. The present work belongs to these trends of studies, for which we provide a brief overview.\smallskip

Borel-Laplace summability procedures have been recently applied to solve partial differential equations. In the seminal work~\cite{lumisch}, the authors obtain positive results on the linear complex heat equation with constant coefficients. This construction was extended to more general linear PDEs by W. Balser in~\cite{ba3}, under the assumption of adequate extension of the initial data to an infinite sector. More recently, M. Hibino~\cite{hibino} has made some advances in the study of linear first order PDEs. Subsequently, several authors have studied complex heat like equations with variable coefficients (see~\cite{balo,copata,ly2}). The second author~\cite{ma1}, both authors~\cite{lama} and the two authors and J. Sanz~\cite{lamasa1} have also contributed in this theory.\smallskip 

Recently, multisummability of formal solutions of PDEs has also been put forward in different works. W. Balser~\cite{ba4} described a multisummability phenomenon in certain PDEs with constant coefficients. S. Ouchi~\cite{ou} constructed multisummable formal solutions of nonlinear PDEs, coming from perturbation of ordinary differential equations. H. Tahara and H. Yamazawa~\cite{taya} have made progresses on general linear PDEs with non constant coefficients under entire initial data. In~\cite{ly1}, G. Lysik constructs summable formal solutions of the one dimensional Burgers equation by means of the Cole-Hopf transform. O. Costin and S. Tanveer~\cite{cota2} construct summable formal power series in time variable to 3D Navier Stokes equations. The authors have obtained results in this direction~\cite{lama1,lama2}.\smallskip 

A recent overview on summability and multisummability techniques under different points of view is displayed in~\cite{loday}.\smallskip

The purpose of the present work is to study the solutions of a family of singularly perturbed partial differential equations from the asymptotic point of view. More precisely, we consider a problem of the form
\begin{multline}\label{e1}
Q(\partial_z)\partial_{t_2}u(t_1,t_2,z,\epsilon)=(P_1(\partial_z,\epsilon)u(t_1,t_2,z,\epsilon))(P_2(\partial_z,\epsilon)u(t_1,t_2,z,\epsilon))+ P(t_1,t_2,\epsilon,\partial_{t_1},\partial_{t_2},\partial_z)\\+f(t_1,t_2,z,\epsilon),
\end{multline}
under initial conditions $u(t_1,0,z,\epsilon)\equiv u(0,t_2,z,\epsilon)\equiv 0$, and where $Q(X)\in\C[X]$. The elements which conform the nonlinear part $P_1,P_2$ are polynomials in their second variable with coefficients being holomorphic functions defined on some neighborhood of the origin, say $D(0,\epsilon_0)$, continuous up to their boundary.\smallskip

Here, $D(0,\epsilon_0)$ stands for the open disc in the complex plane centered at 0, and with positive radius $\epsilon_0>0$. We write $\overline{D}(0,\epsilon_0)$ for its closure.\smallskip

Moreover, $P$ stands for some polynomial of six variables, with complex coefficients, and the forcing term $f(t_1,t_2,z,\epsilon)$ is a holomorphic and bounded function in $D(0,\rho)^2\times H_{\beta'}\times D(0,\epsilon_0)$, for some $\rho>0$, and where $H_{\beta'}$ stands for the horizontal strip
$$H_{\beta'}:=\{z\in\C:|\hbox{Im}(z)|<\beta'\},$$
for some $\beta'>0$.\smallskip

The precise configuration of the elements involved in the problem is stated and described in Section~\ref{seclayout}.\smallskip

This paper provides a step beyond in the study of the asymptotic behavior of the solutions of a subfamily of singularly perturbed partial differential equations of the form (\ref{e1}). We first recall some previous advances made in this respect, which motivate the present framework.\smallskip

In~\cite{lama}, we studied under the asymptotic point of view the solutions of certain family of PDEs of the form 
$$Q(\partial_z)\partial_tu(t,z,\epsilon)=(P_1(\partial_z,\epsilon)u(t,z,\epsilon))(P_2(\partial_z,\epsilon)u(t,z,\epsilon))+P(t,\epsilon,\partial_t,\partial_z)u(t,z,\epsilon)+f(t,z,\epsilon),$$
where the elements involved in the problem depend only on one time variable $t$. Our next aim was to check whether the asymptotic properties of the solutions in this equation can be extended to functions of more number of time variables, as stated in (\ref{e1}). \smallskip

It is worth mentioning that, in the previous work~\cite{lama}, the linear part of the equation, ruled by $P(t,\epsilon,\partial_t,\partial_z)u(t,z,\epsilon)$ was assumed to be more general than in the present configuration, admitting an additional term of the form $c_0(t,z,\epsilon)R(\partial_z)u(t,z,\epsilon)$, where $c_0(t,z,\epsilon)$ is given by a certain holomorphic function defined on a product $D(0,\rho)\times H_{\beta'}\times D(0,\epsilon_0)$. \smallskip

We decided not to incorporate this term in the present study for the sake of simplicity. However, the results can be written with no additional theoretical difficulties by adding the analog of such terms into the equation. As a matter of fact, the decision of not considering this term in the present work is due to emphasize other fact: an outstanding phenomena occurred when dealing with two complex variables, arriving at substantially and qualitatively different asymptotic properties of the solutions attained.\smallskip

In~\cite{family1}, we described a study of a family of equations of the shape (\ref{e1}) which showed a symmetric behaviour with respect to the asymptotic properties of the analytic solutions with respect to both time variables, as initially expected from the generalization of the one-time variable case. More precisely, we proved the following result: given a good covering of $\C^{\star}$, $\{ \mathcal{E}_{p_1,p_2} \}_{\begin{subarray}{l} 0 \leq p_1 \leq \varsigma_1 - 1\\0 \leq p_2 \leq \varsigma_2 - 1\end{subarray}}$ (see Definition~\ref{goodcovering}) involving sectors of opening larger than $\pi/k_2$, there exist sectors with vertex at the origin in $\C$ and finite radius, say $\mathcal{T}_1$ and $\mathcal{T}_2$, such that a family of solutions $\{ u_{p_1,p_2}(t_1,t_2,z,\epsilon) \}_{\begin{subarray}{l} 0 \leq p_1 \leq \varsigma_1 - 1\\0 \leq p_2 \leq \varsigma_2 - 1\end{subarray}}$ of (\ref{e1}) is constructed. The function $u_{p_1,p_2}(t_1,t_2,z,\epsilon)$ turns out to be holomorphic in $\mathcal{T}_1\times\mathcal{T}_2\times H_{\beta'}\times \mathcal{E}_{p_1,p_2}$, for every $0\le p_1\le \varsigma_1-1$ and $0\le p_2\le  \varsigma_2-1$. In addition to this, we obtain in this previous work that the difference of two consecutive (in the sense that they are related to consecutive sectors in the good covering) solutions $u_{p_1,p_2}$ and $u_{p'_1,p'_2}$, of (\ref{e1}) can be classified into two categories: 
\begin{enumerate}
\item Those pairs $((p_1,p_2),(p'_1,p'_2))\in\mathcal{U}_{k_1}$ such that 
$$\sup_{(t_1,t_2,z)\in \mathcal{T}_1\times\mathcal{T}_2\times H_{\beta'}}|u_{p_1,p_2}(t_1,t_2,z,\epsilon)-u_{p'_1,p'_2}(t_1,t_2,z,\epsilon)|\le K_{p}e^{-\frac{M_{p}}{|\epsilon|^{k_1}}},\quad \epsilon\in \mathcal{E}_{p_1,p_2}\cap\mathcal{E}_{p'_1,p'_2};$$
\item and those pairs $((p_1,p_2),(p'_1,p'_2))\in\mathcal{U}_{k_2}$ such that
$$\sup_{(t_1,t_2,z)\in \mathcal{T}_1\times\mathcal{T}_2\times H_{\beta'}}|u_{p_1,p_2}(t_1,t_2,z,\epsilon)-u_{p'_1,p'_2}(t_1,t_2,z,\epsilon)|\le K_{p}e^{-\frac{M_{p}}{|\epsilon|^{k_2}}},\quad \epsilon\in \mathcal{E}_{p_1,p_2}\cap\mathcal{E}_{p'_1,p'_2}.$$
\end{enumerate}
Here, $k_1$ and $k_2$ are different positive integers involved in the definition of the polynomials appearing in the main equation, and $K_p,M_p$ are positive constants.\smallskip 

The application of a two-level Ramis-Sibuya type result entails the existence of a formal power series $\hat{u}(t_1,t_2,z,\epsilon)\in\mathbb{F}[[\epsilon]]$, where $\mathbb{F}$ stands for the Banach space of holomorphic and bounded functions in the domain $\mathcal{T}_1\times\mathcal{T}_2\times H_{\beta'}$, with the supremum norm. Such formal power series is a formal solution of (\ref{e1}) and can be split in the form
$$\hat{u}(t_1,t_2,z,\epsilon)=a(t_1,t_2,z,\epsilon)+\hat{u}_1(t_1,t_2,z,\epsilon)+\hat{u}_2(t_1,t_2,z,\epsilon),$$
where $a(t_1,t_2,z,\epsilon)$ belongs to $\mathbb{F}\{\epsilon\}$, and $\hat{u}_1,\hat{u}_2\in\mathbb{F}[[\epsilon]]$. Moreover, for all $p_1\in\{0,\ldots,\varsigma_1-1\}$ and $p_2\in\{0,\ldots,\varsigma_2-1\}$, the function $u_{p_1,p_2}(t_1,t_2,z,\epsilon)$ can be split analogously:
$$u_{p_1,p_2}(t_1,t_2,z,\epsilon)=a(t_1,t_2,z,\epsilon)+u^1_{p_1,p_2}(t_1,t_2,z,\epsilon)+u^2_{p_1,p_2}(t_1,t_2,z,\epsilon),$$
where $\epsilon\mapsto u^j_{p_1,p_2}(t_1,t_2,z,\epsilon)$ is an $\mathbb{F}$-valued function which admits $\hat{u}_{j}(t_1,t_2,z,\epsilon)$ as its $k_j$-Gevrey asymptotic expansion on $\mathcal{E}_{p_1,p_2}$, for $j=1,2$, seeing $\hat{u}_{j}$ as a formal power series in $\epsilon$, with coefficients in $\mathbb{F}$. In addition to this, and under the assumption that $k_1<k_2$, a multisummability result is also attained. Under the assumption that 
$$\{((p_1^0,p_2^0),(p_1^1,p_2^1)), ((p_1^1,p_2^1),(p_1^2,p_2^2)),\ldots, ((p_1^{2y-1},p_2^{2y-1}),(p_1^{2y},p_2^{2y}))\}\subseteq\mathcal{U}_{k_2}$$
for some $y\in\N:=\{1,2,\ldots\}$, and
$$\mathcal{E}_{p_1^y,p_2^y}\subseteq S_{\pi/k_1}\subseteq\bigcup_{0\le j\le 2y}\mathcal{E}_{p_1^j,p_2^j},$$
for some sector $S_{\pi/k_1}$ with opening larger than $\pi/k_1$, then it holds that $\hat{u}(t_1m,t_2,z,\epsilon)$ is indeed $(k_2,k_1)-$summable on $\mathcal{E}_{p_1^y,p_2^y}$, being its $(k_2,k_1)$-sum given by $u_{p_1^y,p_2^y}$ on $\mathcal{E}_{p_1^y,p_2^y}$.

The role played by $k_1$ and $k_2$ in the previous framework is completely symmetric. The assumption $k_1<k_2$ is innocuous, reaching symmetric results in the case that $k_2<k_1$. In that study, the principal part of any of the equations in the family studied is factorisable as a product of two operators involving a single time variable, yielding a multisummability phenomena in the perturbation parameter $\epsilon$.\smallskip

On the other hand, in the present study, the sign of $k_1-k_2$ is crucial at the time of studying the asymptotic behavior of the analytic solution. In fact, a negative sign provides less information on the asymptotic behavior, which entails only Gevrey estimates whilst the positive one furnishes more precise information, namely multisummability. Here is where the strength of the present results holds. More precisely, we find a family of analytic solutions $\{ u_{p_1,p_2}(t_1,t_2,z,\epsilon) \}_{\begin{subarray}{l} 0 \leq p_1 \leq \varsigma_1 - 1\\0 \leq p_2 \leq \varsigma_2 - 1\end{subarray}}$ of the main problem under study, which are holomorphic in $\mathcal{T}_1\times\mathcal{T}_2\times H_{\beta'}\times \mathcal{E}_{p_1,p_2}$, and such that one of the following hold:
\begin{enumerate}
\item In case $k_2>k_1$, a formal power series $\hat{u}(t_1,t_2,z,\epsilon)\in\mathbb{F}[[\epsilon]]$, formal solution of (\ref{e1}), exists such that for every $(p_1,p_2)\in \{0,\ldots,\varsigma_1-1\}\times\{0,\ldots,\varsigma_2-1\}$, the function $u_{p_1,p_2}(t_1,t_2,z,\epsilon)$ admits $\hat{u}(t_1,t_2,z,\epsilon)$ as its asymptotic expansion of Gevrey order $1/k_1$ in $\mathcal{E}_{p_1,p_2}$ (see Theorem~\ref{teo2}).
\item In case that $k_1>k_2$, a formal power series $\hat{u}(t_1,t_2,z,\epsilon)\in\mathbb{F}[[\epsilon]]$ exists, being formal solution of (\ref{e1}), and such that $\hat{u}(t_1,t_2,z,\epsilon)$ shows analogous properties as those described in the family of equations in~\cite{family1}, i.e. multisummability of the formal solution with Gevrey levels $k_1$ and $k_2$ (see Theorem~\ref{teo3}).
\end{enumerate}

The present study is based on the following approach: after establishing the main problem under study:
\begin{multline}
\left(Q(\partial_{z})\partial_{t_2}+\epsilon^{\Delta_1}t_1^{d_1}\partial_{t_1}^{\delta_{D_1}}
\epsilon^{\tilde{\Delta}_2}t_2^{\tilde{d}_2}\partial_{t_2}^{\tilde{\delta}_{D_2}}R_{D_1,D_2}(\partial_z)+\epsilon^{\tilde{\Delta}_3}t_2^{\tilde{d}_3}\partial_{t_2}^{\tilde{\delta}_{D_3}}R_{D_3}(\partial_z)\right)u(\bt,z,\epsilon)\\
= (P_1(\partial_z,\epsilon)u(\bt,z,\epsilon))(P_2(\partial_z,\epsilon)u(\bt,z,\epsilon)) + \sum_{\stackrel{0\le l_j\le D_j-1}{j=1,2}} \epsilon^{\Delta_{l_1,l_2}}t_1^{d_{l_1}}t_2^{\tilde{d}_{l_2}}\partial_{t_1}^{\delta_{l_1}}\partial_{t_2}^{\tilde{\delta}_{l_2}}R_{l_1,l_2}(\partial_z)u(\bt,z,\epsilon)\\
+ f(\bt,z,\epsilon), \label{ICP_main00}
\end{multline}
where $k_1,k_2\ge1$, $D_1,D_2\ge 2$, $\Delta_1,d_1,\delta_{D_1},\tilde{\Delta}_{2},\tilde{d}_2,\tilde{\delta}_{D_2},\tilde{\Delta}_{3},\tilde{d}_3,\tilde{\delta}_{D_3}$ are integer numbers, and for all $0\le l_1\le D_1-1$ and $0\le l_2\le D_2-1$, we take nonnegative integers $d_{l_1},\tilde{d}_{l_2},\delta_{l_1},\tilde{d}_{l_2}$, and $\Delta_{l_1,l_2}$, under the assumptions (\ref{e120})-(\ref{e331}). Moreover, $Q, R_{D_1,D_2},R_{D_3}$ and $R_{l_1,l_2}$ are polynomials with complex coefficients, for all $0\le l_1\le D_1-1$ and $0\le l_2\le D_2-1$. The polynomials $P_1,P_2$ present coefficients which are holomorphic functions with respect the perturbation parameter on some neighborhood of the origin, under assumptions (\ref{raicesgrandes})-(\ref{e165b}). The forcing term $f(t_1,t_2,z,\epsilon)$ is given by some holomorphic and bounded function on a neighborhood of the origin with respect to both variables and the perturbation parameter $\epsilon$, and some horizontal strip with respect to $z$ variable.\smallskip

We search for analytic solutions of (\ref{ICP_main00}) given as a Laplace and Fourier transform of certain function to be determined:
\begin{equation}\label{e2}
u(t_1,t_2,z,\epsilon)=\frac{k_1k_2}{(2\pi)^{1/2}}\int_{-\infty}^{\infty}\int_{L_{d_1}}\int_{L_{d_2}}\omega_{\bk}^{\bd}(u_1,u_2,m,\epsilon)e^{-\left(\frac{u_1}{\epsilon t_1}\right)^{k_1}-\left(\frac{u_2}{\epsilon t_2}\right)^{k_2}}\frac{du_2}{u_2}\frac{du_1}{u_1},
\end{equation}
where $L_{\gamma_{j}}=\R_{+}e^{i\gamma_j}$, for some appropriate direction $\gamma_j\in\R$, for $j=1,2$. The problem of finding such a function is equivalent (in view of Lemma~\ref{lema257}) to solve an auxiliary convolution equation in the Borel plane. More precisely, there is a one-to-one correspondence between functions $u(t_1,t_2,z,\epsilon)$ of the form (\ref{e2}), which solve (\ref{ICP_main00}), and functions $\omega(\tau_1,\tau_2,m,\epsilon)$ admitting Laplace transform with respect to the first two variables along directions $d_1$ and $d_2$ resp., and Fourier transform with respect to $m$ variable, which turn out to be solutions of a convolution equation (see~\ref{e310}).\smallskip

For every fixed value of the perturbation parameter $\epsilon$,  $(\tau_1,\tau_2,m)\mapsto \omega_{\bk}^{\bd}(\tau_1,\tau_2,m,\epsilon)$ is obtained as the fixed point of the contractive operator $\mathcal{H}_\epsilon$ (see~(\ref{e623}) for its definition) acting on some Banach space of functions owing exponential decay at infinity on the Fourier variable, and defined on some neighborhood of the origin for $(\tau_1,\tau_2)$ in $\C^2$, which can be prolonged to some neighborhood of the origin together with an infinite sector of bisecting direction $d_1$ times an infinite sector with bisecting direction $d_2$; under certain concrete monomial exponential growth at infinity. More precisely, $\omega_{\bk}^{\bd}(\tau_1,\tau_2,m,\epsilon)$ is a continuous function in $(\overline{D}(0,\rho)\cup S_{d_1})\times S_{d_2}\times \R\times D(0,\epsilon_0)\setminus\{0\}$, and holomorphic with respect to $(\tau_1,\tau_2)$ in $(D(0,\rho)\cup S_{d_1})\times S_{d_2}$, and on $D(0,\epsilon_0)\setminus\{0\}$ with respect to the perturbation parameter. In addition to this, there exist constants $\varpi,\mu,\beta,\nu_1,\nu_2>0$ such that
$$|\omega_{\bk}^{\bd}(\tau_1,\tau_2,m,\epsilon)|\le\varpi (1+|m|)^{-\mu}\frac{\left|\frac{\tau_1}{\epsilon}\right|}{1+\left|\frac{\tau_1}{\epsilon}\right|^{2k_1}}\frac{\left|\frac{\tau_2}{\epsilon}\right|}{1+\left|\frac{\tau_2}{\epsilon}\right|^{2k_2}}\exp\left(-\beta|m|+\nu_1\left|\frac{\tau_1}{\epsilon}\right|^{k_1}+\nu_2\left|\frac{\tau_2}{\epsilon}\right|^{k_2}\right),$$
for every $(\tau_1,\tau_2,m,\epsilon)\in (\overline{D}(0,\rho)\cup S_{d_1})\times S_{d_2}\times \R\times D(0,\epsilon_0)\setminus\{0\}$. Laplace and Fourier transforms make sense in order to get (\ref{e2}). At this point, we are able to construct a family of solutions $\{ u_{p_1,p_2}(t_1,t_2,z,\epsilon) \}_{\begin{subarray}{l} 0 \leq p_1 \leq \varsigma_1 - 1\\0 \leq p_2 \leq \varsigma_2 - 1\end{subarray}}$ of (\ref{ICP_main00}), where $u_{p_1,p_2}(t_1,t_2,z,\epsilon)$ is a holomorphic function defined in $\mathcal{T}_1\times\mathcal{T}_2\times H_{\beta'}\times \mathcal{E}_{p_1,p_2}$, with $\mathcal{T}_1$ and $\mathcal{T}_2$ being finite sectors in $\C^{\star}$ with vertex at the origin, and where $\{ \mathcal{E}_{p_1,p_2} \}_{\begin{subarray}{l} 0 \leq p_1 \leq \varsigma_1 - 1\\0 \leq p_2 \leq \varsigma_2 - 1\end{subarray}}$ conforms a good covering at 0 (see Definition~\ref{goodcovering}).\smallskip

The distinction of $k_1<k_2$ and $k_2<k_1$ provide Gevrey asymptotics or multisummability results in Theorem~\ref{teo2}, resp. Theorem~\ref{teo3}. It is worth mentioning that these results lean on the application of a cohomological criteria known as Ramis-Sibuya Theorem; resp. a multilevel version of such result. \smallskip

The fact that a different behavior can be observed with respect to the variables in time is due to the domain of definition of $\omega_{\bk}^{\bd}$ with respect to such variables: a neighborhood of the origin for $(\tau_1,\tau_2)\in\C^2$ which can only be prolonged up to a neighborhood of the origin together with an infinite sector with respect to the first variable; whereas it can not be defined on any neighborhood of the origin with respect to the second time variable, but it does on some infinite sector. This causes the impossibility of application of a deformation path at the time of estimating the difference of two consecutive solutions in order to apply the multilevel version Ramis-Sibuya Theorem. With respect to the study of the main problem in~\cite{family1}, the main difficulty at this point comes due to the fact that Case 1 in Theorem 1 of~\cite{family1} is no longer available.\smallskip

We also find it necessary to justify the fact that $\omega_{\bk}^{\bd}$ can not be defined with respect to $(\tau_1,\tau_2)$ in sets of the form 
\begin{equation}\label{e145}
S_1\times (S_2\cup D(0,\rho_2)),
\end{equation}
for some infinite sectors $S_1$ and $S_2$, and for some $\rho_2>0$. In order to solve the main equation, one needs to divide by $P_m(\tau_1,\tau_2)$ (see (\ref{e326}) for its definition). However, as stated in Section~\ref{secnodef}, the roots of such polynomial lie on sets of the form (\ref{e145}), for any $\rho_2>0$. Therefore, a small divisor phenomena is observed , which does not allow a summability procedure. This occurrence has already been noticed in another context in previous works: in the framework of $q-$difference-differential equations~\cite{lama3}; in the context of multilevel Gevrey solutions of PDEs in the complex domain in~\cite{lama2}, etc.\smallskip

The layout of the paper is as follows.\smallskip

After recalling the definition and the action of Fourier transform in the first part of Section~\ref{sec2}, we describe the main problem under study in Section~\ref{seclayout}, and reduce it to the research of a solution of an auxiliary convolution equation. Such solution is obtained following a fixed point argument in appropriate Banach spaces (see Section~\ref{fixed}), whose main properties are provided in Section~\ref{subsecespbanach}. 

Section~\ref{secnodef} is devoted to motivate the domain of definition of the solution, in contrast to that studied in~\cite{family1}. 

The first main result of our work is Theorem~\ref{teo1}, where the existence of a family of analytic solutions of the main problem is obtained. In Section~\ref{secborellaplace} we recall the Borel summability procedure and two cohomological criteria: Ramis-Sibuya Theorem, and a multilevel version of Ramis-Sibuya Theorem. We conclude the present work with the existence of a formal solution to the problem, and two asymptotic results which connect the formal and the analytic solutions: Theorem~\ref{teo2} states a result on Gevrey asymtotics in a subfamily of equations; Theorem~\ref{teo3} states a multisummability result in another different subfamily of equations under study.

\section{Layout of the main and auxiliary problems}\label{sec2}

This section is devoted to describe the main problem under study. We first recall some facts on the action of Fourier transform on certain Banach spaces of functions.

\subsection{Fourier transform on exponentially decreasing function spaces}

In order to transform the main problem under study into an auxiliary one, easier to handle, we first describe the action of Fourier transform in certain Banach spaces of rapidly decreasing functions.  

\begin{defin} Let $\beta, \mu \in \mathbb{R}$. $E_{(\beta,\mu)}$ stands for the vector space of continuous functions $h : \mathbb{R} \rightarrow \mathbb{C}$ such that
$$ \left\|h(m)\right\|_{(\beta,\mu)} = \sup_{m \in \mathbb{R}} (1+|m|)^{\mu} \exp( \beta |m|) |h(m)| $$
is finite. $E_{(\beta,\mu)}$ turns out to be a Banach space when endowed with the norm $\left\|.\right\|_{(\beta,\mu)}$.
\end{defin}

The following result is stated without proof, which can be found in~\cite{lama}, Proposition 7.

\begin{prop}\label{prop359b}
Let $f \in E_{(\beta,\mu)}$ with $\beta > 0$, $\mu > 1$. The inverse Fourier transform of $f$
$$ \mathcal{F}^{-1}(f)(x) = \frac{1}{ (2\pi)^{1/2} } \int_{-\infty}^{+\infty} f(m) \exp( ixm ) dm,\quad x\in\mathbb{R},$$
can be extended to an analytic function on the strip
$$H_{\beta} := \{ z \in \mathbb{C} / |\mathrm{Im}(z)| < \beta \}.
$$
Let $\phi(m) = im f(m) \in E_{(\beta,\mu - 1)}$. Then, it holds that $ \partial_{z} \mathcal{F}^{-1}(f)(z) = \mathcal{F}^{-1}(\phi)(z)$, for $z \in H_{\beta}$.

Let $g \in E_{(\beta,\mu)}$ and put $\psi(m) = \frac{1}{(2\pi)^{1/2}} f \ast g(m)$, the convolution product of $f$ and $g$, for all $m \in \mathbb{R}$. $\psi$ belongs to  $E_{(\beta,\mu)}$. Moreover, we have $\mathcal{F}^{-1}(f)(z)\mathcal{F}^{-1}(g)(z) = \mathcal{F}^{-1}(\psi)(z)$, for  $z\in H_{\beta}$.
\end{prop}

\subsection{Layout of the main problem}\label{seclayout}

Let $k_1,k_2 \geq 1$ and $D_1,D_2 \geq 2$ be integer numbers. We also consider non negative integer numbers $d_1,\tilde{d}_j,\Delta_1,\tilde{\Delta}_j,\delta_{D_1},\tilde{\delta}_{D_j}$, for $j\in\{2,3\}$. For all $0\le l_1\le D_1-1$ and $0\le l_2\le D_2-1$, let $d_{l_1},\tilde{d}_{l_2},\delta_{l_1},\tilde{d}_{l_2},\Delta_{l_1,l_2}$ be non negative integers.
We assume the previous elements satisfy the next identities: 
\begin{equation}\label{e120}
\frac{2}{k_2}<\tilde{\delta}_{D_2}\le \tilde{\delta}_{D_3} 
\end{equation}
and $\delta_{l_1} < \delta_{l_1+1}$, $\tilde{\delta}_{l_2} < \tilde{\delta}_{l_2+1}$ 
for all $0 \leq l_1 \leq D_1-1$ and $0 \leq l_2 \leq D_2-1$,
\begin{multline}
\Delta_1+\tilde{\Delta}_2-d_1-\tilde{d}_2-1+\delta_{D_1}+\tilde{\delta}_{D_2}=0\qquad \tilde{\Delta}_3-\tilde{d}_3+\tilde{\delta}_{D_3}-1=0\\
d_1=\delta_{D_1}(k_1+1),\qquad k_2+1+\tilde{d}_j=\tilde{\delta}_{D_j}(k_2+1)\quad (j=2,3)
\label{e232}
\end{multline}
Moreover, for every $0\le l_1\le D_1-1$ and $0\le l_2\le D_2-1$, we assume
\begin{equation}\label{e331}
d_{l_1}>\delta_{l_1}(k_1+1),\qquad \tilde{d}_{l_2}>(\tilde{\delta}_{l_2}-1)(k_2+1),
\end{equation}
$$ \Delta_{l_1,l_2}>\delta_{D_1}k_1+(\tilde{\delta}_{D_2}-1)k_2.$$

Let $Q(X), R_{D_1,D_2},R_{D_3}\in\C[X]$, and for $0\le l_1\le D_1-1$ and $0\le l_2\le D_2-1$ we take $R_{l_1,l_2}(X)\in\C[X]$. We consider polynomials $P_1,P_2$ with coefficients belonging to $\mathcal{O}(\overline{D}(0,\epsilon_0))$, such that
\begin{equation}\label{raicesgrandes}
\hbox{deg}(Q)\ge \hbox{deg}(R_{l_1,l_2}),
\end{equation}
for $0\le l_1\le D_1-1$ and $0\le l_2\le D_2-1$. Moreover, we choose these polynomials satisfying
\begin{equation}
\mathrm{deg}(Q) \geq \mathrm{deg}(P_{j}),\quad j=1,2\qquad Q(im)\neq0,\quad m\in\R,\label{assum_deg_Q_R}
\end{equation}
$$\mathrm{deg}(Q)= \mathrm{deg}(R_{D_3})= \mathrm{deg}(R_{D_1,D_2}).$$
More precisely, we assume there exist sectorial annulus $E_{D_3,Q}$ and $E_{D_1,D_2,D_3}$ such that \begin{equation}\label{e165b}
\frac{R_{D_3}(im)}{Q(im)}\in E_{D_3,Q}, \qquad \frac{R_{D_1,D_2}(im)}{Q(im)}\in E_{D_1,D_2,D_3}, 
\end{equation}
for every $m\in\R$. In other words, there exist real numbers $0<r_{j}<R_j$ and $\alpha_j<\beta_j$, for $j=1,2,$ such that
\begin{multline}
E_{D_3,Q}:=\{x\in\C:r_1<|x|<R_1, \hbox{arg}(x)\in(\alpha_1,\beta_1)\},\\
E_{D_1,D_2,Q}:=\{x\in\C:r_2<|x|<R_2, \hbox{arg}(x)\in(\alpha_2,\beta_2)\}.\label{eannulus}
\end{multline}

Throughout the whole work, we denote the pairs of variables in bold letters: $\bt:=(t_1,t_2)$, $\bT:=(T_1,T_2)$, $\btau:=(\tau_1,\tau_2)$, etc.

We consider the following nonlinear initial value problem
\begin{multline}
\left(Q(\partial_{z})\partial_{t_2}+\epsilon^{\Delta_1}t_1^{d_1}\partial_{t_1}^{\delta_{D_1}}
\epsilon^{\tilde{\Delta}_2}t_2^{\tilde{d}_2}\partial_{t_2}^{\tilde{\delta}_{D_2}}R_{D_1,D_2}(\partial_z)+\epsilon^{\tilde{\Delta}_3}t_2^{\tilde{d}_3}\partial_{t_2}^{\tilde{\delta}_{D_3}}R_{D_3}(\partial_z)\right)u(\bt,z,\epsilon)\\
= (P_1(\partial_z,\epsilon)u(\bt,z,\epsilon))(P_2(\partial_z,\epsilon)u(\bt,z,\epsilon)) + \sum_{\stackrel{0\le l_j\le D_j-1}{j=1,2}} \epsilon^{\Delta_{l_1,l_2}}t_1^{d_{l_1}}t_2^{\tilde{d}_{l_2}}\partial_{t_1}^{\delta_{l_1}}\partial_{t_2}^{\tilde{\delta}_{l_2}}R_{l_1,l_2}(\partial_z)u(\bt,z,\epsilon)\\
+ f(\bt,z,\epsilon) \label{ICP_main0}
\end{multline}
with null initial data $u(t_1,0,z,\epsilon)\equiv u(0,t_2,z,\epsilon) \equiv 0$.\smallskip

The forcing term $f(\bt,z,\epsilon)$ is constructed as follows. For $n_1,n_2 \geq 1$, let $m \mapsto F_{n_1,n_2}(m,\epsilon)$  be a family of functions belonging to the Banach space
$E_{(\beta,\mu)}$ for some $\beta > 0$, $\mu > \max( \mathrm{deg}(P_{1})+1, \mathrm{deg}(P_{2})+1)$ and which
depend holomorphically on $\epsilon \in D(0,\epsilon_{0})$. We assume there exist constants $K_{0},T_{0}>0$
such that 
\begin{equation}\label{e165}
\left\|F_{n_1,n_2}(m,\epsilon)\right\|_{(\beta,\mu)} \leq K_{0} (\frac{1}{T_{0}})^{n_1+n_2},
\end{equation}
for all $n_1,n_2 \geq 1$, and $\epsilon \in D(0,\epsilon_{0})$. We deduce that
$$\mathbf{F}(\bT,z,\epsilon) = \sum_{n_1,n_2 \geq 1} \mathcal{F}^{-1}(m \mapsto F_{n_1,n_2}(m,\epsilon))(z) T_1^{n_1}T_2^{n_2} $$
represents a bounded and holomorphic function on $D(0,T_{0}/2)^2 \times H_{\beta'} \times D(0,\epsilon_{0})$ for any $0 < \beta' < \beta$. We define
\begin{equation}
f(\bt,z,\epsilon) = \mathbf{F}(\epsilon t_1,\epsilon t_2 , z,\epsilon).
\label{defin_c_0_f}
\end{equation}
Observe the function $f$ is holomorphic and bounded on $D(0,\rho)^2 \times H_{\beta'} \times D(0,\epsilon_{0})$ where $\rho \epsilon_{0} < T_{0}/2$.

We search for solutions of the main problem (\ref{ICP_main0}), which are time scaled and expressed as a Fourier transform with respect to $z$ variable, in the form 
$$u(\bt,z,\epsilon)=\mathbb{U}(\epsilon t_1,\epsilon t_2,z,\epsilon)=\frac{1}{(2\pi)^{1/2}}\int_{-\infty}^{\infty}U(\epsilon t_1,\epsilon t_2,m,\epsilon)\exp(izm)dm.$$  
The symbol $U(\bT,m,\epsilon)$ satisfies the next equation.

\begin{multline}
\left(Q(im)\epsilon\partial_{T_2}+ \epsilon^{\Delta_1+\tilde{\Delta}_2-d_1-\tilde{d}_2}T_1^{d_1}T_2^{\tilde{d}_2}\epsilon^{\delta_{D_1}+\tilde{\delta}_{D_2}}\partial_{T_1}^{\delta_{D_1}}\partial_{T_2}^{\tilde{\delta}_{D_2}}R_{D_1,D_2}(im)\right.\\
\left.+\epsilon^{\tilde{\Delta}_3-\tilde{d}_3}T_2^{\tilde{d}_3}\epsilon^{\tilde{\delta}_{D_3}}\partial_{T_2}^{\tilde{\delta}_{D_3}}R_{D_3}(im)\right)U(\bT,m,\epsilon)\\
=\frac{1}{(2\pi)^{1/2}}\int_{-\infty}^{\infty}(P_1(i(m-m_1),\epsilon)U(\bT,m-m_1,\epsilon))(P_2(im_1,\epsilon)U(\bT,m_1,\epsilon))dm_1\\
+\sum_{\stackrel{0\le l_j\le D_j-1}{j=1,2}} \epsilon^{\Delta_{l_1,l_2}-d_{l_1}-\tilde{d}_{l_2}+\delta_{l_1}+\tilde{\delta}_{l_2}}T_1^{d_{l_1}}T_2^{\tilde{d}_{l_2}}\partial_{T_1}^{\delta_{l_1}}\partial_{T_2}^{\tilde{\delta}_{l_2}}R_{l_1,l_2}(im)U(\bT,m,\epsilon)\\
+ \mathcal{F}(z\mapsto \mathbf{F}(\bT,z,\epsilon))(m)\label{e32b}
\end{multline}

Our goal is to provide solutions of (\ref{e32b}) in the form of a Laplace transform. Namely, we search for solutions of the form
\begin{equation}\label{e215}
U(\bT,m, \epsilon)=k_1k_2 \int_{L_{\gamma_1}}\int_{L_{\gamma_2}} \omega_{\bk}^{\bd}(u_1,u_2,m,\epsilon) e^{-(\frac{u_1}{T_1})^{k_1}-(\frac{u_2}{T_2})^{k_2}} \frac{du_2}{u_2}\frac{du_1}{u_1},
\end{equation}
where $L_{\gamma_j}=\mathbb{R}_{+}e^{i\gamma_j}$, for some appropriate direction $\gamma_j\in\R$, for $j=1,2,$ which depend on $T_j$. The function $\omega_{\bk}^{\bd}(\btau,m,\epsilon)$ is constructed in the incoming sections as the fixed point of a map defined in certain Banach spaces, studied in the forthcoming sections. For $j=1,2,$ let $S_{d_j}$ be infinite sectors with vertex at the origin and bisecting direction $d_j$, such that $L_{\gamma_j}\subseteq S_{d_j}$. We fix a positive real number $\rho>0$.

In the present section, we depart from a function $\omega_{\bk}^{\bd}(\btau,m,\epsilon)$ continuous on $(\overline{D}(0,\rho)\cup S_{d_1})\times S_{d_2}\times \R\times D(0,\epsilon_0)\setminus\{0\}$, holomorphic with respect to $(\btau,\epsilon)$ in $(D(0,\rho)\cup S_{d_1})\times S_{d_2}\times D(0,\epsilon_0)$, and such that
\begin{equation}\label{e209}
|\omega_{\bk}^{\bd}(\btau,m,\epsilon)| \leq \varpi_{\bd}(1+ |m|)^{-\mu} e^{-\beta|m|}
\frac{ |\frac{\tau_1}{\epsilon}|}{1 + |\frac{\tau_1}{\epsilon}|^{2k_1}}\frac{ |\frac{\tau_2}{\epsilon}|}{1 + |\frac{\tau_2}{\epsilon}|^{2k_2}} \exp( \nu_1 |\frac{\tau_1}{\epsilon}|^{k_1}+\nu_2 |\frac{\tau_2}{\epsilon}|^{k_2})
\end{equation}
for all $\btau\in (\overline{D}(0,\rho)\cup S_{d_1})\times S_{d_2}$, every $m\in\R$ and $\epsilon\in D(0,\epsilon_0)\setminus\{0\}$.

In order to construct the solution, we present a refined form of the problem. For that purpose, we need some preliminary results. We make use of the following relations, which can be found in~\cite{taya}, p. 40:
\begin{multline}
T_1^{\delta_{D_1}(k_1+1)} \partial_{T_1}^{\delta_{D_1}} = (T_1^{k_1+1}\partial_{T_1})^{\delta_{D_1}} +
\sum_{1 \leq p_1 \leq \delta_{D_1}-1} A_{\delta_{D_1},p_1} T_1^{k_1(\delta_{D_1}-p_1)} (T_1^{k_1+1}\partial_{T_1})^{p_1}\\
=(T_1^{k_1+1}\partial_{T_1})^{\delta_{D_1}}+A_{\delta_{D_1}}(T_1,\partial_{T_1})
 \label{expand_op_diff}
\end{multline}
\begin{multline}
T_2^{\tilde{\delta}_{D_j}(k_2+1)} \partial_{T_2}^{\tilde{\delta}_{D_j}} = (T_2^{k_2+1}\partial_{T_2})^{\tilde{\delta}_{D_j}} +
\sum_{1 \leq p_j \leq \tilde{\delta}_{D_j}-1} \tilde{A}_{\tilde{\delta}_{D_j},p_j} T_2^{k_2(\tilde{\delta}_{D_j}-p_j)} (T_2^{k_2+1}\partial_{T_2})^{p_j}\\
=(T_2^{k_2+1}\partial_{T_2})^{\tilde{\delta}_{D_j}}+\tilde{A}_{\tilde{\delta}_{D_j}}(T_2,\partial_{T_2}) \label{expand_op_diff2}
\end{multline}
for some real numbers $A_{\delta_{D_1},p_1}$, $p_1=1,\ldots,\delta_{D_1}-1$ and $\tilde{A}_{\tilde{\delta}_{D_j},p_j}$, $p_j=1,\ldots,\tilde{\delta}_{D_j}-1$, for $j=2,3$. We write $A_{D_1}$ (resp. $\tilde{A}_{D_j}$, for $j=2,3$,) in the place of $A_{\delta_{D_1}}$ (resp. $\tilde{A}_{\tilde{\delta}_{D_j}}$) for the sake of simplicity.

We divide by $\epsilon$ and multiply by $T_2^{k_2+1}$ at both sides of (\ref{e32b}). Under the assumptions displayed in (\ref{e232}) one may apply (\ref{expand_op_diff}) and (\ref{expand_op_diff2}) in order to rewrite equation (\ref{e32b}). This step is important to exhibit the equations as an expression where some operators algebraically well-behaved with respect to Laplace transform appear. The resulting equation is as follows:
\begin{multline}
\left(Q(im)T_2^{k_2+1}\partial_{T_2}+ (T_1^{k_1+1}\partial_{T_1})^{\delta_{D_1}}(T_2^{k_2+1}\partial_{T_2})^{\tilde{\delta}_{D_2}} R_{D_1,D_2}(im)+(T_2^{k_2+1}\partial_{T_2})^{\tilde{\delta}_{D_3}}R_{D_3}(im)\right)U(\bT,m,\epsilon)\\
=\left[-(T_1^{k_1+1}\partial_{T_1})^{\delta_{D_1}}\tilde{A}_{D_2}(T_2,\partial_{T_2})R_{D_1,D_2}(im)-(T_2^{k_2+1}\partial_{T_2})^{\tilde{\delta}_{D_2}}A_{D_1}(T_1,\partial_{T_1})R_{D_1,D_2}(im)\right.\\
\left.-A_{D_1}(T_1,\partial_{T_1})\tilde{A}_{D_2}(T_2,\partial_{T_2})R_{D_1,D_2}(im)- \tilde{A}_{D_3}(T_2,\partial_{T_2})R_{D_3}(im)\right] U(\bT,m,\epsilon)\\
+\frac{T_2^{k_2+1}\epsilon^{-1}}{(2\pi)^{1/2}}\int_{-\infty}^{\infty}(P_1(i(m-m_1),\epsilon)U(\bT,m-m_1,\epsilon))(P_2(im_1,\epsilon)U(\bT,m_1,\epsilon))dm_1\\
+\sum_{\stackrel{0\le l_j\le D_j-1}{j=1,2}} \epsilon^{\Delta_{l_1,l_2}-d_{l_1}-\tilde{d}_{l_2}+\delta_{l_1}+\tilde{\delta}_{l_2}-1}T_1^{d_{l_1}}T_2^{\tilde{d}_{l_2}}\partial_{T_1}^{\delta_{l_1}}\partial_{T_2}^{\tilde{\delta}_{l_2}}R_{l_1,l_2}(im)U(\bT,m,\epsilon)\\
+ T_2^{k_2+1}\epsilon^{-1}\mathcal{F}(z\mapsto \mathbf{F}(\bT,z,\epsilon))(m).\label{e32c}
\end{multline}

The following result allows to establish a one-to-one correspondence between solutions of equation (\ref{e32c}), and an auxiliary equation in the Borel plane, (\ref{e310}). The last equation will be presented afterwards, in this same section. 

\begin{lemma}\label{lema257}
Let $U(\bT,m,\epsilon)$ be the function constructed in (\ref{e215}). Then, the following statements hold:
$$T_j^{k_j+1}\partial_{T_j}U(\bT,m,\epsilon)=k_1k_2\int_{L_{\gamma_1}}\int_{L_{\gamma_2}}(k_ju_j^{k_j})\omega_{\bk}^{\bd}(\bu,m,\epsilon)e^{-\left(\frac{u_1}{T_1}\right)^{k_1}-\left(\frac{u_2}{T_2}\right)^{k_2}}\frac{du_2}{u_2}\frac{du_1}{u_1},\quad j=1,2.$$

\begin{multline*}T_1^{m_1}U(\bT,m,\epsilon)=k_1k_2\int_{L_{\gamma_1}}\int_{L_{\gamma_2}}\left(\frac{u_1^{k_1}}{\Gamma\left(\frac{m_1}{k_1}\right)}\int_0^{u_1^{k_1}}(u_1^{k_1}-s_1)^{\frac{m_1}{k_1}-1}\omega_{\bk}^{\bd}(s_1^{1/k_1},u_2,m,\epsilon)\frac{ds_1}{s_1}\right)\\
\times e^{-\left(\frac{u_1}{T_1}\right)^{k_1}-\left(\frac{u_2}{T_2}\right)^{k_2}}\frac{du_2}{u_2}\frac{du_1}{u_1},\qquad m_1\in\N.
\end{multline*}

\begin{multline*}
T_2^{m_2}U(\bT,m,\epsilon)=k_1k_2\int_{L_{\gamma_1}}\int_{L_{\gamma_2}}\left(\frac{u_2^{k_2}}{\Gamma\left(\frac{m_2}{k_2}\right)}\int_0^{u_2^{k_2}}(u_2^{k_2}-s_2)^{\frac{m_2}{k_2}-1}\omega_{\bk}^{\bd}(u_1,s_2^{1/k_2},m,\epsilon)\frac{ds_2}{s_2}\right)\\
\times e^{-\left(\frac{u_1}{T_1}\right)^{k_1}-\left(\frac{u_2}{T_2}\right)^{k_2}}\frac{du_2}{u_2}\frac{du_1}{u_1},\qquad m_2\in\N.
\end{multline*}

\begin{multline*}
\int_{-\infty}^{\infty}U(\bT,m-m_1,\epsilon)U(\bT,m_1,\epsilon)dm_1\\
=k_1k_2\int_{L_{\gamma_1}}\int_{L_{\gamma_2}}\left( u_1^{k_1}u_2^{k_2}\int_{-\infty}^{\infty}\int_{0}^{u_1^{k_1}}\int_0^{u_2^{k_2}}\omega_{\bk}^{\bd}((u_1^{k_1}-s_1)^{\frac{1}{k_1}},(u_2^{k_2}-s_2)^{\frac{1}{k_2}},m-m_1,\epsilon)\right.\\
\left.\omega_{\bk}^{\bd}(s_1^{\frac{1}{k_1}},s_2^{\frac{1}{k_2}},m_1,\epsilon)\frac{1}{(u_1^{k_1}-s_1)s_1}\frac{1}{(u_2^{k_2}-s_2)s_2}ds_1ds_2\right)e^{-\left(\frac{u_1}{T_1}\right)^{k_1}-\left(\frac{u_2}{T_2}\right)^{k_2}}\frac{du_2}{u_2}\frac{du_1}{u_1}
\end{multline*}

\end{lemma}
\begin{proof}

The first statement is a direct application of the derivation under the integral symbol. The second and third statements are equivalent, so we only give details for the second one.  

In order to give proof for the second statement, we first apply Fubini theorem at the inner and outer integrals. That expression can be rewritten in the next form:
\begin{multline*}
A:=\int_{L_{k_1\gamma_1}}k_1k_2\int_{L_{\gamma_2}}\left(\int_{L_{s_1^{1/k_1},\gamma_1}}\frac{u_1^{k_1-1}}{\Gamma\left(\frac{m_1}{k_1}\right)}(u_1^{k_1}-s_1)^{\frac{m_1}{k_1}-1}e^{-\left(\frac{u_1}{T_1}\right)^{k_1}}\frac{du_1}{u_1}\right)\\
\times \omega_{\bk}^{\bd}(s_1^{1/k_1},u_2,m,\epsilon)e^{-\left(\frac{u_2}{T_2}\right)^{k_2}}\frac{du_2}{u_2}\frac{ds_1}{s_1},
\end{multline*}
where $L_{k_1\gamma_1}:=\{re^{ik_1\gamma_1}:r\ge0\}$ and $L_{s_1^{1/k_1},\gamma_1}=\{re^{i\gamma_1}:r\ge |s_1|^{1/k_1}\}$. We proceed by applying two consecutive deformation paths at the inner integral in the previous expression: first, we apply $h_1=u_1^{k_1}$, and then $h_1-s_1=h_{11}$. We arrive at
\begin{multline*}
A=k_1k_2\int_{L_{k_1\gamma_1}}\frac{1}{\Gamma\left(\frac{m_1}{k_1}\right)}h_{11}^{\frac{m_1}{k_1}-1}e^{-\frac{h_{11}}{T_1^{k_1}}}dh_{11}\\
\times \int_{L_{k_1\gamma_1}}\int_{L_{\gamma_2}} \omega_{\bk}^{\bd}(s_1^{1/k_1},u_2,m,\epsilon)e^{-\frac{s_1}{T_1^{k_1}}-\left(\frac{u_2}{T_2}\right)^{k_2}}\frac{du_2}{u_2}\frac{ds_1}{s_1}
\end{multline*}
The deformation path $\tilde{u}_1=s_1^{1/k_1}$ followed by $h_{12}=\frac{h_{11}}{T_1^{k_1}}$ yields
$$A=U(\bT,m,\epsilon)\frac{1}{\Gamma\left(\frac{m_1}{k_1}\right)}T_1^{m_1}\int_{L_{k_1\gamma_1}-k_1\hbox{arg}(T_1)}h_{12}^{\frac{m_1}{k_1}-1}e^{-h_{12}}dh_{12}.
$$
A deformation path and the definition of Gamma function allow us to conclude that $A=U(\bT,m,\epsilon)$.

The proof of the third formula follows the same lines of arguments, involving Fubini theorem and it is omitted for the sake of brevity.
\end{proof}

\textbf{Remark:} Lemma~\ref{lema257} provides the equivalence of existence of solutions of different equation (\ref{e32c}) and (\ref{e310}), related by Laplace transformation. 

We define the operators
\begin{equation}\label{e558b}
\mathcal{A}_{\delta_{D_1}}\omega_{\bk}(\btau,m,\epsilon)=\sum_{1\le p_1\le \delta_{D_1}-1}\frac{A_{\delta_{D_1},p_1} \tau_1^{k_1}  }{\Gamma(\delta_{D_1}-p_1)}\int_{0}^{\tau_1^{k_1}}(\tau_1^{k_1}-s_1)^{\delta_{D_1}-p_1-1}k_1s_1^{p_1}\omega_{\bk}(s_1^{1/k_1},\tau_2,m,\epsilon)\frac{ds_1}{s_1},
\end{equation}
$$\tilde{\mathcal{A}}_{\tilde{\delta}_{D_j}}\omega_{\bk}(\btau,m,\epsilon)=\sum_{1\le p_j\le \tilde{\delta}_{D_j}-1}\frac{\tilde{A}_{\tilde{\delta}_{D_j},p_j}\tau_2^{k_2}           }{\Gamma(\tilde{\delta}_{D_j}-p_j)}\int_{0}^{\tau_2^{k_2}}(\tau_2^{k_2}-s_2)^{\tilde{\delta}_{D_j}-p_j-1}k_2s_2^{p_j}\omega_{\bk}(\tau_1,s_2^{1/k_2},m,\epsilon)\frac{ds_2}{s_2},$$
for $j=2,3$. Observe that they turn out to be the $m_{k_1}$ (resp. $m_{k_2}$) Borel transform of the operator $A_{D_1}(T_1,\partial_{T_1})$ (resp. $\tilde{A}_{D_j}(T_2,\partial_{T_2})$, for $j=2,3$) (see Section~\ref{secborellaplace} for more details on this).

In view of the assumption,described in (\ref{e331}), we define the natural numbers $d_{l_1,k_1}$ and $d_{l_2,k_2}$ by
\begin{equation}\label{e332}
d_{l_1}=\delta_{l_1}(k_1+1)+d_{l_1,k_1},\quad \tilde{d}_{l_2}=(\tilde{\delta}_{l_2}-1)(k_2+1)+d_{l_2,k_2},
\end{equation}
for all $0\le l_1\le D_1-1$ and $0\le l_2\le D_2-1$.

 Taking into account Lemma~\ref{lema257}, we see that $U(\bT,m,\epsilon)$ satisfies (\ref{e32c}), iff $\omega_{\bk}^{\bd}(\btau,m,\epsilon)$ is a solution of the next equation. 

\begin{multline}
\left(Q(im)+ R_{D_1,D_2}(im)(k_1\tau_1^{k_1})^{\delta_{D_1}}(k_2\tau_2^{k_2})^{\tilde{\delta}_{D_2}-1}+R_{D_3}(im)(k_2\tau_2^{k_2})^{\tilde{\delta}_{D_3}-1}\right)\omega(\btau,m,\epsilon)\\
=-(k_1\tau_1^{k_1})^{\delta_{D_1}}\frac{\tilde{\mathcal{A}}_{D_2}(\tau_2)}{k_2\tau_2^{k_2}}R_{D_1,D_2}(im)\omega(\btau,m,\epsilon)- (k_2\tau_2^{k_2})^{\tilde{\delta}_{D_2}-1}\mathcal{A}_{D_1}(\tau_1)R_{D_1,D_2}(im)\omega(\btau,m,\epsilon)\\
-\mathcal{A}_{D_1}(\tau_1)\frac{\tilde{\mathcal{A}}_{D_2}(\tau_2)}{k_2\tau_2^{k_2}}R_{D_1,D_2}(im)\omega(\btau,m,\epsilon)-\frac{\tilde{\mathcal{A}}_{D_3}(\tau_2)}{k_2\tau_2^{k_2}}R_{D_3}(im)\omega(\btau,m,\epsilon)\\
+\frac{\epsilon^{-1}}{(2\pi)^{\frac{1}{2}}}\frac{\tau_1^{k_1}}{k_2\Gamma\left(1+\frac{1}{k_2}\right)}
\int_0^{\tau_2^{k_2}}(\tau_2^{k_2}-s_2)^{\frac{1}{k_2}}\int_{-\infty}^{\infty}\int_0^{\tau_1^{k_1}}\int_0^{s_2}P_1(i(m-m_1),\epsilon)\\
\times \omega((\tau_1^{k_1}-s_1)^{\frac{1}{k_1}},(s_2-x_2)^{\frac{1}{k_2}},m-m_1,\epsilon)P_2(im_1,\epsilon)\omega(s_1^{\frac{1}{k_1}},x_2^{\frac{1}{k_2}},m_1,\epsilon)\frac{dx_2ds_1dm_1ds_2}{(\tau_1^{k_1}-s_1)s_1(s_2-x_2)x_2}\\
+\sum_{\stackrel{0\le l_j\le D_j-1}{j=1,2}} \epsilon^{\Delta_{l_1,l_2}-d_{l_1}-\tilde{d}_{l_2}+\delta_{l_1}+\tilde{\delta}_{l_2}-1}R_{l_1,l_2}(im)\frac{\tau_1^{k_1}}{k_2\Gamma\left(\frac{d_{l_1,k_1}}{k_1}\right)\Gamma\left(\frac{d_{l_2,k_2}}{k_2}\right)}\\
\times
\int_0^{\tau_2^{k_2}}\int_0^{\tau_1^{k_1}}(\tau_1^{k_1}-s_1)^{\frac{d_{l_1,k_1}}{k_1}-1}(\tau_2^{k_2}-s_2)^{\frac{d_{l_2,k_2}}{k_2}-1}k_1^{\delta_{l_1}}s_1^{\delta_{l_1}}k_2^{\tilde{\delta}_{l_2}}s_2^{\tilde{\delta}_{l_2}}\omega(s_1^{\frac{1}{k_1}},s_2^{\frac{1}{k_2}},m,\epsilon)\frac{ds_1}{s_1}\frac{ds_2}{s_2}\\
+\frac{\epsilon^{-1}}{k_2\Gamma\left(1+\frac{1}{k_2}\right)}\int_{0}^{\tau_2^{k_2}}(\tau_2^{k_2}-s_2)^{\frac{1}{k_2}}\psi_{\bk}(\tau_1,s_2^{\frac{1}{k_2}},m,\epsilon)\frac{ds_2}{s_2},\label{e310}
\end{multline}
where $\psi_{\bk}$ is the formal $m_{k_1}$-Borel transform with respect to $T_1$ and the formal $m_{k_2}$-Borel transform with respect to $T_2$ of $F(\bT,m,\epsilon)$, i.e.
$$\psi_{\bk}(\btau,m,\epsilon) = \sum_{n_1,n_2 \geq 1} F_{n_1,n_2}(m,\epsilon) \frac{\tau_1^{n_1}}{\Gamma(\frac{n_1}{k_1})}\frac{\tau_2^{n_2}}{\Gamma(\frac{n_2}{k_2})}.$$
Observe that $\psi_{\bk}$ is an entire function with respect to $\btau$. Moreover, regarding the construction of $\psi_{\bk}$ and (\ref{e165}), one has 
\begin{multline*}
\left\|\psi_{\bk}(\btau,m,\epsilon)\right\|_{(\bnu,\beta,\mu,\bk,\epsilon)} \leq \sum_{n_1,n_2 \geq 1}
\left\|F_{n_1,n_2}(m,\epsilon)\right\|_{(\beta,\mu)}\\
\times (\sup_{\btau \in (\bar{D}(0,\rho) \cup S_{d_1})\times S_{d_2})}
\frac{1 + |\frac{\tau_1}{\epsilon}|^{2k_1}}{|\frac{\tau_1}{\epsilon}|}\frac{1 + |\frac{\tau_2}{\epsilon}|^{2k_2}}{|\frac{\tau_2}{\epsilon}|} \exp(-\nu_1 |\frac{\tau_1}{\epsilon}|^{k_1}-\nu_2 |\frac{\tau_2}{\epsilon}|^{k_2})
\frac{|\tau_1|^{n_1}|\tau_2|^{n_2}}{\Gamma(\frac{n_1}{k_1})\Gamma(\frac{n_2}{k_2})}) 
\end{multline*}
for all $\epsilon \in D(0,\epsilon_{0}) \setminus \{ 0 \}$, any unbounded sectors $S_{d_1}$ and $S_{d_2}$ centered at 0 and bisecting directions $d_1 \in \mathbb{R}$ and $d_2\in\mathbb{R}$, respectively, for some $\bnu=(\nu_1,\nu_2)\in(0,+\infty)^2$. 

\textbf{Remark:} According to classical estimates and Stirling formula, we observe that $\psi_{\bk}(\btau,m,\epsilon) \in F_{(\bnu,\beta,\mu,\bk,\epsilon)}^{\bd}$. See Definition~\ref{def2}.

We write
\begin{equation}\label{e326}P_m(\btau)=Q(im)+ R_{D_1,D_2}(im)(k_1\tau_1^{k_1})^{\delta_{D_1}}(k_2\tau_2^{k_2})^{\tilde{\delta}_{D_2}-1}+R_{D_3}(im)(k_2\tau_2^{k_2})^{\tilde{\delta}_{D_3}-1}.
\end{equation}

\section{Construction of the solution for a convolution equation}\label{seccons}

The main aim in this section is to provide with a solution of (\ref{e310}) which belongs to certain Banach space of functions satisfying bounds in the form (\ref{e209}). Such function is obtained as a fixed point of an operator acting on Banach spaces, introduced and studied in the incoming section.

\subsection{Banach spaces of exponencial growth}\label{subsecespbanach}

We consider the open disc $D(0,\rho)$ for some $\rho>0$. Let $S_{d_j}$ be open unbounded sectors with bisecting directions $d_j \in \mathbb{R}$, for $j=1,2$, and let $\mathcal{E}$ be an open sector with finite radius $r_{\mathcal{E}}$, all with vertex at $0$ in $\mathbb{C}$.

The following norm is inspired from that considered by the authors in~\cite{family1}. It is an adecquate modification of that described in~\cite{lama}, adapted to the framework of two complex time variables.

\begin{defin}\label{def2} Let $\nu_1,\nu_2,\beta,\mu>0$ and $\rho>0$ be positive real numbers. Let $k_1,k_2 \geq 1$ be integer numbers and let $\epsilon \in \mathcal{E}$. We put $\bnu=(\nu_1,\nu_2)$, $\bk=(k_1,k_2)$, $\bd=(d_1,d_2)$, and denote
$F_{(\bnu,\beta,\mu,\bk,\epsilon)}^{\bd}$ the vector space of continuous functions $(\btau,m) \mapsto h(\btau,m)$ on the set
$(\bar{D}(0,\rho) \cup S_{d_1})\times  S_{d_2} \times \mathbb{R}$, which are holomorphic with respect to $\btau$ on $(D(0,\rho) \cup S_{d_1})\times  S_{d_2} $ and such that
\begin{multline}
||h(\btau,m)||_{(\bnu,\beta,\mu,\bk,\epsilon)}\\
=
\sup_{\stackbin{\btau \in (\bar{D}(0,\rho) \cup S_{d_1})\times  S_{d_2}}{m \in \mathbb{R}}} (1+|m|)^{\mu}
\frac{1 + |\frac{\tau_1}{\epsilon}|^{2k_1}}{|\frac{\tau_1}{\epsilon}|}\frac{1 + |\frac{\tau_2}{\epsilon}|^{2k_2}}{|\frac{\tau_2}{\epsilon}|}\exp( \beta|m| - \nu_1|\frac{\tau_1}{\epsilon}|^{k_1}-\nu_2|\frac{\tau_2}{\epsilon}|^{k_2} ) |h(\tau,m)|
\end{multline}
is finite. The normed space
$(F_{(\bnu,\beta,\mu,\bk,\epsilon)}^{\bd},||.||_{(\bnu,\beta,\mu,\bk,\epsilon)})$ is a Banach space.
\end{defin}

We fix $\epsilon \in \mathcal{E}$, $\mu,\beta,>0$ in the whole subsection. We also choose $\bnu=(\nu_1,\nu_2)\in (0,\infty)^2$, $\bd=(d_1,d_2)\in\R^2$, and $\bk=(k_1,k_2)\in\mathbb{N}^2$.

We first state some technical results. The first one follows directly from the definition of the norm of the Banach space.

\begin{lemma}\label{lema1} Let $(\btau,m) \mapsto a(\btau,m)$ be a bounded continuous function on
$(\bar{D}(0,\rho) \cup S_{d_1})\times S_{d_2} \times \mathbb{R}$, holomorphic with respect to $\btau$ on $(D(0,\rho) \cup S_{d_1})\times  S_{d_2}$. Then,
$$
|| a(\btau,m) h(\btau,m) ||_{(\bnu,\beta,\mu,\bk,\epsilon)} \leq
\left( \sup_{\btau \in (\bar{D}(0,\rho) \cup S_{d_1})\times  S_{d_2},m \in \mathbb{R}} |a(\btau,m)| \right)
||h(\btau,m)||_{(\bnu,\beta,\mu,\bk,\epsilon)}
$$
for all $h(\btau,m) \in F_{(\bnu,\beta,\mu,\bk,\epsilon)}^{\bd}$.
\end{lemma}

\begin{lemma}\label{lema2}
Let $\boldsymbol{\sigma}=(\sigma_1,\sigma_2)\in(0,\infty)^2$, and assume that $a_{\boldsymbol{\sigma},\bk}$ is a holomorphic function of $(D(0,\rho)\cup S_{d_1})\times S_{d_2}$, continuous up to $(\overline{D}(0,\rho)\cup S_{d_1})\times S_{d_2}$, such that
$$|a_{\boldsymbol{\sigma},\bk}(\btau)|\le \frac{1}{(1+|\tau_1|^{k_1})^{\sigma_1}(1+|\tau_2|^{k_2})^{\sigma_2}},$$
for every $\btau\in (\overline{D}(0,\rho)\cup S_{d_1})\times  S_{d_2}$. We take $0\le \tilde{\sigma}_j\le\sigma_j$ for $j=1,2$. Assume that one of the following hold:
\begin{itemize}
\item $\sigma_3\ge0$ and $\sigma_3+\sigma_4\le \sigma_2-\tilde{\sigma}_2$,
\item $\sigma_3=\frac{\xi}{k_2}-1$ and $\sigma_3+\frac{1}{k_2}\le \sigma_2-\tilde{\sigma}_2$,
\end{itemize}
where $\xi>1$. Then, there exists $C_1>0$, depending on $\bk,\nu_2,\tilde{\sigma}_j,\sigma_\ell,$ $j=1,2$, $\ell=1,\ldots,4$, such that
\begin{multline*}
\left\|a_{\boldsymbol{\sigma},\bk}(\btau)\tau_1^{\tilde{\sigma}_1k_1}\tau_2^{\tilde{\sigma}_2k_2}\int_0^{\tau_2^{k_2}}(\tau_2^{k_2}-s_2)^{\sigma_3}s_2^{\sigma_4}f(\tau_1,s_2^{\frac{1}{k_2}},m)ds_2\right\|_{(\bnu,\beta,\mu,\bk,\epsilon)}\\
\le C_1|\epsilon|^{k_2(1+\sigma_3+\sigma_4-\sigma_2+\tilde{\sigma}_2)}\left\|f(\btau,m)\right\|_{(\bnu,\beta,\mu,\bk,\epsilon)},
\end{multline*}
for every $f\in F^{\bd}_{(\bnu,\beta,\mu,\bk,\epsilon)}$.
\end{lemma}
\begin{proof}
There exists $C_{1.1}>0$ only depending on $\sigma_1,\sigma_2,k_1,k_2$ such that
$$|a_{\boldsymbol{\sigma},\bk}(\btau)\tau_1^{\tilde{\sigma}_1k_1}\tau_2^{\tilde{\sigma}_2k_2}|\le \frac{C_{1.1}}{(1+|\tau_2|^{k_2})^{\sigma_2-\tilde{\sigma}_2}},$$ 
for every $\btau\in (\overline{D}(0,\rho)\cup S_{d_1})\times S_{d_2}$. We apply the definition of the norm of $F_{(\bnu,\beta,\mu,\bk,\epsilon)}^{\bd}$ to arrive at 
\begin{multline*}
\left\|a_{\boldsymbol{\sigma},\bk}(\btau)\tau_1^{\tilde{\sigma}_1k_1}\tau_2^{\tilde{\sigma}_2k_2}\int_0^{\tau_2^{k_2}}(\tau_2^{k_2}-s_2)^{\sigma_3}s_2^{\sigma_4}f(\tau_1,s_2^{\frac{1}{k_2}},m)ds_2\right\|_{(\bnu,\beta,\mu,\bk,\epsilon)}\\
\le C_{1.1}\left\|f(\btau,m)\right\|_{(\bnu,\beta,\mu,\bk,\epsilon)}\sup_{\tau_2\in S_{d_2}}\frac{1+\left|\frac{\tau_2}{\epsilon}\right|^{2k_2}}{\left|\frac{\tau_2}{\epsilon}\right|}\exp\left(-\nu_2\left|\frac{\tau_2}{\epsilon}\right|^{k_2}\right)\frac{1}{(1+|\tau_2|^{k_2})^{\sigma_2-\tilde{\sigma}_2}}\\
\times \int_0^{|\tau_2|^{k_2}}(|\tau_2|^{k_2}-h)^{\sigma_3}h^{\sigma_4}\frac{\frac{h^{\frac{1}{k_2}}}{|\epsilon|}}{1+\frac{h^2}{|\epsilon|^{2k_2}}}\exp\left(\nu_2\frac{h}{|\epsilon|^{k_2}}\right)dh.
\end{multline*}
The proof concludes with the steps providing a bound for $C_2(\epsilon)$ in the proof of Proposition 2 in~\cite{lama}.
\end{proof}

An analogous result holds by interchanging the role of the time variables. 

\begin{lemma}\label{lema22}
Under the same hypotheses as in Lemma~\ref{lema2}, assume that 
\begin{itemize}
\item $\sigma_3\ge0$ and $\sigma_3+\sigma_4\le \sigma_1-\tilde{\sigma}_1$,
\item $\sigma_3=\frac{\xi}{k_1}-1$ and $\sigma_3+\frac{1}{k_1}\le \sigma_1-\tilde{\sigma}_1$,
\end{itemize}
where $\xi>1$. Then, there exists $C_1>0$, depending on $\bk,\nu_1,\tilde{\sigma}_j,\sigma_\ell,$ for $j=1,2$ and $\ell=1\ldots,4$, such that
\begin{multline*}
\left\|a_{\boldsymbol{\sigma},\bk}(\btau)\tau_1^{\tilde{\sigma}_1k_1}\tau_2^{\tilde{\sigma}_2k_2}\int_0^{\tau_1^{k_1}}(\tau_1^{k_1}-s_1)^{\sigma_3}s_1^{\sigma_4}f(s_1^{\frac{1}{k_1}},\tau_2,m)ds_1\right\|_{(\bnu,\beta,\mu,\bk,\epsilon)}\\
\le C_1|\epsilon|^{k_1(1+\sigma_3+\sigma_4-\sigma_1+\tilde{\sigma}_1)}\left\|f(\btau,m)\right\|_{(\bnu,\beta,\mu,\bk,\epsilon)},
\end{multline*}
for every $f\in F^{\bd}_{(\bnu,\beta,\mu,\bk,\epsilon)}$.
\end{lemma}

Grouping the integral operators in Lemma~\ref{lema2} and Lemma~\ref{lema22} the following result is attained.

\begin{lemma}\label{lema23}
Let $\boldsymbol{\sigma}\in (0,\infty)^{2}$. Assume that $a_{\boldsymbol{\sigma},\bk}$ is given as in Lemma~\ref{lema2}. Let $0\le\tilde{\sigma}_j\le\sigma_j$ for $j=1,2$, and $\sigma_{31},\sigma_{32},\sigma_{41},\sigma_{42}$ be real numbers such that
\begin{itemize}
\item $\sigma_{3j}\ge0$ and $\sigma_{3j}+\sigma_{4j}\le \sigma_j-\tilde{\sigma}_j$,
\item $\sigma_{3j}=\frac{\xi_j}{k_j}-1$ and $\sigma_{3j}+\frac{1}{k_j}\le \sigma_j-\tilde{\sigma}_j$,
\end{itemize}
for $j=1,2$ and where $\xi_j>1$. Then, there exists $C_1>0$ depending on $\bk,\bnu,\sigma_j,\tilde{\sigma}_j,\sigma_{3j},\sigma_{4j}$ for $j=1,2$, such that
\begin{multline*}
\left\|a_{\boldsymbol{\sigma},\bk}(\btau)\tau_1^{\tilde{\sigma}_1k_1}\tau_2^{\tilde{\sigma}_2k_2}\int_0^{\tau_1^{k_1}}\int_0^{\tau_2^{k_2}}(\tau_1^{k_1}-s_1)^{\sigma_{31}}s_1^{\sigma_{41}}(\tau_2^{k_2}-s_2)^{\sigma_{32}}s_2^{\sigma_{42}}f(s_1^{\frac{1}{k_1}},s_2^{\frac{1}{k_2}},m)ds_2ds_1\right\|_{(\bnu,\beta,\mu,\bk,\epsilon)}\\
\le C_1^2|\epsilon|^{k_1(1+\sigma_{31}+\sigma_{41}-\sigma_1+\tilde{\sigma}_1)+k_2(1+\sigma_{32}+\sigma_{42}-\sigma_2+\tilde{\sigma}_2)}\left\|f(\btau,m)\right\|_{(\bnu,\beta,\mu,\bk,\epsilon)},
\end{multline*}
for every $f\in F^{\bd}_{(\bnu,\beta,\mu,\bk,\epsilon)}$.
\end{lemma}

The proof of Proposition 1 in~\cite{lama} can be adapted with minor modifications to the Banach spaces under study.

\begin{lemma}\label{lemaaux}
Let $\gamma_2>0$. Assume that $1/k_2\le \gamma_2\le 1$. Then, there exists $C_2>0$ (depending on $\bnu,\bk,\gamma_2$) such that
$$\left\|\int_{0}^{\tau_2^{k_2}}(\tau_2^{k_2}-s_2)^{\gamma_2}f(\tau_1,s_2^{\frac{1}{k_2}},m)\frac{ds_2}{s_2}\right\|_{(\bnu,\beta,\mu,\bk,\epsilon)}\le C_2|\epsilon|^{k_2\gamma_2}\left\|f(\btau,m)\right\|_{(\bnu,\beta,\mu,\bk,\epsilon)},$$
for every $f(\btau,m)\in F^{\bd}_{(\bnu,\beta,\mu,\bk,\epsilon)}$. 
\end{lemma}

The symmetric statement of Lemma~\ref{lemaaux}, obtained by interchanging the role of $\tau_2$ and $\tau_1$ is derived straightforward from Lemma~\ref{lemaaux}. We finally state the following auxiliary lemma.

\begin{lemma}\label{lema6}
Let $\bsigma$ and $a_{\bsigma,\bk}$ be as in Lemma~\ref{lema2}. Assume that $P_1,P_2,R\in\C[X]$ such that
$$\hbox{deg}(R)\ge \hbox{deg}(P_1),\quad \hbox{deg}(R)\ge \hbox{deg}(P_2),\quad R(im)\neq 0$$
for every $m\in\R$. Assume that $\mu>\max\{\hbox{deg}(P_1)+1,\hbox{deg}(P_2)+1\}$. We take $\tilde{\sigma}_{j}\le \sigma_j$ for $j=1,2$. Then, there exists a constant $C_3>0$ (depending on $Q_1,Q_2,R,\mu,\bk,\bnu$) such that 
\begin{multline*}
\left\|\frac{a_{\bsigma,\bk}(\btau)}{R(im)}\tau_1^{\tilde{\sigma}_1k_1}\tau_2^{\tilde{\sigma}_2k_2}\int_{0}^{\tau_2^{k_2}}(\tau_2^{k_2}-s_2)^{\frac{1}{k_2}}\int_{-\infty}^{\infty}\int_0^{\tau_1^{k_1}}\int_0^{s_2}P_1(i(m-m_1))\right.\\
\left.f((\tau_1^{k_1}-s_1)^{\frac{1}{k_1}},(s_2^{k_1}-x_2)^{\frac{1}{k_2}},m-m_1)P_2(im_1) g(s_1^{\frac{1}{k_1}},x_2^{\frac{1}{k_2}},m_1)\frac{dx_2ds_1dm_1ds_2}{(\tau_1^{k_1}-s_1)s_1(s_2-x_2)x_2}\right\|\\
\le C_3|\epsilon|\left\|f(\btau,m)\right\|_{(\bnu,\beta,\mu,\bk,\epsilon)}\left\|g(\btau,m)\right\|_{(\bnu,\beta,\mu,\bk,\epsilon)},
\end{multline*}
for every $f(\btau,m),g(\btau,m)\in F^{d}_{(\bnu,\beta,\mu,k,\epsilon)}$.
\end{lemma}
\begin{proof}
We follow analogous estimates as in the proof of Proposition 3 in~\cite{lama} to arrive at 
\begin{multline*}
\left\|\frac{a_{\bsigma,\bk}(\btau)}{R(im)}\tau_1^{\tilde{\sigma}_1k_1}\tau_2^{\tilde{\sigma}_2k_2}\int_{0}^{\tau_2^{k_2}}(\tau_2^{k_2}-s_2)^{\frac{1}{k_2}}\int_{-\infty}^{\infty}\int_0^{\tau_1^{k_1}}\int_0^{s_2}P_1(i(m-m_1))\right.\\
\left.\times f((\tau_1^{k_1}-s_1)^{\frac{1}{k_1}},(s_2^{k_1}-x_2)^{\frac{1}{k_2}},m-m_1)P_2(im_1)g(s_1^{\frac{1}{k_1}},x_2^{\frac{1}{k_2}},m_1)\frac{dx_2ds_1dm_1ds_2}{(\tau_1^{k_1}-s_1)s_1(s_2-x_2)x_2}\right\|\\
\le \left(\sup_{\tau_1\in (D(0,\rho_1)\cup S_{d_1})}\frac{|\tau_1|^{\tilde{\sigma}_1k_1}}{(1+|\tau_1|^{k_1})^{\sigma_1}}\frac{1+\left|\frac{\tau_1}{\epsilon}\right|^{2k_1}}{\left|\frac{\tau_1}{\epsilon}\right|}\int_0^{|\tau_1|^{k_1}}\frac{\frac{(|\tau_1|^{k_1}-h_1)^{\frac{1}{k_1}}}{|\epsilon|}}{1+\frac{(|\tau_1|^{k_1}-h_1)^2}{|\epsilon|^{2k_1}}}\frac{\frac{h_1^{\frac{1}{k_1}}}{|\epsilon|}}{1+\frac{h_1^2}{|\epsilon|^{2k_1}}}\frac{dh_1}{(|\tau_1|^{k_1}-h_1) h_1}\right)\\
\times \left(\sup_{\tau_2\in S_{d_2}}\frac{|\tau_2|^{\tilde{\sigma}_2k_2}}{(1+|\tau_2|^{k_2})^{\sigma_2}}\frac{1+\left|\frac{\tau_2}{\epsilon}\right|^{2k_2}}{\left|\frac{\tau_2}{\epsilon}\right|}\exp\left(-\nu_2\left|\frac{\tau_2}{\epsilon}\right|^{k_2}\right)\int_0^{|\tau_2|^{k_2}}(|\tau_2|^{k_2}-h_2)^{\frac{1}{k_2}}\right.\\
\left.\int_0^{|s_2|}\frac{\frac{(h_2-x_2)^{\frac{1}{k_2}}}{|\epsilon|}}{1+\frac{(h_2-x_2)^2}{|\epsilon|^{2k_2}}}\exp\left(\nu_2\frac{h_2}{|\epsilon|^{k_2}}\right)\frac{\frac{x_2^{\frac{1}{k_2}}}{|\epsilon|}}{1+\frac{x_2^2}{|\epsilon|^{2k_2}}}\frac{dx_2dh_2}{(h_2-x_2) x_2}\right)\left\|f(\btau,m)\right\|_{(\bnu,\beta,\mu,\bk,\epsilon)}\left\|g(\btau,m)\right\|_{(\bnu,\beta,\mu,\bk,\epsilon)}.
\end{multline*}
On the one hand, the expression $|\tau_2|^{\tilde{\sigma}_2k_2}/(1+|\tau_2|^{k_2})^{\sigma_2}$ is bounded. Moreover, 
\begin{multline*}
A_2:=\sup_{\tau_2\in S_{d_2}}\frac{1+\left|\frac{\tau_2}{\epsilon}\right|^{2k_2}}{\left|\frac{\tau_2}{\epsilon}\right|}\exp\left(-\nu_2\left|\frac{\tau_2}{\epsilon}\right|^{k_2}\right)\int_0^{|\tau_2|^{k_2}}(|\tau_2|^{k_2}-h_2)^{\frac{1}{k_2}}\\
\int_0^{|s_2|}\frac{\frac{(h_2-x_2)^{\frac{1}{k_2}}}{|\epsilon|}}{1+\frac{(h_2-x_2)^2}{|\epsilon|^{2k_2}}}\exp\left(\nu_2\frac{h_2}{|\epsilon|^{k_2}}\right)\frac{\frac{x_2^{\frac{1}{k_2}}}{|\epsilon|}}{1+\frac{x_2^2}{|\epsilon|^{2k_2}}}\frac{dx_2dh_2}{(h_2-x_2) x_2}
\end{multline*}
can be estimated following the same steps as in the study of upper bounds for $C_{3.2}$ in formula (35) of~\cite{lama}. We get the existence of $C_{2.1}>0$ such that $A_2\le C_{3.2}|\epsilon|$. It only rests to prove that $A_1$ is upper bounded, where
\begin{align*}
A_1:=&\sup_{\tau_1\in (D(0,\rho)\cup S_{d_1})}\frac{|\tau_1|^{\tilde{\sigma}_1k_1}}{(1+|\tau_1|^{k_1})^{\sigma_1}}\frac{1+\left|\frac{\tau_1}{\epsilon}\right|^{2k_1}}{\left|\frac{\tau_1}{\epsilon}\right|}\int_0^{|\tau_1|^{k_1}}\frac{\frac{(|\tau_1|^{k_1}-h_1)^{\frac{1}{k_1}}}{|\epsilon|}}{1+\frac{(|\tau_1|^{k_1}-h_1)^2}{|\epsilon|^{2k_1}}}\frac{\frac{h_1^{\frac{1}{k_1}}}{|\epsilon|}}{1+\frac{h_1^2}{|\epsilon|^{2k_1}}}\frac{dh_1}{(|\tau_1|^{k_1}-h_1) h_1}\\
=&\sup_{\tau_1\in (D(0,\rho)\cup S_{d_1})}\tilde{A}_1.
\end{align*}
We distinguish two cases. First, we assume that $|\tau_1|\ge C$, for some $C>0$. Then, it holds that 
$$\frac{|\tau_1|^{\tilde{\sigma}_1k_1}}{(1+|\tau_1|^{k_1})^{\sigma_1}}$$ is upper bounded, and by putting $x=|\tau_1/\epsilon|$ one can estimate $\tilde{A}_1$ from above by 
$$\sup_{x\ge\tilde{C}}\frac{1+x^2}{x^{1/k_1}}\int_0^{\infty}\frac{dh}{(1+(x-h)^2)(1+h^2)}.$$
for some $\tilde{C}>0$. We apply Corollary 4.9 in~\cite{cota} to conclude that 
$$\tilde{A}_1\le \sup_{x\ge\tilde{C}}\frac{1+x^2}{x^{1/k_1}}\frac{j_1}{x^2+1},$$
for some $j_1>0$. The previous expression is upper bounded by a positive constant. Second, in the case that $|\tau_1|<C$, we have $(1+|\tau_1|^{k_1})^{\sigma_1}\ge 1$. We put $x=(|\tau_1|/|\epsilon|)^{k_1}$ to get that 
$$\sup_{\tau_1\in (D(0,\rho)\cup S_{d_1}),|\tau_1|\le C}\tilde{A}_1\le \sup_{x\ge0}x\frac{1+x^2}{x^{1/k_1}}\int_0^x\frac{(x-h_1)^{\frac{1}{k_1}}}{1+(x-h_1)^2}\frac{h_1^{\frac{1}{k_1}}}{1+h_1^{2}}\frac{dh_1}{h_1(x-h_1)}.$$
A partial fraction decomposition yields 
$$\int_0^x\frac{(x-h_1)^{\frac{1}{k_1}}}{1+(x-h_1)^2}\frac{h_1^{\frac{1}{k_1}}}{1+h_1^{2}}\frac{dh_1}{h_1(x-h_1)}\le \frac{j_k}{x^{1-\frac{2}{k_1}}(x^2+4)},\qquad x\ge0,$$
for some $j_k>0$, valid for $k\ge 2$. This concludes the existence of a positive upper bound for $A_1$, and the proof follows from this point.
\end{proof}

\subsubsection{Domain of existence for the solution}\label{secnodef}

The purpose of this section is twofold. On the one hand, we motivate the fact that any actual holomorphic solution $\omega(\btau,m,\epsilon)$ of (\ref{e310}) is not well defined on sets of the form $S_{d_1}\times (S_{d_2}\cup D(0,\rho_2))$, for $d_1,d_2\in\R$ and any choice of $\rho_2>0$.This is due to a small divisor phenomenon observed, which does not allow to proceed with a summability procedure. On the second hand, we aim to display geometric conditions on the natural domains in which the solution is defined.

In order to motivate that the natural domains of definition of a solution cannot be of the form $S_{d_1}\times (S_{d_2}\cup D(0,\rho_2))$, for $d_1,d_2\in\R$ and $\rho_2>0$, let $\rho_2>0$. We rewrite the equation $P_{m}(\btau)=0$ (see (\ref{e326}) for the definition of $P_{m}$) in the form
\begin{equation}\label{e494}
\tau_2^{k_2(\tilde{\delta}_2-1)}=\frac{-Q(im)}{R_{D_1,D_2}(im)k_1^{\delta_{D_1}}k_2^{\tilde{\delta}_{D_2-1}}\tau_1^{k_1\delta_{D_1}}+R_{D_3}(im)k_2^{\tilde{\delta}_{D_3}-1}\tau_2^{k_2(\tilde{\delta}_{D_3}-\tilde{\delta}_{D_2})}}.
\end{equation} 
We put $T_2=\tau_2^{k_2(\tilde{\delta}_{D_2}-1)}$, and write (\ref{e494}) in the form $\Psi(T_2)=T_2$, where
\begin{equation}\label{e495}
\Psi(T_2):=-Q(im)\left(R_{D_1,D_2}(im)k_1^{\delta_{D_1}}k_2^{\tilde{\delta}_{D_2-1}}\tau_1^{k_1\delta_{D_1}}+R_{D_3}(im)k_2^{\tilde{\delta}_{D_3}-1}T_2^{\frac{\tilde{\delta}_{D_3}-\tilde{\delta}_{D_2}}{\tilde{\delta}_{D_2}-1}}\right)^{-1}.
\end{equation}

\begin{lemma}\label{lema502}
Let $d_1,d_2\in\R$. Under the assumption that $\frac{\tilde{\delta}_{D_3}-\tilde{\delta}_{D_2}}{\tilde{\delta}_{D_2}-1}\in\N\setminus\{0\}$, there exists $\tau_1\in S_{d_1}$ such that the following statements hold:
\begin{enumerate}
\item $\Psi$ is a map from $E:=\overline{D}(0,\left(\frac{\rho_2}{2}\right)^{k_2(\tilde{\delta}_{D_{2}}-1)})$ into itself.
\item $\Psi:E\to E$ is a shrinking map.
\end{enumerate}
\end{lemma}
\begin{proof}
Let $\tau_1\in S_{d_1}$ with large enough modulus in such a way that
\begin{equation}\label{e511}
\left|R_{D_1,D_2}(im)k_1^{\delta_{D_1}}k_2^{\tilde{\delta}_{D_2-1}}+R_{D_3}(im)k_2^{\tilde{\delta}_{D_3}-1}\frac{T_2^{\frac{\tilde{\delta}_{D_3}-\tilde{\delta}_{D_2}}{\tilde{\delta}_{D_2}-1}}}{\tau_1^{k_1\delta_{D_1}}}\right|\ge |R_{D_3}(im)|,
\end{equation}
for every $m\in\R$ and all $T_2\in \hat{S}_{d_2}\cup D(0,\hat{\rho}_2)$. Here, $\hat{S}_{d_2}$ stand for the infinite sector defined by $\hat{S}_{d_2}:=\{T_2\in\C^{\star}: \tau_2^{k_2(\tilde{\delta}_{D_2}-1)}\in S_{d_2}\}$, and $\hat{\rho}_2=\rho_2^{k_2(\tilde{\delta}_2-1)}$. The assumption (\ref{e165b}) on the geometry of the problem and (\ref{e511}) yield
\begin{multline*}
|\Psi(T_2)|\le\frac{|Q(im)|}{|\tau_1|^{k_1\delta_{D_1}}}\left|R_{D_1,D_2}(im)k_1^{\delta_{D_1}}k_2^{\tilde{\delta}_{D_2-1}}+R_{D_3}(im)k_2^{\tilde{\delta}_{D_3}-1}\frac{T_2^{\frac{\tilde{\delta}_{D_3}-\tilde{\delta}_{D_2}}{\tilde{\delta}_{D_2}-1}}}{\tau_1^{k_1\delta_{D_1}}}\right|^{-1}\\
\le \frac{|Q(im)|}{|\tau_1|^{k_2\delta_{D_1}}|R_{D_3}(im)|}\le \left(\frac{\rho_2}{2}\right)^{k_2(\tilde{\delta}_{D_2}-1)},
\end{multline*}
for large enough $|\tau_1|$. As a result, we get the fist statement in the result. We have
\begin{multline*}
|\Psi'(z)|\le \frac{|R_{D_3}(im)|k_2^{\tilde{\delta}_{D_3}-1}\frac{\tilde{\delta}_{D_3}-\tilde{\delta}_{D_2}}{\tilde{\delta}_{D_2}-1}\left(\frac{\rho_2}{2}\right)^{k_2(\tilde{\delta}_{D_2}-1)\frac{\tilde{\delta}_{D_3}-\tilde{\delta}_{D_2}}{\tilde{\delta}_{D_2}-1}-1}|Q(im)|}{(|\tau_1|^{k_1\delta_{D_1}}|R_{D_3}(im)|)^2}\\
\le k_2^{\tilde{\delta}_{D_3}-1}\frac{\tilde{\delta}_{D_3}-\tilde{\delta}_{D_2}}{\tilde{\delta}_{D_2}-1}\left(\frac{\rho_2}{2}\right)^{k_2(\tilde{\delta}_{D_3}-2\tilde{\delta}_{D_2}+1)}\left|\frac{Q(im)}{R_{D_3}(im)}\right|\frac{1}{|\tau_1|^{2k_1\delta_{D_1}}}\le\frac{1}{2},
\end{multline*}
for every $z\in D(0,\left(\frac{\rho_2}{2}\right)^{k_2(\tilde{\delta}_{D_2}-1)})$, $m\in\R$, and large enough $|\tau_1|$. We get that
$$|\Psi(T_2)-\Psi(T'_2)|\le\sup_{z\in[T_2,T'_2]}|\Psi'(z)||T_2-T'_2|\le \frac{1}{2}|T_2-T'_2|,$$
for every $T_2,T'_2\in \overline{D}(0,\left(\frac{\rho_2}{2}\right)^{k_2(\tilde{\delta}_{D_2}-1)}$. The application of the mean value theorem entails the second statement of the result.
\end{proof}

As a consequence of Lemma~\ref{lema502}, we deduce that $\Psi$ has a unique fixed point in $E$, hence, there exists a unique solution of $\Psi(T_2)=T_2$ for $T_2\in E$, say $T_0$. The solutions of (\ref{e494}) are the solutions of $\tau_2^{k_2(\tilde{\delta}_{D_2}-1}=T_0$. As a matter of fact, the $k_2(\tilde{\delta}_{D_2}-1)$ roots of $T_0$ belong the disc $D(0,\frac{\rho_2}{2})$.

\textbf{Remark:} Observe that, in the case that $\tilde{\delta}_{D_2}=\tilde{\delta}_{D_3}$, the equation $\Psi(T_2)=T_2$ can be solved directly in terms of $\tau_1$. In this case, the $k_2(\tilde{\delta}_{D_2}-1)$ roots of $P_{m}(\btau)=0$ lay on $D(0,\frac{\rho_2}{2})$, and we can not define $\omega(\btau,m,\epsilon)$ in any set of the form $S_{d_1}\times (S_{d_2}\cup D(0,\rho_2))$.

\vspace{0.3cm} 

In the next paragraphs we display geometric conditions on the problem, which allow us to attain lower estimates on $P_m(\btau)$, defined in (\ref{e326}). On the way, the choice of directions $d_1$ and $d_2$ is made accordingly with the geometry of the problem.

We write

$$\frac{P_m(\btau)}{Q(im)}=1+\tau_2^{k_2(\tilde{\delta}_{D_2}-1)}\left(\frac{R_{D_1,D_2}(im)}{Q(im)}k_1^{\delta_{D_1}}k_2^{\tilde{\delta}_{D_2}-1}\tau_1^{k_1\delta_{D_1}}+\frac{R_{D_3}(im)}{Q(im)}k_2^{\tilde{\delta}_{D_3}-1}\tau_2^{k_2(\tilde{\delta}_{D_3}-\tilde{\delta}_{D_2})}\right).$$

We distinguish different cases.
\begin{enumerate}
\item[1.] In case that $\tau_1\in D(0,\rho_1)$, for some small enough $\rho_1>0$.
\begin{enumerate}
\item[1.1.] If $\tau_2\in D(0,\rho_2)$, for small enough $\rho_2>0$. Regarding (\ref{e165b}), there exist $r^1_{D_1,D_2,Q},r^1_{D_3,Q}>0$ such that 
\end{enumerate}
\end{enumerate}
\begin{multline*}
 \tau_2^{k_2(\tilde{\delta}_{D_2}-1)}\left(\frac{R_{D_1,D_2}(im)}{Q(im)}k_1^{\delta_{D_1}}k_2^{\tilde{\delta}_{D_2}-1}\tau_1^{k_1\delta_{D_1}}+\frac{R_{D_3}(im)}{Q(im)}k_2^{\tilde{\delta}_{D_3}-1}\tau_2^{k_2(\tilde{\delta}_{D_3}-\tilde{\delta}_{D_2})}\right)\\
 \le \rho_2^{k_2(\tilde{\delta}_{D_2}-1)}\left(r^{1}_{D_1,D_2,Q}k_1^{\delta_{D_1}}k_2^{\tilde{\delta}_{D_2}-1}\rho_1^{k_1\delta_{D_1}}+r^1_{D_3,Q}k_2^{\tilde{\delta}_{D_3}-1}\rho_2^{k_2(\tilde{\delta}_{D_3}-\tilde{\delta}_{D_2})}\right)\le \frac{1}{4}
\end{multline*}
for every $m\in\R$, every $\tau_1\in D(0,\rho_1)$, and $\tau_2\in D(0,\rho_2)$. We conclude that
\begin{equation}\label{e479}
\left|\frac{P_{m}(\btau)}{Q(im)}\right|\ge C_1,
\end{equation}
for some positive constant $C_1$, common for all $m\in\R$, $\tau_1\in D(0,\rho_1)$, and $\tau_2\in D(0,\rho_2)$.

\begin{enumerate}
\item[1.2.] Assume that $\tau_2\in S_{d_2}$, with $|\tau_2|\ge \rho_0$, for some fixed $\rho_0>0$. We write
\end{enumerate}

\begin{multline*}
\frac{R_{D_1,D_2}(im)}{Q(im)}k_1^{\delta_{D_1}}k_2^{\tilde{\delta}_{D_2}-1}\tau_1^{k_1\delta_{D_1}}+\frac{R_{D_3}(im)}{Q(im)}k_2^{\tilde{\delta}_{D_3}-1}\tau_2^{k_2(\tilde{\delta}_{D_3}-\tilde{\delta}_{D_2})}\\
=\frac{R_{D_3}(im)}{Q(im)}k_2^{\tilde{\delta}_{D_3}-1}\tau_2^{k_2(\tilde{\delta}_{D_3}-\tilde{\delta}_{D_2})}(1+A(m,\btau)),
\end{multline*}
where
$$A(m,\btau):=\frac{R_{D_1,D_2}(im)}{R_{D_3}(im)}\frac{k_1^{\delta_{D_1}}k_2^{\tilde{\delta}_{D_2}-1}\tau_1^{k_1\delta_{D_1}}}{k_1^{\tilde{\delta}_{D_3}-1}\tau_2^{k_2(\tilde{\delta}_{D_3}-\tilde{\delta}_{D_2})}}.$$

From the assumptions made in (\ref{e165b}), we get that  
$$|A(m,\btau)|\le r_{D_1,D_2,D_3}^1\frac{k_1^{\delta_{D_1}}k_2^{\tilde{\delta}_{D_2}-1}}{k_2^{\tilde{\delta}_{D_3}-1}}\frac{\rho_1^{k_1\delta_{D_1}}}{\rho_0^{k_2(\tilde{\delta}_{D_3}-\tilde{\delta}_{D_2}}},$$
for some $r_{D_1,D_2,D_3}^1>0$. Taking small enough $\rho_1>0$, we can write
$$1+A(m,\btau)=\rho_{m,\btau}e^{i\theta_{m,\btau}},$$
with $\rho_{m,\btau}$ close to 1 and $\theta_{m,\btau}$ close to 0, uniformly for every $m\in\R$ and all $\tau_2\in S_{d_2}$, with $|\tau_2|>\rho_0$ and $\tau_1\in D(0,\rho_1)$.

Therefore, we have
$$\frac{P_m(\btau)}{Q(im)}=1+k_2^{\tilde{\delta}_{D_3}-1}\frac{R_{D_3}(im)}{Q(im)}\tau_2^{k_2(\tilde{\delta}_{D_3}-1)}\rho_{m,\btau}e^{i\theta_{m,\btau}}.$$

Let $\tau_{2,k}$ be the roots satisfying
$$\tau_{2,k}^{k_2(\tilde{\delta}_{D_3}-1)}=\frac{-1}{k_{2}^{\tilde{\delta}_{D_3}-1}\rho_{m,\btau}}\frac{Q(im)}{R_{D_3}(im)}e^{-i\theta_{m,\btau}},$$
for $k=0,\ldots,k_2(\tilde{\delta}_{D_3}-1)-1$. We select the sector $S_{d_2}$ in such a way that if $\tau_2\in S_{d_2}$, then it can be expressed as $\tau_2=\rho e^{i\theta}\tau_{2,k}$ for some fixed $k$, some $\theta$ close to 0, $\theta\neq 0$, and any $\rho>0$. We get
$$\frac{P_m(\btau)}{Q(im)}=1-\rho^{k_2(\tilde{\delta}_{D_3}-1)}e^{i\theta k_2(\tilde{\delta}_{D_3}-1)}=\rho^{k_2(\tilde{\delta}_{D_3}-1)}\left(-e^{i\theta k_2(\tilde{\delta}_{D_3}-1)}+\frac{1}{\rho^{k_2(\tilde{\delta}_{D_3}-1)}}\right).$$
Now, there exists $C_1>0$ such that 
$$\left|-e^{i\theta k_2(\tilde{\delta}_{D_3}-1)}+\frac{1}{\rho^{k_2(\tilde{\delta}_{D_3}-1)}}\right|\ge C_1$$
for every $\rho\ge 0$. By construction, we also have $\rho^{k_2(\tilde{\delta}_{D_3}-1)}=|\tau_2|^{k_2(\tilde{\delta}_{D_3}-1)}/|\tau_{2,k}|^{k_2(\tilde{\delta}_{D_3}-1)}.$ We deduce the existence of $C_2>0$ such that $\rho^{k_2(\tilde{\delta}_{D_3}-1)}\ge C_2|\tau_2|^{k_2(\tilde{\delta}_{D_3}-1)}.$

As a result, we see that
\begin{equation}\label{e525}
\left|\frac{P_{m}(\btau)}{Q(im)}\right|\ge C_1C_2|\tau_2|^{k_2(\tilde{\delta}_{D_3}-1)},
\end{equation}
for every $\tau_2\in S_{d_2}$ with $|\tau_2|\ge \rho_0$ and $\tau_1\in D(0,\rho_1)$, for some small enough $\rho_1>0$.

\begin{enumerate}
\item[2.] Assume that $\tau_1\in S_{d_1}$ with $|\tau_1|\ge \rho_1$ for some fixed $\rho_1>0$, and $\tau_2\in S_{d_2}$. 
\end{enumerate}

We select $S_{d_1}$ in such a way that for $\tau_1\in S_{d_1}$ one can write 
$$\tau_1=\xi_1 e^{i\theta_1}\tau_{2}^{\frac{k_2(\tilde{\delta}_{d_3}-\tilde{\delta}_{D_2})}{k_1\delta_{D_1}}}\left(\frac{R_{D_3}(im)}{R_{D_1,D_2}(im)}\right)^{\frac{1}{k_1\delta_{D_1}}},$$
(here, we have chosen any particular $1/k_1\delta_{D_1}$ root), for some $\xi_1>0$ and $\theta_1$ close to 0, when $\tau_2\in S_{d_2}$. Since $|\tau_1|\ge \rho_1$, we have that $\xi_1>\nu_1>0$ for some fixed $\nu_1>0$.

\textbf{Remark:} This factorization is a particular case of a so-called blow up in the desingularization procedure. We refer to the excellent textbook of Y. Ilyashenko and S. Yakovenko~\cite{iy}, Chapter 1, Section 8, for an introduction to the geometric aspects.

We write
$$\frac{P_m(\btau)}{Q(im)}=1+\tau_2^{k_2(\tilde{\delta}_{D_3}-1)}\xi_1^{k_1\delta_{D_1}}\frac{R_{D_3}(im)}{Q(im)}\left(k_1^{\delta_{D_1}}k_2^{\tilde{\delta}_{D_2}-1}e^{i\theta_1 k_1\delta_{D_1}}+\frac{k_2^{\tilde{\delta}_{D_3}-1}}{\xi_1^{k_1\delta_{D_1}}}\right).$$

Again, taking into account (\ref{e165b}) one can select a sector $S_{d_2}$ which additionally satisfies 
$$\left|\frac{1}{\tau_2^{k_2(\tilde{\delta}_{D_3}-1)}\xi_1^{k_1\delta_{D_1}}}+\frac{R_{D_3}(im)}{Q(im)}(k_1^{\delta_{D_1}}k_2^{\tilde{\delta}_{D_2}-1}e^{i\theta_1k_1\delta_{D_1}}+\frac{k_2^{\tilde{\delta}_{D_3}-1}}{\xi_1^{k_1\delta_{D_1}}})\right|\ge C>0,$$
for some $C>0$, valid for every $\tau_2\in S_{d_2}$, and $\xi_1>\nu_1>0$.
As a result, we get
\begin{multline}\label{e547}
\left|\frac{P_m(\btau)}{Q(im)}\right|\ge C|\xi_1|^{k_1\delta_{D_1}}|\tau_2|^{k_2(\tilde{\delta}_{D_3}-1)}\\
= C|\tau_1|^{k_1\delta_{D_1}}|\tau_2|^{k_2(\tilde{\delta}_{D_2}-1)}\left|\frac{R_{D_1,D_2}(im)}{R_{D_3}(im)}\right|
\ge \tilde{C}|\tau_1|^{k_1\delta_{D_1}}|\tau_2|^{k_2(\tilde{\delta}_{D_2}-1)},
\end{multline}
for some $\tilde{C}>0$ and all $m\in\R$.

As a summary, we have achieved the following result.
\begin{prop}\label{prop614}
There exist $d_1,d_2\in\R$ and $\rho_1>0$ such that for every $m\in\R$ and all $\tau_1\in D(0,\rho_1)\cup S_{d_1}$, $\tau_2\in S_{d_2}$ one has
\begin{equation}\label{e557}
\left|\frac{P_m(\btau)}{Q(im)}\right|\ge C(1+|\tau_1|^{k_1})^{\delta_{D_1}}f(\btau),
\end{equation}
for some $C>0$, and where $f(\btau)$ is defined by
$$f(\btau)=\left\{ \begin{array}{lcc}
             (1+|\tau_2|^{k_2})^{\tilde{\delta}_{D_3}-1} &  if  & |\tau_1|\le \rho_1 \\
             (1+|\tau_2|^{k_2})^{\tilde{\delta}_{D_2}-1} &  if  & |\tau_1|> \rho_1.
             \end{array}
   \right.
$$
\end{prop}

\textbf{Remark:} Without loss of generality, we may assume that $\rho_1\le \rho$, where $\rho>0$ is the radius of the disc of holomorphy with respect to the first time variable appearing in Section~\ref{seclayout}. 

\subsection{Fixed point of a convolution operator in Banach spaces}\label{fixed}

The main purpose of this section is to obtain the existence of a fixed point on certain operator defined in a Banach space. It will allow us to construct the analytic solution of the main problem under study, (\ref{ICP_main0}).

For every $\epsilon\in D(0,\epsilon_0)\setminus\{0\}$, we consider the operator $\mathcal{H}_\epsilon$ defined by 
\begin{multline}
\mathcal{H}_{\epsilon}(\omega(\btau,m))\\
=\frac{-(k_1\tau_1^{k_1})^{\delta_{D_1}}}{P_m(\btau)}\frac{\tilde{\mathcal{A}}_{D_2}(\tau_2)}{k_2\tau_2^{k_2}}R_{D_1,D_2}(im)\omega(\btau,m)- \frac{(k_2\tau_2^{k_2})^{\tilde{\delta}_{D_2}-1}}{P_m(\btau)}\mathcal{A}_{D_1}(\tau_1)R_{D_1,D_2}(im)\omega(\btau,m)\\
-\frac{1}{P_m(\btau)}\mathcal{A}_{D_1}(\tau_1)\frac{\tilde{\mathcal{A}}_{D_2}(\tau_2)}{k_2\tau_2^{k_2}}R_{D_1,D_2}(im)\omega(\btau,m)-\frac{1}{P_m(\btau)}\frac{\tilde{\mathcal{A}}_{D_3}(\tau_2)}{k_2\tau_2^{k_2}}R_{D_3}(im)\omega(\btau,m)\\
+\frac{1}{P_m(\btau)}\frac{\epsilon^{-1}}{(2\pi)^{\frac{1}{2}}}\frac{\tau_1^{k_1}}{k_2\Gamma\left(1+\frac{1}{k_2}\right)}
\int_0^{\tau_2^{k_2}}(\tau_2^{k_2}-s_2)^{\frac{1}{k_2}}\int_{-\infty}^{\infty}\int_0^{\tau_1^{k_1}}\int_0^{s_2}P_1(i(m-m_1),\epsilon)\\
\times \omega((\tau_1^{k_1}-s_1)^{\frac{1}{k_1}},(s_2-x_2)^{\frac{1}{k_2}},m-m_1)P_2(im_1,\epsilon)\omega(s_1^{\frac{1}{k_1}},x_2^{\frac{1}{k_2}},m_1)\frac{dx_2ds_1dm_1ds_2}{(\tau_1^{k_1}-s_1)s_1(s_2-x_2)x_2}\\
+\frac{1}{P_m(\btau)}\sum_{\stackrel{0\le l_j\le D_j-1}{j=1,2}} \epsilon^{\Delta_{l_1,l_2}-d_{l_1}-\tilde{d}_{l_2}+\delta_{l_1}+\tilde{\delta}_{l_2}-1}R_{l_1,l_2}(im)\frac{\tau_1^{k_1}}{k_2\Gamma\left(\frac{d_{l_1,k_1}}{k_1}\right)\Gamma\left(\frac{d_{l_2,k_2}}{k_2}\right)}\\
\times
\int_0^{\tau_2^{k_2}}\int_0^{\tau_1^{k_1}}(\tau_1^{k_1}-s_1)^{\frac{d_{l_1,k_1}}{k_1}-1}(\tau_2^{k_2}-s_2)^{\frac{d_{l_2,k_2}}{k_2}-1}k_1^{\delta_{l_1}}s_1^{\delta_{l_1}}k_2^{\tilde{\delta}_{l_2}}s_2^{\tilde{\delta}_{l_2}}\omega(s_1^{\frac{1}{k_1}},s_2^{\frac{1}{k_2}},m)\frac{ds_1}{s_1}\frac{ds_2}{s_2}\\
+\frac{\epsilon^{-1}}{k_2P_m(\btau)\Gamma\left(1+\frac{1}{k_2}\right)}\int_{0}^{\tau_2^{k_2}}(\tau_2^{k_2}-s_2)^{\frac{1}{k_2}}\psi_{\bk}(\tau_1,s_2^{\frac{1}{k_2}},m,\epsilon)\frac{ds_2}{s_2}.\label{e623}
\end{multline}

\begin{prop}\label{prop653}
Assume that the hypotheses (\ref{e120})-(\ref{e165b}) hold. There exist $\varpi,\xi,R>0$ such that if
$$\left\|\psi_{\bk}(\btau,m)\right\|_{(\bnu,\beta,\mu,\bk,\epsilon)}\le \xi,\qquad \max\{R_1,R_2\}\le R,$$
, where $R_1,R_2$ are the geometric conditions determined in (\ref{eannulus}), for all $\epsilon\in D(0,\epsilon_0)\setminus\{0\}$. Then, the operator $\mathcal{H}_{\epsilon}$ defined in (\ref{e623}) admits a unique fixed point $\omega_{\bk}^{\bd}(\btau,m,\epsilon)\in F^{\bd}_{(\bnu,\beta,\mu,\bk,\epsilon)}$ such that $\left\|\omega_{\bk}^{\bd}(\btau,m,\epsilon)\right\|_{(\bnu,\beta,\mu,\bk,\epsilon)}\le \varpi$, for all $\epsilon\in D(0,\epsilon_0)\setminus\{0\}$.
\end{prop}
\begin{proof}
Take $d_1,d_2\in\R$ and $\rho_1>0$ determined in Proposition~\ref{prop614}. First, we apply Lemma~\ref{lema1} and Lemma~\ref{lema2} to get that
\begin{multline}
\left\|\frac{\tau_1^{k_1\delta_{D_1}}R_{D_1,D_2}(im)}{P_m(\btau)}\int_0^{\tau_2^{k_2}}(\tau_2^{k_2}-s_2)^{\delta_{D_2}-p_2-1}s_2^{p_2-1}\omega(\tau_1,s_2^{\frac{1}{k_2}},m)ds_2\right\|_{(\bnu,\beta,\mu,\bk,\epsilon)}\\
\le \frac{C_1}{C}\sup_{m\in\R}\left|\frac{R_{D_1,D_2}(im)}{Q(im)}\right| \left\|\omega(\btau,m)\right\|_{(\bnu,\beta,\mu,\bk,\epsilon)},\label{e658}
\end{multline}
for every $1\le p_2\le \tilde{\delta}_{D_2}-1$.

In view of Lemma~\ref{lema1} and Lemma~\ref{lema22} we have

\begin{multline}
\left\|\tau_2^{k_2(\tilde{\delta}_{D_2}-1)}\frac{R_{D_1,D_2}(im)}{P_m(\btau)}\tau_1^{k_1}
\int_0^{\tau_1^{k_1}}(\tau_1^{k_1}-s_1)^{\delta_{D_1}-p_1-1}s_1^{p_1-1}\omega(s_1^{\frac{1}{k_1}},\tau_2,m)ds_1\right\|_{(\bnu,\beta,\mu,\bk,\epsilon)}\\
\le \frac{C_1}{C}\sup_{\tau_2\in S_{d_2}}\frac{|\tau_2|^{k_2(\tilde{\delta}_{D_2}-1)}}{(1+|\tau_2|^{k_2})^{\tilde{\delta}_{D_2}-1}}\sup_{m\in\R}\frac{1}{|Q(im)|}\left\|\omega(\btau,m)\right\|_{(\bnu,\beta,\mu,\bk,\epsilon)},\label{e659}
\end{multline}
for every $1\le p_1\le \delta_{D_1}-1$.

Moreover, from Lemma~\ref{lema1} and Lemma~\ref{lema23} we have

\begin{multline}
\left\|\tau_1^{k_1}\frac{R_{D_1,D_2}(im)}{P_m(\btau)}\int_0^{\tau_2^{k_2}}\int_0^{\tau_1^{k_1}}(\tau_2^{k_2}-s_2)^{\tilde{\delta}_{D_2}-p_2-1}s_2^{p_2-1}(\tau_1^{k_1}-s_1)^{\delta_{D_1}-p_1-1}s_1^{p_1-1}\right.\\
\left.\times \omega(s_1^{\frac{1}{k_1}},s_2^{\frac{1}{k_2}},m)ds_1ds_2\right\|_{(\bnu,\beta,\mu,\bk,\epsilon)}\le\frac{C_1}{C}\sup_{m\in\R}\left|\frac{R_{D_1,D_2}(im)}{Q(im)}\right| \left\|\omega(\btau,m)\right\|_{(\bnu,\beta,\mu,\bk,\epsilon)},\label{e660}
\end{multline}
for every $1\le p_1\le \delta_{D_1}-1$ and $1\le p_2\le \tilde{\delta}_{D_2}-1$.

We apply Lemma~\ref{lema1} and Lemma~\ref{lema2} to get 
\begin{multline}
\left\|\frac{1}{P_m(\btau)}\int_0^{\tau_2^{k_2}}(\tau_2^{k_2}-s_2)^{\delta_{D_3}-p_3-1}s_2^{p_3-1}\omega(\tau_1,s_2^{\frac{1}{k_2}},m)ds_2\right\|_{(\bnu,\beta,\mu,\bk,\epsilon)}\\
\le \frac{C_1}{C}\sup_{m\in\R}\left|\frac{1}{Q(im)}\right| \left\|\omega(\btau,m)\right\|_{(\bnu,\beta,\mu,\bk,\epsilon)},\label{e661}
\end{multline}
for every $1\le p_3\le \tilde{\delta}_{D_3}-1$.

Regarding Lemma~\ref{lema1} and Lemma~\ref{lema6}, we deduce that

\begin{multline}
\left\|\frac{\epsilon^{-1}}{P_m(\btau)}\tau_1^{k_1}\int_0^{\tau_2^{k_2}}(\tau_2^{k_2}-s_2)^{\frac{1}{k_2}}\int_{-\infty}^{\infty}\int_0^{\tau_1^{k_1}}\int_0^{s_2}P_1(i(m-m_1),\epsilon)\right.\\
\left.\times \omega((\tau_1^{k_1}-s_1)^{\frac{1}{k_1}},(s_2-x_2)^{\frac{1}{k_2}},m-m_1)P_2(im_1,\epsilon)\omega(s_1^{\frac{1}{k_1}},x_2^{\frac{1}{k_2}},m_1)\frac{dx_2ds_1dm_1ds_2}{(\tau_1^{k_1}-s_1)s_1(s_2-x_2)x_2}\right\|\\
\le |\epsilon|\frac{C_3}{\hbox{max}_{m\in\R}|Q(im)|}\left\|\omega(\btau,m)\right\|_{(\bnu,\beta,\mu,\bk,\epsilon)}^2\label{e662}
\end{multline}

We apply Lemma~\ref{lema1} and Lemma~\ref{lema23} to get
\begin{multline}
\left\|\frac{R_{l_1,l_2}(im)}{P_m(\btau)}\tau_1^{k_1}\int_0^{\tau_2^{k_2}}\int_0^{\tau_1^{k_1}}(\tau_1^{k_1}-s_1)^{\frac{d_{l_1,k_1}}{k_1}-1}(\tau_2^{k_2}-s_2)^{\frac{d_{l_2,k_2}}{k_2}-1}s_1^{\delta_{l_1}-1}s_2^{\tilde{\delta}_{l_2}-1}\right.\\
\left.\times \omega(s_1^{\frac{1}{k_1}},s_2^{\frac{1}{k_2}},m)ds_1 ds_2 \right\|_{(\bnu,\beta,\mu,\bk,\epsilon)} \le\frac{C_1}{C} \sup_{m\in\R}\left|\frac{R_{l_1,l_2}(im)}{Q(im)}\right|\left\|\omega(\btau,m)\right\|_{(\bnu,\beta,\mu,\bk,\epsilon)}
\end{multline}\label{e663}
for every $0\le l_j\le D_j-1$, for $j=1,2$.

Finally, the application of Lemma~\ref{lema1} and Lemma~\ref{lemaaux} yield
\begin{equation}\label{e664}
\left\|\frac{1}{P_m(\btau)}\int_{0}^{\tau_2^{k_2}}(\tau_2^{k_2}-s_2)^{\frac{1}{k_2}}\psi_{\bk}(\tau_1,s_2^{\frac{1}{k_2}},m,\epsilon)\frac{ds_2}{s_2}\right\|_{(\bnu,\beta,\mu,\bk,\epsilon)}\le\frac{C_1}{C}\sup_{m\in\R}\frac{1}{|Q(im)|}|\epsilon|\left\|\psi_{\bk}(\btau,m)\right\|_{(\bnu,\beta,\mu,\bk,\epsilon)}
\end{equation}

Take small enough $\varpi,\xi,\epsilon_0>0$ and assume that 
$$\sup_{m\in\R}\left|\frac{R_{D_1,D_2}(im)}{Q(im)}\right|\le R \quad \hbox{ and } \sup_{m\in\R}\left|\frac{R_{D_3}(im)}{Q(im)}\right|\le R,$$ 
in such a way that


\begin{multline}
k_1^{\delta_{D_1}}\frac{C_1 R}{C}\sum_{1\le p_2\le\tilde{\delta}_{D_2}-1}\frac{|A_{\tilde{\delta}_{D_2},p_2}|}{\Gamma(\tilde{\delta}_{D_2}-p_2)}\varpi+k_2^{\tilde{\delta}_{D_2}-1}\frac{C_1R}{C}\sum_{1\le p_1\le\delta_{D_1}-1}\frac{|A_{\delta_{D_1},p_1}|}{\Gamma(\delta_{D_1}-p_1)}\varpi\\
\frac{C_1R}{C}\sum_{1\le p_1\le\delta_{D_1}-1}\sum_{1\le p_2\le\tilde{\delta}_{D_2}-1}\frac{|A_{\delta_{D_1},p_1}|}{\Gamma(\delta_{D_1}-p_1)}\frac{|A_{\tilde{\delta}_{D_2},p_2}|}{\Gamma(\tilde{\delta}_{D_2}-p_2)}\varpi\\
\frac{C_1R}{C}\sum_{1\le p_3\le\tilde{\delta}_{D_3}-1}\frac{|A_{\tilde{\delta}_{D_3},p_3}|}{\Gamma(\tilde{\delta}_{D_3}-p_3)}\varpi+C_3\sup_{m\in\R}\frac{1}{|Q(im)|}\frac{1}{(2\pi)^{\frac{1}{2}}}\frac{1}{k_2\Gamma(1+\frac{1}{k_2})}\varpi^2\\
+\frac{C_1}{C}\sum_{\stackrel{0\le l_j\le D_j-1}{j=1,2}} \epsilon_0^{\Delta_{l_1,l_2}-\delta_{D_1}k_1-\tilde{\delta}_{D_2}k_2+k_2-1}\sup_{m\in\R}\left|\frac{R_{l_1,l_2}(im)}{Q(im)}\right|\frac{k_1^{\delta_{l_1}}k_2^{\tilde{\delta}_{l_2}-1}}{\Gamma\left(\frac{d_{l_1,k_1}}{k_1}\right)\Gamma\left(\frac{d_{l_2,k_2}}{k_1}\right)}\varpi\\
+\frac{C_1}{C k_2\Gamma\left(1+\frac{1}{k_2}\right)} \sup_{m\in\R}\frac{1}{|Q(im)|}\xi\le\varpi.\label{e677}
\end{multline}

Taking into account (\ref{e658}-\ref{e664}) and (\ref{e677}), we get that $\mathcal{H}_{\epsilon}(\overline{D}(0,\varpi))\subseteq \overline{D}(0,\varpi)$. Here, $\overline{D}(0,\varpi)$ stands for the closed disc of radius $\varpi$ centered at the origin in the Banach space $F^{\bd}_{(\bnu,\beta,\mu,\bk,\epsilon)}$. Now, let $\omega_1,\omega_2\in F^{\bd}_{(\bnu,\beta,\mu,\bk,\epsilon)}$, with $\left\|\omega_j(\btau,m)\right\|_{(\bnu,\beta,\mu,\bk,\epsilon)}\le\varpi$. We now prove that 
\begin{equation}\label{e741b}
\left\|\mathcal{H}_\epsilon(\omega_1)-\mathcal{H}_\epsilon(\omega_2)\right\|_{(\bnu,\beta,\mu,\bk,\epsilon)}\le\frac{1}{2}\left\|\omega_1-\omega_2\right\|_{(\bnu,\beta,\mu,\bk,\epsilon)}.
\end{equation}
At this point, the classical contractive mapping theorem acting on the complete metric space $\overline{D}(0,\varpi)\subseteq  F^{\bd}_{(\bnu,\beta,\mu,\bk,\epsilon)}$ guarantees the existence of a fixed point for $\mathcal{H}_\epsilon$. Let us check (\ref{e741b}).

Analogous estimates as in the first part of the proof yield

\begin{multline}
\left\|\frac{\tau_1^{k_1\delta_{D_1}}R_{D_1,D_2}(im)}{P_m(\btau)}\int_0^{\tau_2^{k_2}}(\tau_2^{k_2}-s_2)^{\delta_{D_2}-p_2-1}s_2^{p_2-1}(\omega_1(\tau_1,s_2^{\frac{1}{k_2}},m)-\omega_2(\tau_1,s_2^{\frac{1}{k_2}},m))ds_2\right\|_{(\bnu,\beta,\mu,\bk,\epsilon)}\\
\le \frac{C_1}{C}\sup_{m\in\R}\left|\frac{R_{D_1,D_2}(im)}{Q(im)}\right| \left\|\omega_1(\btau,m)-\omega_2(\btau,m)\right\|_{(\bnu,\beta,\mu,\bk,\epsilon)},\label{e658b}
\end{multline}
for every $1\le p_2\le \tilde{\delta}_{D_2}-1$. Also Lemma~\ref{lema1} and Lemma~\ref{lema22} yield

\begin{multline}
\left\|\tau_2^{k_2(\tilde{\delta}_{D_2}-1)}\frac{R_{D_1,D_2}(im)}{P_m(\btau)}\tau_1^{k_1}
\int_0^{\tau_1^{k_1}}(\tau_1^{k_1}-s_1)^{\delta_{D_1}-p_1-1}s_1^{p_1-1}(\omega_1(s_1^{\frac{1}{k_1}},\tau_2,m)-\omega_2(s_1^{\frac{1}{k_1}},\tau_2,m))ds_1\right\|_{(\bnu,\beta,\mu,\bk,\epsilon)}\\
\le \frac{C_1}{C}\sup_{\tau_2\in S_{d_2}}\frac{|\tau_2|^{k_2(\tilde{\delta}_{D_2}-1)}}{(1+|\tau_2|^{k_2})^{\tilde{\delta}_{D_2}-1}}\sup_{m\in\R}\frac{1}{|Q(im)|}\left\|\omega_1(\btau,m)-\omega_2(\btau,m)\right\|_{(\bnu,\beta,\mu,\bk,\epsilon)},\label{e659b}
\end{multline}
for every $1\le p_1\le \delta_{D_1}-1$. Lemma~\ref{lema1} and Lemma~\ref{lema23} guarantee that

\begin{multline}
\left\|\tau_1^{k_1}\frac{R_{D_1,D_2}(im)}{P_m(\btau)}\int_0^{\tau_2^{k_2}}\int_0^{\tau_1^{k_1}}(\tau_2^{k_2}-s_2)^{\tilde{\delta}_{D_2}-p_2-1}s_2^{p_2-1}(\tau_1^{k_1}-s_1)^{\delta_{D_1}-p_1-1}s_1^{p_1-1}\right.\\
\left.\times (\omega_1(s_1^{\frac{1}{k_1}},s_2^{\frac{1}{k_2}},m)-\omega_1(s_1^{\frac{1}{k_1}},s_2^{\frac{1}{k_2}},m))ds_1ds_2\right\|_{(\bnu,\beta,\mu,\bk,\epsilon)}\\
\le\frac{C_1^2}{C}\sup_{m\in\R}\left|\frac{R_{D_1,D_2}(im)}{Q(im)}\right| \left\|w(\btau,m)\right\|_{(\bnu,\beta,\mu,\bk,\epsilon)},\label{e660b}
\end{multline}
for every $1\le p_1\le \delta_{D_1}-1$ and $1\le p_2\le \tilde{\delta}_{D_2}-1$.

We apply Lemma~\ref{lema1} and Lemma~\ref{lema2} to get 
\begin{multline}
\left\|\frac{1}{P_m(\btau)}\int_0^{\tau_2^{k_2}}(\tau_2^{k_2}-s_2)^{\tilde{\delta}_{D_3}-p_3-1}s_2^{p_3-1}(\omega_1(\tau_1,s_2^{\frac{1}{k_2}},m)-\omega_2(\tau_1,s_2^{\frac{1}{k_2}},m))ds_2\right\|_{(\bnu,\beta,\mu,\bk,\epsilon)}\\
\le \frac{C_1}{C}\sup_{m\in\R}\left|\frac{1}{Q(im)}\right| \left\|\omega_1(\btau,m)-\omega_2(\btau,m)\right\|_{(\bnu,\beta,\mu,\bk,\epsilon)},\label{e661b}
\end{multline}
for every $1\le p_3\le \tilde{\delta}_{D_3}-1$.

In order to study the convolution operator, we need to give some details on the procedure. Put 
$$W_1:=\omega_{1}((\tau_1-s_1)^{1/k_1},(s_2-x_2)^{1/k_2},m-m_{1}) - \omega_{2}((\tau_1-s_1)^{1/k_1},(s_2-x_2)^{1/k_2},m-m_{1}),$$
and $W_2:=\omega_{1}(s_1^{1/k_1},x_2^{1/k_2},m_{1}) - \omega_{2}(s_1^{1/k_1},x_2^{1/k_2},m_{1}).$
Then, taking into account that
\begin{multline}
P_{1}(i(m-m_{1}),\epsilon)\omega_{1}((\tau_1-s_1)^{1/k_1},(s_2-x_2)^{1/k_2},m-m_{1})P_{2}(im_{1},\epsilon)\omega_{1}(s_1^{1/k_1},x_2^{1/k_2},m_{1})\\
-P_{1}(i(m-m_{1}),\epsilon)\omega_{2}((\tau_1-s_1)^{1/k_1},(s_2-x_2)^{1/k_2},m-m_{1})P_{2}(im_{1},\epsilon)\omega_{2}(s_1^{1/k_1},x_2^{1/k_2},m_{1})\\
= P_{1}(i(m-m_{1}),\epsilon)W_1 P_2(im_1,\epsilon)\omega_{1}(s_1^{1/k_1},x_2^{1/k_2},m_{1})\\
+ P_{1}(i(m-m_{1}),\epsilon)\omega_{2}((\tau_1-s_1)^{1/k_1},(s_2-x_2)^{1/k_2},m-m_{1})P_2(im_1,\epsilon)W_2,
\end{multline} 
and due to Lemma~\ref{lema1} and Lemma~\ref{lema6}, we proceed with analogous estimates as in (\ref{e662}) to get that
\begin{multline}
\left\|\frac{\epsilon^{-1}}{P_m(\btau)}\tau_1^{k_1}\int_0^{\tau_2^{k_2}}(\tau_2^{k_2}-s_2)^{\frac{1}{k_2}}\int_{-\infty}^{\infty}\int_0^{\tau_1^{k_1}}\int_0^{s_2}P_1(i(m-m_1),\epsilon)\right.\\
\times (\omega_1((\tau_1^{k_1}-s_1)^{\frac{1}{k_1}},(s_2-x_2)^{\frac{1}{k_2}},m-m_1)-\omega_2((\tau_1^{k_1}-s_1)^{\frac{1}{k_1}},(s_2-x_2)^{\frac{1}{k_2}},m-m_1))\\
\left.\times P_2(im_1,\epsilon)(\omega_1(s_1^{\frac{1}{k_1}},x_2^{\frac{1}{k_2}},m_1)-\omega_2(s_1^{\frac{1}{k_1}},x_2^{\frac{1}{k_2}},m_1))\frac{dx_2ds_1dm_1ds_2}{(\tau_1^{k_1}-s_1)s_1(s_2-x_2)x_2}\right\|\\
\le |\epsilon|\frac{C_3}{\hbox{max}_{m\in\R}|Q(im)|}\left(\left\|\omega_1(\btau,m)\right\|_{(\bnu,\beta,\mu,\bk,\epsilon)}+\left\|\omega_2(\btau,m)\right\|_{(\bnu,\beta,\mu,\bk,\epsilon)}\right)\left\|\omega_1(\btau,m)-\omega_2(\btau,m)\right\|_{(\bnu,\beta,\mu,\bk,\epsilon)}\label{e662b}
\end{multline}

Finally, we apply Lemma~\ref{lema1} and Lemma~\ref{lema23} to get
\begin{multline}
\left\|\frac{R_{l_1,l_2}(im)}{P_m(\btau)}\tau_1^{k_1}\int_0^{\tau_2^{k_2}}\int_0^{\tau_1^{k_1}}(\tau_1^{k_1}-s_1)^{\frac{d_{l_1,k_1}}{k_1}-1}(\tau_2^{k_2}-s_2)^{\frac{d_{l_2,k_2}}{k_2}-1}s_1^{\delta_{l_1}-1}s_2^{\tilde{\delta}_{l_2}-1}\right.\\
\left.\times (\omega_1(s_1^{\frac{1}{k_1}},s_2^{\frac{1}{k_2}},m)-\omega_2(s_1^{\frac{1}{k_1}},s_2^{\frac{1}{k_2}},m))ds_1 ds_2 \right\|_{(\bnu,\beta,\mu,\bk,\epsilon)}  \\
\le\frac{C_1}{C} \sup_{m\in\R}\left|\frac{R_{l_1,l_2}(im)}{Q(im)}\right|\left\|\omega_1(\btau,m)-\omega_2(\btau,m)\right\|_{(\bnu,\beta,\mu,\bk,\epsilon)}
\end{multline}\label{e663b}
for every $1\le l_j\le D_j-1$, for $j=1,2$.

We choose small enough $\varpi,\epsilon_0>0$ and assume that 
$$\sup_{m\in\R}\left|\frac{R_{D_1,D_2}(im)}{Q(im)}\right|\le R\quad \hbox{ and } \sup_{m\in\R}\left|\frac{R_{D_3}(im)}{Q(im)}\right|\le R,$$ 
to satisfy

\begin{multline*}
k_1^{\delta_{D_1}}\frac{C_1R}{C}\sum_{1\le p_2\le\tilde{\delta}_{D_2}-1}\frac{|A_{\tilde{\delta}_{D_2},p_2}|}{\Gamma(\tilde{\delta}_{D_2}-p_2)}+k_2^{\tilde{\delta}_{D_2}-1}\frac{C_1R}{C}\sum_{1\le p_1\le\delta_{D_1}-1}\frac{|A_{\delta_{D_1},p_1}|}{\Gamma(\delta_{D_1}-p_1)}\\
\frac{C_1R}{C}\sum_{1\le p_1\le\delta_{D_1}-1}\sum_{1\le p_2\le\tilde{\delta}_{D_2}-1}\frac{|A_{\delta_{D_1},p_1}|}{\Gamma(\delta_{D_1}-p_1)}\frac{|A_{\tilde{\delta}_{D_2},p_2}|}{\Gamma(\tilde{\delta}_{D_2}-p_2)}\\
\frac{C_1R}{C}\sum_{1\le p_3\le\tilde{\delta}_{D_3}-1}\frac{|A_{\tilde{\delta}_{D_3},p_3}|}{\Gamma(\tilde{\delta}_{D_3}-p_3)}+2C_3\sup_{m\in\R}\frac{1}{|Q(im)|}\frac{1}{(2\pi)^{\frac{1}{2}}}\frac{1}{k_2\Gamma(1+\frac{1}{k_2})}\varpi\\
+\frac{C_1}{C}\sum_{\stackrel{0\le l_j\le D_j-1}{j=1,2}} \epsilon_0^{\Delta_{l_1,l_2}-\delta_{D_1}k_1-\tilde{\delta}_{D_2}k_2+k_2-1}\sup_{m\in\R}\left|\frac{R_{l_1,l_2}(im)}{Q(im)}\right|\frac{k_1^{\delta_{l_1}}k_2^{\tilde{\delta}_{l_2}-1}}{\Gamma\left(\frac{d_{l_1,k_1}}{k_1}\right)\Gamma\left(\frac{d_{l_2,k_2}}{k_1}\right)}\\
\le\frac{1}{2}.
\end{multline*}

Then, (\ref{e741b}) holds, and the proof is complete.
\end{proof}

The following is a direct consequence of the previous result.
\begin{corol}
The function $\omega_{\bk}^{\bd}(\btau,m,\epsilon)$, obtained in Proposition~\ref{prop653} is a continuous function in $(\overline{D}(0,\rho)\cup S_{d_1})\times S_{d_2}\times \R\times D(0,\epsilon_0)\setminus\{0\}$, and holomorphic with respect to $\btau$ in the set $(D(0,\rho)\cup S_{d_1})\times S_{d_2}$ and on $D(0,\epsilon_0)\setminus\{0\}$ with respect to the perturbation parameter $\epsilon$. Moreover, it turns out to be a solution of (\ref{e310}), which satisfies there exists $\varpi>0$ such that
\begin{equation}\label{e807}
|\omega_{\bk}^{\bd}(\btau,m,\epsilon)|\le \varpi(1+|m|)^{-\mu}
\frac{|\frac{\tau_1}{\epsilon}|}{1 + |\frac{\tau_1}{\epsilon}|^{2k_1}}\frac{|\frac{\tau_2}{\epsilon}|}{1 + |\frac{\tau_2}{\epsilon}|^{2k_2}}\exp( -\beta|m| + \nu_1|\frac{\tau_1}{\epsilon}|^{k_1}+\nu_2|\frac{\tau_2}{\epsilon}|^{k_2} ),
\end{equation}
for every $(\btau,m,\epsilon)\in (\overline{D}(0,\rho)\cup S_{d_1})\times S_{d_2}\times \R\times D(0,\epsilon_0)\setminus\{0\}$.
\end{corol}

\section{Family of analytic solutions of the main problem}

In this section, we consider the main problem under study, namely (\ref{ICP_main0}), under the conditions (\ref{e120})-(\ref{e331}) on the parameters involved, and also on the geometry of the problem, (\ref{raicesgrandes})-(\ref{e165b}). In order to construct the analytic solution of the problem, we recall the definition of a good covering in $\C^{\star}$.

\begin{defin}\label{goodcovering} Let $\varsigma_1,\varsigma_2 \geq 2$ be integer numbers. Let $\{ \mathcal{E}_{p_1,p_2} \}_{\begin{subarray}{l} 0 \leq p_1 \leq \varsigma_1 - 1\\0 \leq p_2 \leq \varsigma_2 - 1\end{subarray}}$ be a finite family of open sectors with vertex at $0$, and radius $\epsilon_{0}$. In addition to this, we assume the opening of every sector is chosen to be slightly larger than $\pi/k_2$ in the case that $k_1<k_2$, and slightly larger than $\pi/k_1$ in case $k_2<k_1$.

We assume that the intersection of three different sectors in the good covering is empty, and
$\cup_{\begin{subarray}{l} 0 \leq p_1 \leq \varsigma_1 - 1\\0 \leq p_2 \leq \varsigma_2 - 1\end{subarray}} \mathcal{E}_{p_1,p_2} = \mathcal{U} \setminus \{ 0 \}$,
for some neighborhood of 0, $\mathcal{U}\in\mathbb{C}$. Such set of sectors is called a good covering in $\mathbb{C}^{\ast}$.

\end{defin}

\begin{defin}\label{defgood2} Let $\varsigma_1,\varsigma_2\ge 2$ and $\{ \mathcal{E}_{p_1,p_2} \}_{\begin{subarray}{l} 0 \leq p_1 \leq \varsigma_1 - 1\\0 \leq p_2 \leq \varsigma_2 - 1\end{subarray}}$ be a good covering in $\mathbb{C}^{\ast}$. Let
$\mathcal{T}_j$ be open bounded sectors centered at 0 with radius $r_{\mathcal{T}_j}$ for $j\in\{1,2\}$, and consider two families of sectors as follows: let
$$ S_{\mathfrak{d}_{p_1},\theta_1,\epsilon_{0}r_{\mathcal{T}_1}} =
\{ T_1 \in \mathbb{C}^{\ast} / |T_1| < \epsilon_{0}r_{\mathcal{T}_1} \ \ , \ \ |\mathfrak{d}_{p_1} - \mathrm{arg}(T_1)| < \theta_1/2 \}, $$
$$ S_{\tilde{\mathfrak{d}}_{p_2},\theta_2,\epsilon_{0}r_{\mathcal{T}_2}} =
\{ T_2 \in \mathbb{C}^{\ast} / |T_2| < \epsilon_{0}r_{\mathcal{T}_2} \ \ , \ \ |\mathfrak{d}_{p_2} - \mathrm{arg}(T_2)| < \theta_2/2 \},$$
with opening $\theta_j > \pi/k_j$, and where  $\mathfrak{d}_{p_1},\tilde{\mathfrak{d}}_{p_2} \in \mathbb{R}$, for all $0 \leq p_1 \leq \varsigma_1-1$ and $0 \leq p_2 \leq \varsigma_2-1$ is the couple of directions $d_1,d_2\in\R$ mentioned in Proposition~\ref{prop614}, whenever $\mathcal{E}_{p_1,p_2}$ is the domain of definition of the perturbation parameter $\epsilon$. 

In addition to that, the sectors $S_{\mathfrak{d}_{p_1},\theta_1,\epsilon_{0}r_{\mathcal{T}_1}}$ and $S_{\tilde{\mathfrak{d}}_{p_2},\theta_2,\epsilon_{0}r_{\mathcal{T}_2}}$ are such that for all $0 \leq p_1 \leq \varsigma_1 - 1$, $0 \leq p_2 \leq \varsigma_2 - 1$, $\bt \in \mathcal{T}_1\times \mathcal{T}_2$, and $\epsilon \in \mathcal{E}_{p_1,p_2}$, one has
$$\epsilon t_1 \in S_{\mathfrak{d}_{p_1},\theta_1,\epsilon_{0}r_{\mathcal{T}_1}}\hbox{ and }\epsilon t_2 \in S_{\tilde{\mathfrak{d}}_{p_2},\theta_2,\epsilon_{0}r_{\mathcal{T}_2}}.$$

\noindent We say that the family
$\{ (S_{\mathfrak{d}_{p_1},\theta_1,\epsilon_{0}r_{\mathcal{T}_1}})_{0 \leq p_1 \leq \varsigma_1-1}, (S_{\tilde{\mathfrak{d}}_{p_2},\theta_2,\epsilon_{0}r_{\mathcal{T}_2}})_{0 \leq p_2 \leq \varsigma_2-1} ,\mathcal{T}_1\times \mathcal{T}_2 \}$
is associated to the good covering $\{ \mathcal{E}_{p_1,p_2} \}_{\begin{subarray}{l}0 \leq p_1 \leq \varsigma_1 - 1\\0 \leq p_2 \leq \varsigma_2 - 1\end{subarray}}$.
\end{defin}

Let $\varsigma_1,\varsigma_2\ge 2$ and $\{ \mathcal{E}_{p_1,p_2} \}_{\begin{subarray}{l} 0 \leq p_1 \leq \varsigma_1 - 1\\0 \leq p_2 \leq \varsigma_2 - 1\end{subarray}}$ be a good covering in $\mathbb{C}^{\ast}$. We assume the family $\{ (S_{\mathfrak{d}_{p_1},\theta_1,\epsilon_{0}r_{\mathcal{T}_1}})_{0 \leq p_1 \leq \varsigma_1-1}, (S_{\tilde{\mathfrak{d}}_{p_2},\theta_2,\epsilon_{0}r_{\mathcal{T}_2}})_{0 \leq p_2 \leq \varsigma_2-1} ,\mathcal{T}_1\times \mathcal{T}_2 \}$ is associated to the previous good covering.

The existence of a solution $\omega_{\bk}^{\bd}(\btau,m,\epsilon)$ of the auxiliary problem (\ref{e310}) turns out to provide an actual solution of the main problem via Laplace and Fourier transform, in view of the constraints satisfied by $\omega_{\bk}^{\bd}(\btau,m,\epsilon)$, see (\ref{e807}). More precisely, for every $0\le p_1\le \varsigma_1$ and $0\le p_2\le \varsigma_2-1$, the function

\begin{equation}\label{e817}
u_{p_1,p_2}(\bt,z,\epsilon)=\frac{k_1k_2}{(2\pi)^{1/2}}\int_{-\infty}^{+\infty}
\int_{L_{\gamma_{p_1}}}\int_{L_{\gamma_{p_2}}}
\omega_{\bk}^{\mathfrak{d}_{p_1},\tilde{\mathfrak{d}}_{p_2}}(u_1,u_2,m,\epsilon) e^{-(\frac{u_1}{\epsilon t_1})^{k_1}-(\frac{u_2}{\epsilon t_2})^{k_2}} e^{izm} \frac{du_2}{u_2}\frac{du_1}{u_1} dm,
\end{equation}
is holomorphic on the domain $(\mathcal{T}_1\cap D(0,h'))\times(\mathcal{T}_2\cup D(0,h'))\times H_{\beta'}\times \mathcal{E}_{p_1,p_2}$, for any $0<\beta'<\beta$ and some $h'>0$.

The first main result of the present work is devoted to the construction of a family of actual holomorphic solutions to the equation (\ref{ICP_main0}) for null initial data. Each of the elements in the family of solutions is associated to an element of a good covering with respect to the complex parameter $\epsilon$. The strategy leans on the control of the difference of two solutions defined in domains with nonempty intersection with respect to the perturbation parameter $\epsilon$. The construction of each analytic solution in terms of two Laplace transforms in different time variables requires to distinguish different cases, depending on the coincidence of the integration paths or not.

\begin{theo}\label{teo1}
Let the hypotheses of Proposition~\ref{prop653} hold. Then, for every element $\mathcal{E}_{p_1,p_2}$ in the good covering in $\C^{\star}$, there exists a solution $u_{p_1,p_2}(\bt,z,\epsilon)$ of the main problem under study (\ref{ICP_main0}) defined and holomorphic on $(\mathcal{T}_1\cap D(0,h'))\times(\mathcal{T}_2\cup D(0,h'))\times H_{\beta'}\times \mathcal{E}_{p_1,p_2}$, for any $0<\beta'<\beta$ and some $h'>0$. 

Moreover, for every two different multiindices $(p_1,p_2),(p'_1,p'_2)\in\{0,\ldots,\varsigma_1-1\}\times\{0,\ldots,\varsigma_2-1\}$, one of the following situations hold:

\begin{itemize}
\item Case 1: $\mathcal{E}_{p_1,p_2}\cap \mathcal{E}_{p'_1,p'_2}=\emptyset$.
\item Case 2: $\mathcal{E}_{p_1,p_2}\cap \mathcal{E}_{p'_1,p'_2}\neq\emptyset$. The path $L_{\gamma_{p_2}}$ coincides with $L_{\gamma_{p'_2}}$ but $L_{\gamma_{p_1}}$ does not coincide with $L_{\gamma_{p'_1}}$. Then, it holds that
\begin{equation}
\sup_{\bt \in (\mathcal{T}_1 \cap D(0,h''))\times (\mathcal{T}_2 \cap D(0,h'')), z \in H_{\beta'}}
|u_{p_1,p_2}(\bt,z,\epsilon) - u_{p'_1,p'_2}(\bt,z,\epsilon)| \leq K_{p}e^{-\frac{M_p}{|\epsilon|^{k_1}}},
\label{exp_small_difference_u_p11}
\end{equation}
for every $\epsilon\in\mathcal{E}_{p_1,p_2}\cap\mathcal{E}_{p'_1,p'_2}$. In that case, we say that $((p_1,p_2),(p'_1,p'_2))$ belongs to the subset $\mathcal{U}_{k_1}$ of $\{0,\ldots,\varsigma_1-1\}\times\{0,\ldots,\varsigma_2-1\}$.
\item Case 3: $\mathcal{E}_{p_1,p_2}\cap \mathcal{E}_{p'_1,p'_2}\neq\emptyset$. Neither, the path $L_{\gamma_{p_2}}$ coincides with $L_{\gamma_{p'_2}}$, nor $L_{\gamma_{p_1}}$ coincides with $L_{\gamma_{p'_1}}$. Then, it holds that
\begin{equation}
\sup_{\bt \in (\mathcal{T}_1 \cap D(0,h''))\times (\mathcal{T}_2 \cap D(0,h'')), z \in H_{\beta'}}
|u_{p_1,p_2}(\bt,z,\epsilon) - u_{p'_1,p'_2}(\bt,z,\epsilon)| \leq K_{p}\max \left\{ e^{-\frac{M_p}{|\epsilon|^{k_1}}}, e^{-\frac{M_p}{|\epsilon|^{k_2}}} \right\},
\label{exp_small_difference_u_p12}
\end{equation}
for every $\epsilon\in\mathcal{E}_{p_1,p_2}\cap\mathcal{E}_{p'_1,p'_2}$.
\end{itemize}
\end{theo}

\begin{proof}
The existence of the solution $u_{p_1,p_2}(\bt,z,\epsilon)$, for every $0\le p_1\le \varsigma_1$ and $0\le p_2\le \varsigma_2-1$ is guaranteed from the construction described previously.

We now give proof for the second statement of the result, namely, the existence of an exponential decay to 0, with respect to the perturbation parameter, of the difference of two consecutive solutions in the good covering, uniformly with respect to $(\bt,z)$. 

The proof is close to that of Theorem 1 in~\cite{family1}, but for the sake of clarity, we give a complete description.\par

\textbf{Case 2:} Assume that the path $L_{\gamma_{p_2}}$ coincides with $L_{\gamma_{p'_2}}$, and $L_{\gamma_{p_1}}$ does not coincide with $L_{\gamma_{p'_1}}$. Then, using that
 $u_1 \mapsto \omega_{\bk}^{\mathfrak{d}_{p_1},\tilde{\mathfrak{d}}_{p_2}}(u_1,u_2,m,\epsilon) \exp( -(\frac{u_1}{\epsilon t_1})^{k_1} )/u_1$ is holomorphic on $D(0,\rho)$ for all
$(m,\epsilon) \in \mathbb{R} \times (D(0,\epsilon_{0}) \setminus \{ 0 \})$, and every $u_2\in L_{\gamma_{p_2}}$, one can deform one of the integration paths and write
$$I=\int_{L_{\gamma_{p_1}}}\omega_{\bk}^{\mathfrak{d}_{p_1},\tilde{\mathfrak{d}}_{p_2}}(u_1,u_2,m,\epsilon)e^{-\left(\frac{u_1}{\epsilon t_1}\right)^{k_1}}\frac{du_1}{u_1}-\int_{L_{\gamma_{p'_1}}}\omega_{\bk}^{\mathfrak{d}_{p_1},\tilde{\mathfrak{d}}_{p_2}}(u_1,u_2,m,\epsilon)e^{-\left(\frac{u_1}{\epsilon t_1}\right)^{k_1}}\frac{du_1}{u_1}$$
in the form
\begin{multline*}
\int_{L_{\rho_1/2,\gamma_{p_1}}}
\omega_{\bk}^{\mathfrak{d}_{p_1},\tilde{\mathfrak{d}}_{p_2}}(u_1,u_2,m,\epsilon) e^{-(\frac{u_1}{\epsilon t_1})^{k_1}} \frac{du_1}{u_1} \\ 
-\int_{L_{\rho_1/2,\gamma_{p'_1}}}
\omega_{\bk}^{\mathfrak{d}_{p_1},\tilde{\mathfrak{d}}_{p_2}}(u_1,u_2,m,\epsilon) e^{-(\frac{u_1}{\epsilon t_1})^{k_1}} \frac{du_1}{u_1}\\
+ \int_{C_{\rho_1/2,\gamma_{p'_1},\gamma_{p_1}}}
\omega_{\bk}^{\mathfrak{d}_{p_1},\tilde{\mathfrak{d}}_{p_2}}(u_1,u_2,m,\epsilon) e^{-(\frac{u_1}{\epsilon t_1})^{k_1}} \frac{du_1}{u_1}. 
\end{multline*}
where $L_{\rho_1/2,\gamma_{p_1}} = [\rho_1/2,+\infty)e^{i\gamma_{p_1}}$,
$L_{\rho_1/2,\gamma_{p'_1}} = [\rho_1/2,+\infty)e^{i\gamma_{p'_1}}$ and
$C_{\rho_1/2,\gamma_{p'_1},\gamma_{p_1}}$ is an arc of circle connecting
$(\rho_1/2)e^{i\gamma_{p'_1}}$ and $(\rho_1/2)e^{i\gamma_{p_1}}$ with the adequate orientation. The positive real number $\rho_1$ is determined in Proposition~\ref{prop614}.\medskip

We get the existence of constants $C_{p_1,p'_1},M_{p_1,p'_1}>0$ such that
$$|I|\le C_{p_1,p'_1}\varpi_{\mathfrak{d}_{p_1},\tilde{\mathfrak{d}}_{p_2}}(1+ |m|)^{-\mu} e^{-\beta|m|} \frac{ |\frac{u_2}{\epsilon}|}{1 + |\frac{u_2}{\epsilon}|^{2k_2}}
\exp( \nu_2 |\frac{u_2}{\epsilon}|^{k_2})e^{-\frac{M_{p_1,p'_1}}{|\epsilon|^{k_1}}},$$
for $t_1\in\mathcal{T}_1\cap D(0,h')$ and $\epsilon\in\mathcal{E}_{p_1,p_2}\cap\mathcal{E}_{p'_1,p'_2}$ and $u_2\in L_{\gamma_{p_2}}$. We have
\begin{multline}
|u_{p_1,p_2}(\bt,z,\epsilon)-u_{p'_1,p'_2}(\bt,z,\epsilon)|\\
\le \frac{k_1k_2}{(2\pi)^{1/2}}C_{p_1,p'_1}\left(\int_{-\infty}^{\infty}(1+|m|)^{-\mu}e^{-\beta|m|}e^{-m|\hbox{Im}(z)|}dm\right)\\
\times \int_{L_{\gamma_{p_2}}}\frac{ |\frac{u_2}{\epsilon}|}{1 + |\frac{u_2}{\epsilon}|^{2k_2}}
\exp( \nu_2 |\frac{u_2}{\epsilon}|^{k_2})\exp(-\left(\frac{u_2}{\epsilon t_2}\right)^{k_2})\left|\frac{du_2}{u_2}\right|    e^{-\frac{M_{p_1,p'_1}}{|\epsilon|^{k_1}}}.
\end{multline}
The last integral is estimated via the reparametrization $u_2=re^{\gamma_{p_2}\sqrt{-1}}$ and the change of variable $r=|\epsilon|s$ by 
$$\int_0^{\infty}\frac{1}{1+s^2}e^{-\delta_{2}s^{k_2}}ds,$$
for some $\delta_{2}>0$, whenever $t_2\in\mathcal{T}_2\cap D(0,h')$.

The estimates given in the enunciate of Case 2 follows from here.

\begin{figure}
	\centering
	\includegraphics[width=0.6\textwidth]{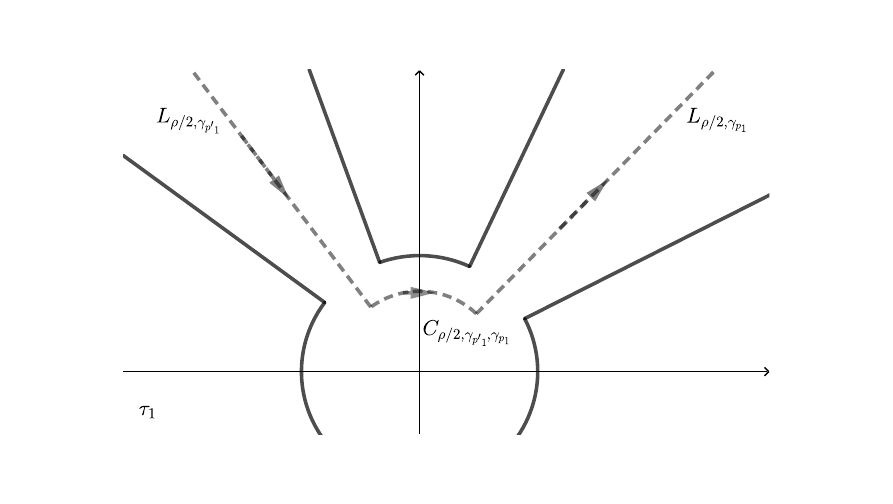}
	\caption{Path deformation in Case 2}
\end{figure}

\textbf{Case 3:} Assume that neither $L_{\gamma_{p_1}}$ coincides with $L_{\gamma_{p'_1}}$, nor $L_{\gamma_{p_2}}$ coincides with $L_{\gamma_{p'_2}}$.

Owing to the fact that $u_1 \mapsto \omega_{\bk}^{\mathfrak{d}_{p_1},\tilde{\mathfrak{d}}_{p_2}}(u_1,u_2,m,\epsilon) \exp( -(\frac{u_1}{\epsilon t_1})^{k_1} )/u_1$ is holomorphic on $D(0,\rho)$ for all
$(m,\epsilon) \in \mathbb{R} \times (D(0,\epsilon_{0}) \setminus \{ 0 \})$, and every $u_2\in L_{\gamma_{p_2}}$ we deform the integration paths with respect to the first time variable and write

$$u_{p_1,p_2}(\bt,z,\epsilon)-u_{p'_1,p'_2}(\bt,z,\epsilon)=J_1-J_2+J_3,$$
where
$$J_1=\frac{k_1k_2}{(2\pi)^{1/2}} \int_{L_{\gamma_{p_1},1}}\int_{L_{\gamma_{p_2}}}\int_{-\infty}^{\infty} \omega_{\bk}^{\mathfrak{d}_{p_1},\tilde{\mathfrak{d}}_{p_2}}(u_1,u_2,m,\epsilon)
e^{-(\frac{u_1}{\epsilon t_1})^{k_1}-(\frac{u_2}{\epsilon t_2})^{k_2}}e^{izm} dm\frac{du_2}{u_2}\frac{du_1}{u_1}.$$

$$J_2=\frac{k_1k_2}{(2\pi)^{1/2}} \int_{L_{\gamma_{p'_1},1}}\int_{L_{\gamma_{p'_2}}}\int_{-\infty}^{\infty} \omega_{\bk}^{\mathfrak{d}_{p'_1},\tilde{\mathfrak{d}}_{p'_2}}(u_1,u_2,m,\epsilon)
e^{-(\frac{u_1}{\epsilon t_1})^{k_1}-(\frac{u_2}{\epsilon t_2})^{k_2}}e^{izm} dm\frac{du_2}{u_2}\frac{du_1}{u_1}.$$

\begin{multline*}
J_3=\frac{k_1k_2}{(2\pi)^{1/2}} \int_{0}^{\frac{\rho_1}{2}e^{i\theta}}\left(\int_{-\infty}^{\infty}\left(\int_{L_{\gamma_{p_2}}} \omega_{\bk}^{\mathfrak{d}_{p_1},\tilde{\mathfrak{d}}_{p_2}}(u_1,u_2,m,\epsilon)
e^{-(\frac{u_2}{\epsilon t_2})^{k_2}} \frac{du_2}{u_2}\right.\right.\\
\left.\left.-\int_{L_{\gamma_{p'_2}}} \omega_{\bk}^{\mathfrak{d}_{p'_1},\tilde{\mathfrak{d}}_{p'_2}}(u_1,u_2,m,\epsilon)
e^{-(\frac{u_2}{\epsilon t_2})^{k_2}} \frac{du_2}{u_2}\right)e^{izm}dm\right)e^{-(\frac{u_1}{\epsilon t_1})^{k_1}}\frac{du_1}{u_1}, 
\end{multline*}
where $\frac{\rho_1}{2}e^{i\theta}$ is such that $\theta$ is an argument between $\gamma_{p_1}$ and $\gamma_{p'_1}$. The path $L_{\gamma_{p_1},1}$ (resp. $L_{\gamma_{p'_1},1}$) consists of the concatenation of the arc of circle connecting $\frac{\rho_1}{2}e^{i\theta}$ with $\frac{\rho_1}{2}e^{i\gamma_{p_1}}$ (resp. with $\frac{\rho_1}{2}e^{i\gamma_{p'_1}}$) and the half line $[\frac{\rho_1}{2}e^{i\gamma_{p_1}},\infty)$ (resp. $[\frac{\rho_1}{2}e^{i\gamma_{p'_1}},\infty)$).

\begin{figure}
	\centering
	\includegraphics[width=0.3\textwidth]{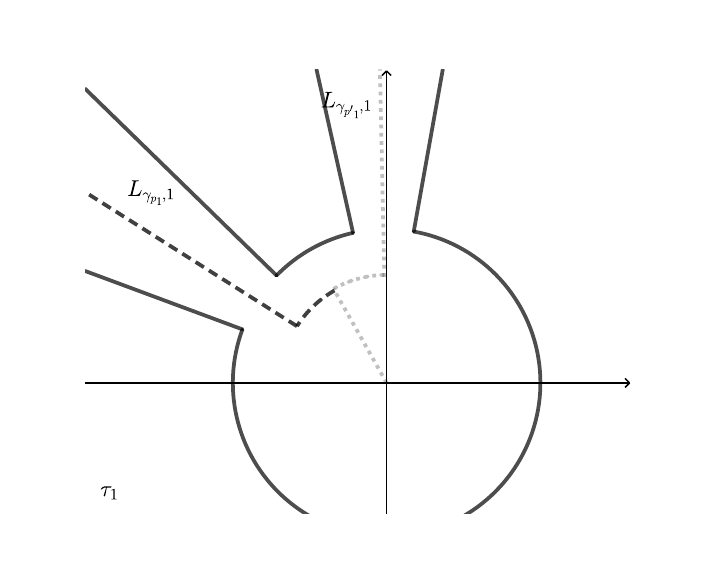} \vline\vline \includegraphics[width=0.3\textwidth]{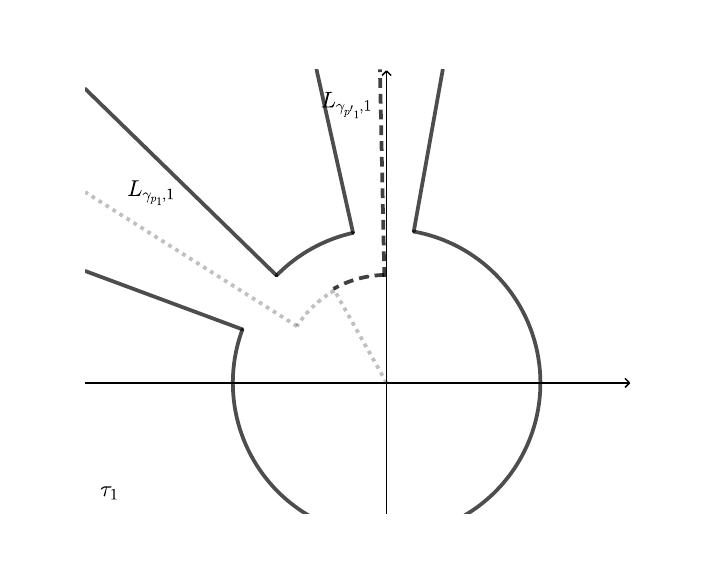} \vline\vline \includegraphics[width=0.3\textwidth]{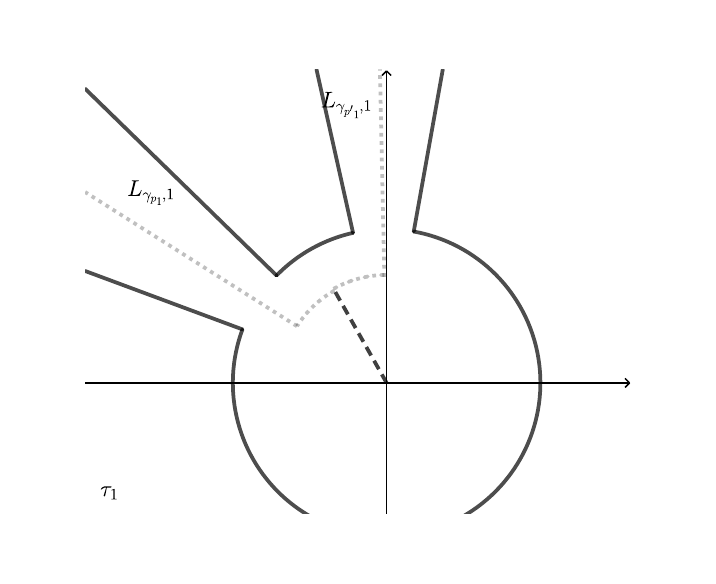}\\
	\includegraphics[width=0.3\textwidth]{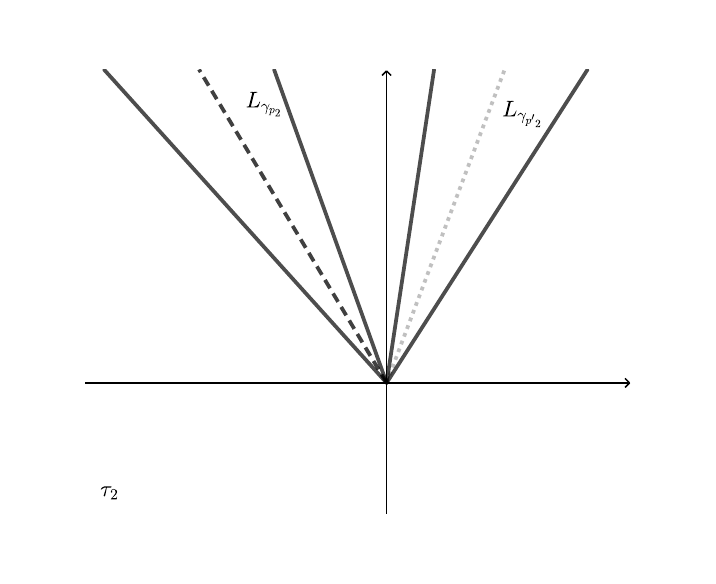} \vline\vline \includegraphics[width=0.3\textwidth]{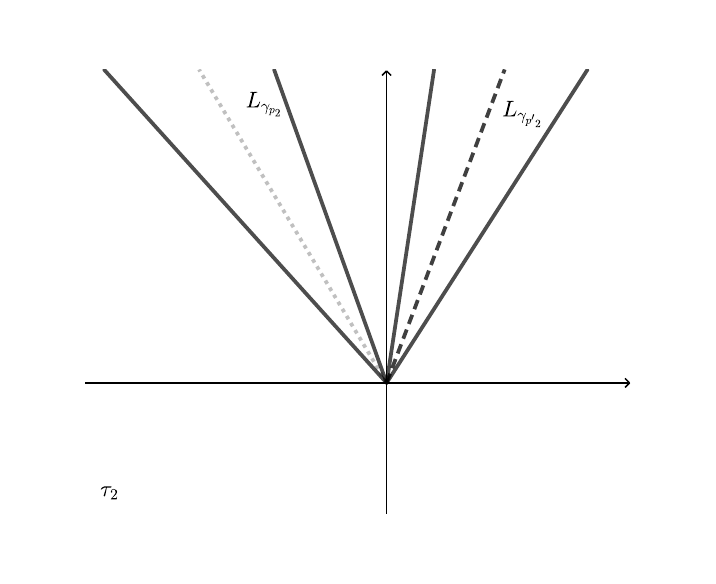} \vline\vline \includegraphics[width=0.3\textwidth]{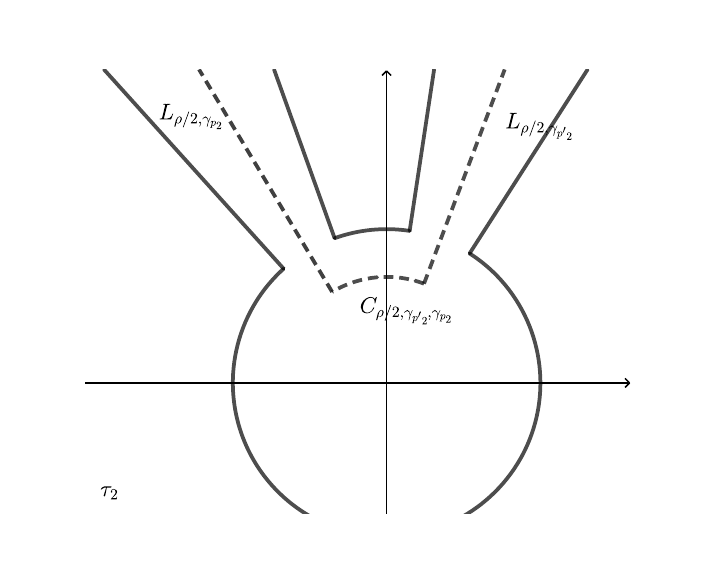}
	\caption{Path deformation in Case 3}
\end{figure}

We first give estimates for $|J_1|$. We have
\begin{multline*}
\left|\int_{L_{\gamma_{p_2}}} \omega_{\bk}^{\mathfrak{d}_{p_1},\tilde{\mathfrak{d}}_{p_2}}(u_1,u_2,m,\epsilon)
e^{-(\frac{u_2}{\epsilon t_2})^{k_2}}\frac{du_2}{u_2}\right|\le \varpi_{\mathfrak{d}_{p_1},\tilde{\mathfrak{d}}_{p_2}}(1+ |m|)^{-\mu} e^{-\beta|m|} \frac{ |\frac{u_1}{\epsilon}|}{1 + |\frac{u_1}{\epsilon}|^{2k_1}}
\exp( \nu_1 |\frac{u_1}{\epsilon}|^{k_1})\\
\times \int_{L_{\gamma_{p_2}}}\left(\frac{ |\frac{u_2}{\epsilon}|}{1 + |\frac{u_2}{\epsilon}|^{2k_2}}\exp(\nu_2\left|\frac{u_2}{\epsilon}\right|^{k_2}\right)|e^{-\left(\frac{u_2}{\epsilon t_2}\right)^{k_2}}|\left|\frac{du_2}{u_2}\right|\\
\le \varpi_{\mathfrak{d}_{p_1},\tilde{\mathfrak{d}}_{p_2}}C_{p_2}(1+ |m|)^{-\mu} e^{-\beta|m|} \frac{ |\frac{u_1}{\epsilon}|}{1 + |\frac{u_1}{\epsilon}|^{2k_1}}
\exp( \nu_1 |\frac{u_1}{\epsilon}|^{k_1}),
\end{multline*}
for some $C_{p_2}>0$, and $t_2\in\mathcal{T}_2\cap D(0,h')$. Using the parametrization $u_2=re^{\gamma_{p_2}\sqrt{-1}}$ and the change of variable $r=|\epsilon|s$. Using analogous estimations as in the Case 1, we arrive at
$$|J_1|\le C_{p,1}e^{-\frac{M_{p,1}}{|\epsilon|^{k_1}}},$$
for some $C_{p,1},M_{p,1}>0$, for all $\epsilon\in\mathcal{E}_{p_1,p_2}\cap \mathcal{E}_{p'_1,p'_2}$, where $t_1\in\mathcal{T}_{1}\cap D(0,h')$ and $t_2\in\mathcal{T}_{2}\cap D(0,h')$, $z\in H_{\beta'}$.

Analogous calculations yield to 
$$|J_2|\le C_{p,2}e^{-\frac{M_{p,2}}{|\epsilon|^{k_1}}},$$
for some $C_{p,2},M_{p,2}>0$, for all $\epsilon\in\mathcal{E}_{p_1,p_2}\cap \mathcal{E}_{p'_1,p'_2}$, where $t_1\in\mathcal{T}_{1}\cap D(0,h')$ and $t_2\in\mathcal{T}_{2}\cap D(0,h')$, $z\in H_{\beta'}$.

In order to give upper bounds for $|J_3|$, we consider
$$\left|\int_{L_{\gamma_{p_2}}} \omega_{\bk}^{\mathfrak{d}_{p_1},\tilde{\mathfrak{d}}_{p_2}}(u_1,u_2,m,\epsilon)
e^{-(\frac{u_2}{\epsilon t_2})^{k_2}} \frac{du_2}{u_2}-\int_{L_{\gamma_{p'_2}}} \omega_{\bk}^{\mathfrak{d}_{p'_1},\tilde{\mathfrak{d}}_{p'_2}}(u_1,u_2,m,\epsilon)
e^{-(\frac{u_2}{\epsilon t_2})^{k_2}} \frac{du_2}{u_2}\right|.$$

Since $u_1$ belongs to the disc $D(0,\rho_1)$, we know that the function 
$$u_2\mapsto \omega_{\bk}^{\mathfrak{d}_{p_1},\tilde{\mathfrak{d}}_{p_2}}(u_1,u_2,m,\epsilon) e^{-(\frac{u_2}{\epsilon t_2})^{k_2}}\frac{1}{u_2}$$
is holomorphic on the disc $D(0,\rho)$. In this framework, one is able to deform the integration path in order to write the difference as the next sum

\begin{multline*}
\int_{L_{\rho_1/2,\gamma_{p_2}}}
\omega_{\bk}^{\mathfrak{d}_{p_1},\tilde{\mathfrak{d}}_{p_2}}(u_1,u_2,m,\epsilon) e^{-(\frac{u_2}{\epsilon t_2})^{k_2}} \frac{du_2}{u_2} \\ 
-\int_{L_{\rho_1/2,\gamma_{p'_2}}}
\omega_{\bk}^{\mathfrak{d}_{p_1},\tilde{\mathfrak{d}}_{p_2}}(u_1,u_2,m,\epsilon) e^{-(\frac{u_2}{\epsilon t_2})^{k_2}} \frac{du_2}{u_2}\\
+ \int_{C_{\rho_1/2,\gamma_{p'_2},\gamma_{p_2}}}
\omega_{\bk}^{\mathfrak{d}_{p_1},\tilde{\mathfrak{d}}_{p_2}}(u_1,u_2,m,\epsilon) e^{-(\frac{u_2}{\epsilon t_2})^{k_2}} \frac{du_2}{u_2}.
\end{multline*}
We get the previous expression is upper estimated by
$$\varpi_{\mathfrak{d}_{p_1},\tilde{\mathfrak{d}}_{p_2}}C_{p_2,p'_2}(1+ |m|)^{-\mu} e^{-\beta|m|} \frac{ |\frac{u_1}{\epsilon}|}{1 + |\frac{u_1}{\epsilon}|^{2k_1}}
\exp( \nu_1 |\frac{u_1}{\epsilon}|^{k_1})\exp\left(-\frac{M_{p_2,p'_2}}{|\epsilon|^{k_2}}\right),$$
for $\epsilon\in\mathcal{E}_{p_1,p_2}\cap \mathcal{E}_{p'_1,p'_2}$,  $t_2\in\mathcal{T}_{2}\cap D(0,h')$, $u_1\in[0,\rho_1/2e^{i\theta}]$.
We finally get 
\begin{multline*}
|J_3|\le\frac{k_1k_2}{(2\pi)^{1/2}} C_{p_2,p'_2}\varpi_{\mathfrak{d}_{p_1},\tilde{\mathfrak{d}}_{p_2}}    \left(\int_{-\infty}^{\infty}(1+|m|)^{-\mu}e^{-\beta|m|}e^{-m|\hbox{Im}(z)|}dm\right)\\
\times\left(\int_{0}^{\rho_1/2e^{i\theta}}\frac{ |\frac{u_1}{\epsilon}|}{1 + |\frac{u_1}{\epsilon}|^{2k_1}}\exp(\nu_1|\frac{u_1}{\epsilon}|^{k_1})|e^{-\left(\frac{u_1}{\epsilon t_1}\right)^{k_1}}|\left|\frac{d u_1}{u_1}\right|\right)\exp\left(-\frac{M_{p_2,p'_2}}{|\epsilon|^{k_2}}\right).
\end{multline*}
We conclude that 
$$|J_3|\le K_{p,3}e^{-\frac{M_{p,3}}{|\epsilon|^{k_2}}},$$
uniformly for $(t_1,t_2)\in (\mathcal{T}_1\cap D(0,h''))\times (\mathcal{T}_2\cap D(0,h''))$ for some $h''>0$, and $z\in H_{\beta'}$ for any fixed $\beta'<\beta$, where $K_{\boldsymbol{p},3},M_{\boldsymbol{p},3}$ are positive constants.

\end{proof}

\textbf{Remark:} Observe that, in case that the path $L_{\gamma_{p_1}}$ coincides with $L_{\gamma_{p'_1}}$, but $L_{\gamma_{p_2}}$ does not coincide with $L_{\gamma_{p'_2}}$, then it is not possible to obtain estimates on the difference of two solutions in the form $\exp(-M/|\epsilon|^{k_2})$, as it happens in Case 2. The reason is that we can not deform the path $L_{\gamma_{p_2}}-L_{\gamma_{p'_2}}$ since the function $w_{\bk}^{\mathfrak{d}_{p_1},\tilde{\mathfrak{d}}_{p_2}}(\btau,m,\epsilon)$ and $w_{\bk}^{\mathfrak{d}_{p'_1},\tilde{\mathfrak{d}}_{p'_2}}(\btau,m,\epsilon)$ are not holomorphic on a disc centered at 0 respect to $\tau_2$.\par

\section{Asymptotics of the problem in the perturbation parameter}

\subsection{$k-$Summable formal series and Ramis-Sibuya Theorem}\label{secborellaplace}

For the sake of completeness, we recall the definition of $k-$Borel summability of formal series with coefficients in a Banach space, and Ramis-Sibuya Theorem. A reference for the details on the first part is~\cite{ba}, whilst the second part of this section can be found in~\cite{ba2}, p. 121, and~\cite{hssi}, Lemma XI-2-6.
 
\begin{defin} Let $k \geq 1$ be an integer. A formal series
$$\hat{X}(\epsilon) = \sum_{j=0}^{\infty}  \frac{ a_{j} }{ j! } \epsilon^{j} \in \mathbb{F}[[\epsilon]]$$
with coefficients in a Banach space $( \mathbb{F}, ||.||_{\mathbb{F}} )$ is said to be $k-$summable
with respect to $\epsilon$ in the direction $d \in \mathbb{R}$ if \medskip

{\bf i)} there exists $\rho \in \mathbb{R}_{+}$ such that the following formal series, called formal
Borel transform of $\hat{X}$ of order $k$ 
$$ \mathcal{B}_{k}(\hat{X})(\tau) = \sum_{j=0}^{\infty} \frac{ a_{j} \tau^{j}  }{ j!\Gamma(1 + \frac{j}{k}) } \in \mathbb{F}[[\tau]],$$
is absolutely convergent for $|\tau| < \rho$, \medskip

{\bf ii)} there exists $\delta > 0$ such that the series $\mathcal{B}_{k}(\hat{X})(\tau)$ can be analytically continued with
respect to $\tau$ in a sector
$S_{d,\delta} = \{ \tau \in \mathbb{C}^{\ast} : |d - \mathrm{arg}(\tau) | < \delta \} $. Moreover, there exist $C >0$, and $K >0$
such that
$$ ||\mathcal{B}(\hat{X})(\tau)||_{\mathbb{F}}
\leq C e^{ K|\tau|^{k}} $$
for all $\tau \in S_{d, \delta}$.
\end{defin}
If this is so, the vector valued Laplace transform of order $k$ of $\mathcal{B}_{k}(\hat{X})(\tau)$ in the direction $d$ is defined by
$$ \mathcal{L}^{d}_{k}(\mathcal{B}_{k}(\hat{X}))(\epsilon) = \epsilon^{-k} \int_{L_{\gamma}}
\mathcal{B}_{k}(\hat{X})(u) e^{ - ( u/\epsilon )^{k} } ku^{k-1}du,$$
along a half-line $L_{\gamma} = \mathbb{R}_{+}e^{i\gamma} \subset S_{d,\delta} \cup \{ 0 \}$, where $\gamma$ depends on
$\epsilon$ and is chosen in such a way that
$\cos(k(\gamma - \mathrm{arg}(\epsilon))) \geq \delta_{1} > 0$, for some fixed $\delta_{1}$, for all
$\epsilon$ in a sector
$$ S_{d,\theta,R^{1/k}} = \{ \epsilon \in \mathbb{C}^{\ast} : |\epsilon| < R^{1/k} \ \ , \ \ |d - \mathrm{arg}(\epsilon) |
< \theta/2 \},$$
where $\frac{\pi}{k} < \theta < \frac{\pi}{k} + 2\delta$ and $0 < R < \delta_{1}/K$. The
function $\mathcal{L}^{d}_{k}(\mathcal{B}_{k}(\hat{X}))(\epsilon)$
is called the $k-$sum of the formal series $\hat{X}(t)$ in the direction $d$. It is bounded and holomorphic on the sector
$S_{d,\theta,R^{1/k}}$ and has the formal series $\hat{X}(\epsilon)$ as Gevrey asymptotic
expansion of order $1/k$ with respect to $\epsilon$ on $S_{d,\theta,R^{1/k}}$. This means that for all
$\frac{\pi}{k} < \theta_{1} < \theta$, there exist $C,M > 0$ such that
$$ ||\mathcal{L}^{d}_{k}(\mathcal{B}_{k}(\hat{X}))(\epsilon) - \sum_{p=0}^{n-1}
\frac{a_p}{p!} \epsilon^{p}||_{\mathbb{F}} \leq CM^{n}\Gamma(1+ \frac{n}{k})|\epsilon|^{n} $$
for all $n \geq 1$, all $\epsilon \in S_{d,\theta_{1},R^{1/k}}$.\medskip

 Multisummability of a formal power series is a recursive process that allows to compute the sum of a formal power series in different Gevrey orders. One of the approaches to multisummability is that stated by W. Balser, which can be found in~\cite{ba}, Theorem 1, p.57. Roughly speaking, given a formal power series $\hat{f}$ which can be decomposed into a sum $\hat{f}(z)=\hat{f}_1(z)+\ldots+\hat{f}_m(z)$ such that each of the terms $\hat{f}_j(z)$ is $k_j$-summable, with sum given by $f_j$, then, $\hat{f}$ turns out to be multisummable, and its multisum is given by $f_1(z)+\ldots+f_m(z)$. More precisely, one has the following definition.

\begin{defin} Let $(\mathbb{F},\left\|\cdot\right\|_{\mathbb{F}})$ be a complex Banach space and let $0<k_2<k_1$. Let $\mathcal{E}$ be a bounded open sector with vertex at 0, and opening $\frac{\pi}{k_1}+\delta_1$ for some $\delta_1>0$, and let $\mathcal{F}$ be a bounded open sector with vertex at the origin in $\C$, with opening $\frac{\pi}{k_2}+\delta_2$, for some $\delta_2>0$ and such that $\mathcal{E}\subseteq\mathcal{F}$ holds.\smallskip

A formal power series $\hat{f}(\epsilon)\in\mathbb{F}[[\epsilon]]$ is said to be $(k_1,k_2)-$summable on $\mathcal{E}$ if there exist $\hat{f}_2(\epsilon)\in\mathbb{F}[[\epsilon]]$ which is $k_2-$summable on $\mathcal{F}$, with $k_2$-sum given by $f_2:\mathcal{F}\to\mathbb{F}$, and $\hat{f}_1(\epsilon)\in\mathbb{F}[[\epsilon]]$ which is $k_1-$summable on $\mathcal{E}$, with $k_1$-sum given by $f_1:\mathcal{E}\to\mathbb{F}$, such that $\hat{f}=\hat{f}_1+\hat{f}_2$. Furthermore, the holomorphic function $f(\epsilon)=f_1(\epsilon)+f_2(\epsilon)$ on $\mathcal{E}$ is called the $(k_1,k_2)-$sum of $\hat{f}$ on $\mathcal{E}$. In that situation, $f(\epsilon)$ can be obtained from the analytic continuation of the $k_2-$Borel transform of $\hat{f}$ by the successive application of accelerator operators and Laplace transform of order $k_1$, see Section 6.1 in~\cite{ba}.
\end{defin}

We recall the reader the classical version of Ramis-Sibuya Theorem for the Gevrey asymptotics as stated in~\cite{hssi} in the framework of our good covering $\{ \mathcal{E}_{p_1,p_2} \}_{\begin{subarray}{l} 0 \leq p_1 \leq \varsigma_1 - 1\\0 \leq p_2 \leq \varsigma_2 - 1\end{subarray}}$, given above in Definition 3.

\noindent {\bf Theorem (RS)} {\it Let $0<k_1<k_2$ be integer numbers. Let $(\mathbb{F},||.||_{\mathbb{F}})$ be a Banach space over $\mathbb{C}$ and
$\{ \mathcal{E}_{p_1,p_2} \}_{\begin{subarray}{l} 0 \leq p_1 \leq \varsigma_1 - 1\\0 \leq p_2 \leq \varsigma_2 - 1\end{subarray}}$ be a good covering in $\mathbb{C}^{\ast}$, such that the aperture of every sector is slightly larger than $\pi/k_2$. For all
$0 \leq p_1 \leq \varsigma_1 - 1$, $0\le p_2\le \varsigma_2-1$, let $G_{p_1,p_2}$ be a holomorphic function from $\mathcal{E}_{p_1,p_2}$ into
the Banach space $(\mathbb{F},||.||_{\mathbb{F}})$ and let the cocycle $\Theta_{(p_1,p_2)(p'_1,p'_2)}(\epsilon) = G_{p_1,p_2}(\epsilon) - G_{p'_1,p'_2}(\epsilon)$
be a holomorphic function from the sector $Z_{(p_1,p_2),(p'_1,p'_2)} = \mathcal{E}_{p_1,p_2} \cap \mathcal{E}_{p'_1,p'_2}\neq\emptyset$ into $\mathbb{E}$. We make the following assumptions.\medskip

\noindent {\bf 1)} The functions $G_{p_1,p_2}(\epsilon)$ are bounded as $\epsilon \in \mathcal{E}_{p_1,p_2}$ tends to the origin
in $\mathbb{C}$, for all $0 \leq p_1 \leq \varsigma_1 - 1$ and all $0\le p_2\le \varsigma_2-1$.\medskip

\noindent {\bf 2)} The functions $\Theta_{(p_1,p_2)(p'_1,p'_2)}(\epsilon)$ are exponentially flat of order $1/k_1$ on $Z_{(p_1,p_2)(p'_1,p'_2)}$, for all
$0 \leq p_1,p'_1 \leq \varsigma_1-1$,  and $0\le p_2,p'_2 \le \varsigma_2-1$. This means that there exist constants $C_{p_1,p_2,p'_1,p'_2},A_{p_1,p_2,p'_1,p'_2}>0$ such that
$$ ||\Theta_{(p_1,p_2)(p'_1,p'_2)}(\epsilon)||_{\mathbb{F}} \leq C_{p_1,p_2,p'_1,p'_2}e^{-A_{p_1,p_2,p'_1p'_2}/|\epsilon|^{k_1}} $$
for all $\epsilon \in Z_{(p_1,p_2),(p'_1,p'_2)}$, all $0 \leq p_1,p'_1 \leq \varsigma_1-1$ and $0\le p_2,p'_2\le \varsigma_2-1$.\medskip

Then, for all $0 \leq p_1 \leq \nu_1 - 1$ and $0\le p_2\le\varsigma_2-1$, the functions $G_{p_1,p_2}(\epsilon)$ admit a common formal power series $\hat{G}(\epsilon) \in \mathbb{F}[[\epsilon]]$ as asymptotic expansion of Gevrey order $1/k_1$.}

A novel version of Ramis-Sibuya Theorem has been developed in~\cite{takei}, and has provided successful results in previous works by the authors,~\cite{lama1},~\cite{lama2,family1}. A version of the result in two different levels which fits our needs is now given without proof, which can be found in~\cite{lama1},~\cite{lama2}.

\vspace{0.3cm}

\noindent {\bf Theorem (multilevel-RS)} {\it Assume that $0<k_2<k_1$ are integer numbers. Let $(\mathbb{F},||.||_{\mathbb{F}})$ be a Banach space over $\mathbb{C}$ and
$\{ \mathcal{E}_{p_1,p_2} \}_{\begin{subarray}{l} 0 \leq p_1 \leq \varsigma_1 - 1\\0 \leq p_2 \leq \varsigma_2 - 1\end{subarray}}$ be a good covering in $\mathbb{C}^{\ast}$, where all the sectors have an opening slightly larger than $\pi/k_1$. For all
$0 \leq p_1 \leq \varsigma_1 - 1$ and $0\le p_2\le \varsigma_2-1$, let $G_{p_1,p_2}$ be a holomorphic function from $\mathcal{E}_{p_1,p_2}$ into
the Banach space $(\mathbb{F},||.||_{\mathbb{F}})$ and for every $(p_1,p_2),(p'_1,p'_2)\in\{0,\ldots,\varsigma_1-1\}\times\{0,\ldots,\varsigma_2-1\}$ such that $\mathcal{E}_{p_1,p_2}\cap\mathcal{E}_{p'_1,p'_2}\neq\emptyset$ we define $\Theta_{(p_1,p_2)(p'_1,p'_2)}(\epsilon) = G_{p_1,p_2}(\epsilon) - G_{p'_1,p'_2}(\epsilon)$
be a holomorphic function from the sector $Z_{(p_1,p_2),(p'_1,p'_2)} = \mathcal{E}_{p_1,p_2} \cap \mathcal{E}_{p'_1,p'_2}$ into $\mathbb{F}$.
We make the following assumptions.\medskip

\noindent {\bf 1)} The functions $G_{p_1,p_2}(\epsilon)$ are bounded as $\epsilon \in \mathcal{E}_{p_1,p_2}$ tends to the origin
in $\mathbb{C}$, for all $0 \leq p_1 \leq \varsigma_1 - 1$ and $0\le p_2\le \varsigma_2-1$.\medskip

\noindent {\bf 2)} $(\{0,\ldots,\varsigma_1-1\}\times\{0,\ldots,\varsigma_2\})^2=\mathcal{U}_0\cup\mathcal{U}_{k_1}\cup\mathcal{U}_{k_2}$, where 

$((p_1,p_2),(p'_1,p'_2))\in\mathcal{U}_0$ iff $\mathcal{E}_{p_1,p_2}\cap\mathcal{E}_{p'_1,p'_2}=\emptyset$,

$((p_1,p_2),(p'_1,p'_2))\in\mathcal{U}_{k_1}$ iff $\mathcal{E}_{p_1,p_2}\cap\mathcal{E}_{p'_1,p'_2}\neq\emptyset$ and 
$$ ||\Theta_{(p_1,p_2),(p'_1,p'_2)}(\epsilon)||_{\mathbb{F}} \leq C_{p_1,p_2,p'_1,p'_2}e^{-A_{p_1,p_2,p'_1,p'_2}/|\epsilon|^{k_1}} $$
for all $\epsilon \in Z_{(p_1,p_2),(p'_1,p'_2)}$.

$((p_1,p_2),(p'_1,p'_2))\in\mathcal{U}_{k_2}$ iff $\mathcal{E}_{p_1,p_2}\cap\mathcal{E}_{p'_1,p'_2}\neq\emptyset$ and 
$$ ||\Theta_{(p_1,p_2),(p'_1,p'_2)}(\epsilon)||_{\mathbb{F}} \leq C_{p_1,p_2,p'_1,p'_2}e^{-A_{p_1,p_2,p'_1,p'_2}/|\epsilon|^{k_2}} $$
for all $\epsilon \in Z_{(p_1,p_2),(p'_1,p'_2)}$.

Then, there exists a convergent power series $a(\epsilon)\in \mathbb{F}\{\epsilon\}$ and two formal power series $\hat{G}^1(\epsilon),\hat{G}^2(\epsilon)\in\mathbb{F}[[\epsilon]]$ such that $G_{p_1,p_2}(\epsilon)$ can be split in the form
$$G_{p_1,p_2}(\epsilon)=a(\epsilon)+G^{1}_{p_1,p_2}(\epsilon)+G^{2}_{p_1,p_2}(\epsilon),$$
where $G^{j}_{p_1,p_2}(\epsilon)\in\mathcal{O}(\mathcal{E}_{p_1,p_2},\mathbb{F})$, and admits $\hat{G}^j(\epsilon)$ as its asymptotic expansion of Gevrey order $1/k_j$ on $\mathcal{E}_{p_1,p_2}$, for $j\in\{1,2\}$.\smallskip

Moreover, assume that
$$\{((p_1^0,p_2^0),(p_1^1,p_2^1)), ((p_1^1,p_2^1),(p_1^2,p_2^2)), \ldots, ((p_1^{2y-1},p_2^{2y-1}),(p_1^{2y},p_2^{2y}))     \}$$
is a subset of $\mathcal{U}_{k_1}$, for some positive integer $y$, and
$$\mathcal{E}_{p_1^{y},p_2^y}\subseteq S_{\pi/k_2}\subseteq\bigcup_{0\le j\le 2y}\mathcal{E}_{p_1^{j},p_2^{j}},$$
for some sector $S_{\pi/k_2}$ with opening larger than $\pi/k_2$. Then, the formal power series $\hat{G}(\epsilon)$ is $(k_1,k_2)-$summable on $\mathcal{E}_{p_1^y,p_2^y}$ and its $(k_1,k_2)-$sum is $G_{p_1^y,p_2^y}(\epsilon)$ on $\mathcal{E}_{p_1^y,p_2^y}$.}




\subsection{Formal solution and asymptotic behavior in the complex parameter}

The second and third main results state the existence of a formal power series in the perturbation parameter $\epsilon$, with coefficients in the Banach space $\mathbb{F}$ of holomorphic and bounded functions on $(\mathcal{T}_1\cap D(0,h''))\times (\mathcal{T}_2\cap D(0,h''))\times H_{\beta'}$, with the norm of the supremum. Here $h''$, $\mathcal{T}_1,\mathcal{T}_2$ are determined in Theorem~\ref{teo1}.

It is worth observing the different asymptotic behavior of the analytic solutions of the problem depending on $k_1$ and $k_2$. More precisely, in case that $k_1<k_2$, Theorem~\ref{teo2} shows a Gevrey estimates occurrence, whilst $k_2<k_1$ displays a multisummability phenomenon; in contrast to the results observed in~\cite{family1}, where multisummability is always observed. 

\begin{theo}\label{teo2}
Let $k_2>k_1$. Under the assumptions of Theorem~\ref{teo1}, a formal power series
$$\hat{u}(\bt,z,\epsilon)=\sum_{m\ge0}H_m(\bt,z)\epsilon^m/m!\in\mathbb{F}[[\epsilon]]$$
exists, with the following properties. $\hat{u}$ is a formal solution of (\ref{ICP_main0}). Moreover, for every $p_1\in\{0,\ldots,\varsigma_1-1\}$ and $p_2\in\{0,\ldots,\varsigma_2-1\}$, the function $u_{p_1,p_2}(\bt,z,\epsilon)$ admits $\hat{u}(\bt,z,\epsilon)$ as asymptotic expansion of Gevrey order $1/k_1$. This means that 
$$\sup_{\bt\in (\mathcal{T}_1\cap D(0,h''))\times (\mathcal{T}_2\cap D(0,h'')),z\in H_{\beta'}}\left|u_{p_1,p_2}(\bt,z,\epsilon)-\sum_{m=0}^{N-1}h_m(\bt,z)\frac{\epsilon^m}{m!}\right|\le CM^{N}\Gamma(1+\frac{N}{k_1})|\epsilon|^{N},$$
for every $\epsilon\in\mathcal{E}_{p_1,p_2}$ and all integer $N\ge0$.\smallskip
\end{theo}
\begin{proof}
Let $(u_{p_1,p_2}(\bt,z,\epsilon))_{\begin{subarray}{l}0\le p_1\le \varsigma_1-1\\0\le p_2\le \varsigma_2-1\end{subarray}}$ be the family constructed in Theorem~\ref{teo1}. We recall that $(\mathcal{E}_{p_1,p_2})_{\begin{subarray}{l}0\le p_1\le \varsigma_1-1\\0\le p_2\le \varsigma_2-1\end{subarray}}$ is a good covering in $\C^{\star}$, with all its components being finite sectors of opening slightly larger than $\pi/k_2$. 

The function $G_{p_1,p_2}(\epsilon):=(t_1,t_2,z)\mapsto u_{p_1,p_2}(t_1,t_2,z,\epsilon)$ belongs to $\mathcal{O}(\mathcal{E}_{p_1,p_2},\mathbb{F})$. We consider $\{(p_1,p_2),(p'_1,p'_2)\}$ such that $(p_1,p_2)$ and $(p'_1,p'_2)$ belong to $\{0,\ldots,\varsigma_1-1\}\times \{0,\ldots,\varsigma_2-1\}$, and $\mathcal{E}_{p_1,p_2}$ and $\mathcal{E}_{p'_1,p'_2}$ are consecutive sectors in the good covering, so their intersection is not empty. In view of (\ref{exp_small_difference_u_p11}) and (\ref{exp_small_difference_u_p12}), one has that $\Delta_{(p_1,p_2),(p'_1,p'_2)}(\epsilon):=G_{p_1,p_2}(\epsilon)- G_{p'_1,p'_2}(\epsilon)$ satisfies exponentially flat bounds of Gevrey order $k_1$, due to $\mathcal{U}_{k_1}$ coincides with $\{0,\ldots,\varsigma_1\}\times\{0,\ldots,\varsigma_2\}$. Ramis-Sibuya Theorem guarantees the existence of a formal power series $\hat{G}(\epsilon)\in\mathbb{F}[[\epsilon]]$ such that $G_{p_1,p_2}$ admits $\hat{G}(\epsilon)$ as its Gevrey asymptotic expansion of order $k_1$, say
$$\hat{G}(\epsilon)=:\hat{u}(\bt,z,\epsilon)=\sum_{m\ge0}H_m(\bt,z)\frac{\epsilon^{m}}{m!}.$$
Let us check that $\hat{u}(\bt,z,\epsilon)$ is a formal solution of (\ref{ICP_main0}). For every $0\le p_1\le \varsigma_1-1$, $0\le p_2\le \varsigma_2-1$, the existence of an asymptotic expansion concerning $G_{p_1,p_2}(\epsilon)$ and $\hat{G}(\epsilon)$ implies that
\begin{equation}\label{e1363}
\lim_{\epsilon\to 0,\epsilon\in\mathcal{E}_{p_1,p_2}}\sup_{(\bt,z)\in(\tau_1\cap D(0,h''))\times (\tau_2\cap D(0,h''))\times H_{\beta'}}|\partial_{\epsilon}^{\ell}u_{p_1,p_2}(\bt,z,\epsilon)-H_{\ell}(\bt)|=0,
\end{equation} 
for every $\ell\in\mathbb{N}$. By construction, the function $u_{p_1,p_2}(\bt,z,\epsilon)$ is a solution of (\ref{ICP_main0}). Taking derivatives of order $m\ge0$ with respect to $\epsilon$ on that equation yield

\begin{multline}\label{e1367}
Q(\partial_{z})\partial_{t_2}+\sum_{m_1+m_2=m}\partial_\epsilon^{m_1}(\epsilon^{\tilde{\Delta}_2})t_2^{\tilde{d}_2}\partial_{t_2}^{\tilde{\delta}_{D_2}}R_{D_1,D_2}(\partial_z)\partial_\epsilon^{m_2}u_{p_1,p_2}(\bt,z,\epsilon)\\
+\sum_{m_1+m_2=m}\partial_\epsilon^{m_1}(\epsilon^{\tilde{\Delta}_3})t_2^{\tilde{d}_3}\partial_{t_2}^{\tilde{\delta}_{D_3}}R_{D_3}(\partial_z)\partial_\epsilon^{m_2}u_{p_1,p_2}(\bt,z,\epsilon)\\
 =\sum_{m_1+m_2=m}\frac{m!}{m_1!m_2!}\left(\sum_{m_{11}+m_{12}=m_1}\frac{m_1!}{m_{11}!m_{12}!}\partial_\epsilon^{m_{11}}P_1(\partial_z,\epsilon)\partial_{\epsilon}^{m_{12}}u_{p_1,p_2}(\bt,z,\epsilon)\right)\\
\times \left(\sum_{m_{21}+m_{22}=m_2}\frac{m_2!}{m_{21}!m_{22}!}\partial_\epsilon^{m_{21}}P_2(\partial_z,\epsilon)\partial_{\epsilon}^{m_{22}}u_{p_1,p_2}(\bt,z,\epsilon)\right)\\
+\sum_{\stackrel{0\le l_j\le D_j}{j=1,2}}\left(\sum_{m_{1}+m_2=m}\frac{m!}{m_1!m_2!}\partial_\epsilon^{m_1}(\epsilon^{\Delta_{l_1,l_2}})t_1^{d_{l_1}}t_2^{\tilde{d}_{l_2}}\partial_{t_1}^{\delta_{l_1}}\partial_{t_2}^{\tilde{\delta}_{l_2}}R_{l_1,l_2}(\partial_z)\partial_\epsilon^{m_2}u_{p_1,p_2}(\bt,z,\epsilon)\right)\\
+ \partial_{\epsilon}^{m}f(\bt,z,\epsilon),
\end{multline}
for every $m\ge 0$ and $(\bt,z,\epsilon)\in (\mathcal{T}_1\cap D(0,h''))\times (\mathcal{T}_2\cap D(0,h''))\times H_{\beta'}\times\mathcal{E}_{p_1,p_2}$. Tending $\epsilon\to 0$ in (\ref{e1367}) together with (\ref{e1363}), we obtain a recursion formula for the coefficients of the formal solution.

\begin{multline}\label{e1368}
Q_1(\partial_{z})Q_2(\partial_z)\partial_{t_1}\partial_{t_2}H_{m}(\bt,z)+\frac{m!}{(m-\tilde{\Delta}_2)!}t_2^{\tilde{d}_2}\partial_{t_2}^{\tilde{\delta}_{D_2}}R_{D_1,D_2}(\partial_z)H_{m-\tilde{\Delta}_2}(\bt,z)\\
\hfill+\frac{m!}{(m-\tilde{\Delta}_3)!}t_2^{\tilde{d}_3}\partial_{t_2}^{\tilde{\delta}_{D_3}}R_{D_3}(\partial_z)H_{m-\tilde{\Delta}_3}(\bt,z)\\
=\sum_{m_{1}+m_{2}=m} \frac{m!}{m_{1}!m_{2}!}\left(\sum_{m_{11}+m_{12}=m_1}\frac{m_1!}{m_{11}!m_{12}!}\partial_\epsilon^{m_{11}}P_1(\partial_z,0)H_{m_{12}}(\bt,z)\right)\hfill\\
\times \left(\sum_{m_{21}+m_{22}=m_2}\frac{m_2!}{m_{21}!m_{22}!}\partial_\epsilon^{m_{21}}P_2(\partial_z,0)H_{m_{12}}(\bt,z)\right)\\
+\sum_{0\le l_1\le D_1,0\le l_2\le D_2}\frac{m!}{(m-\Delta_{l_1,l_2})!}t_1^{d_{l_1}}t_2^{\tilde{d}_{l_2}}\partial_{t_1}^{\delta_{l_1}}\partial_{t_2}^{\tilde{\delta}_{l_2}}R_{l_1,l_2}(\partial_z)H_{m-\Delta_{l_1,l_2}}(\bt,z)\\
+ \partial_{\epsilon}^{m}f(\bt,z,0),
\end{multline}
for every $m\ge \max\{\max_{1\le l_1\le D_1,1\le l_2\le D_2}\Delta_{l_1,l_2},\tilde{\Delta}_2,\tilde{\Delta}_3\}$, and $(\bt,z,\epsilon)\in (\mathcal{T}_1\cap D(0,h''))\times (\mathcal{T}_2\cap D(0,h''))\times H_{\beta'}$. From the analyticity of $f$ with respect to $\epsilon$ in a vicinity of the origin we get
\begin{equation}
f(\bt,z,\epsilon) = \sum_{m \geq 0} \frac{(\partial_{\epsilon}^{m}f)(\bt,z,0)}{m!}\epsilon^{m}, \label{Taylor_f}
\end{equation}
for every $\epsilon\in D(0,\epsilon_0)$ and $(\bt,z)$ as above. On the other hand, a direct inspection from the recursion formula (\ref{e1368}) and (\ref{Taylor_f}) allow us to affirm that the formal power series $\hat{u}(\bt,z,\epsilon) = \sum_{m \geq 0} H_{m}(\bt,z)\epsilon^{m}/m!$ solves the equation (\ref{ICP_main0}).
\end{proof}

\begin{theo}\label{teo3}
Let $k_1>k_2$. Under the assumptions of Theorem~\ref{teo1}, a formal power series
$$\hat{u}(\bt_1,t_2,z,\epsilon)=\sum_{m\ge0}h_m(t_1,t_2,z)\epsilon^m/m!\in\mathbb{F}[[\epsilon]]$$
exists, with the following properties. $\hat{u}$ is a formal solution of (\ref{ICP_main0}). In addition to that, $\hat{u}$ can be split in the form
$$\hat{u}(\bt,z,\epsilon)=a(\bt,z,\epsilon)+\hat{u}_{1}(\bt,z,\epsilon)+\hat{u}_{2}(\bt,z,\epsilon),$$
where $a(\bt,z,\epsilon)\in\mathbb{F}\{\epsilon\}$, and $\hat{u}_{1},\hat{u}_{2}\in\mathbb{F}[[\epsilon]]$. Moreover, for every $p_1\in\{0,\ldots,\varsigma_1-1\}$ and $p_2\in\{0,\ldots,\varsigma_2-1\}$, the function $u_{p_1,p_2}(\bt,z,\epsilon)$ can be written as
$$u_{p_1,p_2}(\bt,z,\epsilon)=a(\bt,z,\epsilon)+u^1_{p_1,p_2}(\bt,z,\epsilon)+u^2_{p_1,p_2}(\bt,z,\epsilon),$$
where $\epsilon\mapsto u^j_{p_1,p_2}(\bt,z,\epsilon)$ is an $\mathbb{F}-$valued function which admits $\hat{u}_{j}(\bt,z,\epsilon)$ as its $k_j-$Gevrey asymptotic expansion on $\mathcal{E}_{p_1,p_2}$, for $j=1,2$.

Moreover, assume that
$$\{((p_1^0,p_2^0),(p_1^1,p_2^1)), ((p_1^1,p_2^1),(p_1^2,p_2^2)), \ldots, ((p_1^{2y-1},p_2^{2y-1}),(p_1^{2y},p_2^{2y}))     \}$$
is a subset of $\mathcal{U}_{k_1}$, for some positive integer $y$, and
$$\mathcal{E}_{p_1^{y},p_2^y}\subseteq S_{\pi/k_2}\subseteq\bigcup_{0\le j\le 2y}\mathcal{E}_{p_1^{j},p_2^{j}},$$
for some sector $S_{\pi/k_2}$ with opening larger than $\pi/k_2$. Then, $\hat{u}(\bt,z,\epsilon)$ is $(k_1,k_2)-$summable on $\mathcal{E}_{p_1^y,p_2^y}$ and its $(k_1,k_2)-$sum is $u_{p_1^y,p_2^y}(\epsilon)$ on $\mathcal{E}_{p_1^y,p_2^y}$.
\end{theo}

\begin{proof}
Let $(u_{p_1,p_2}(\bt,z,\epsilon))_{\begin{subarray}{l}0\le p_1\le \varsigma_1-1\\0\le p_2\le \varsigma_2-1\end{subarray}}$ be the family constructed in Theorem~\ref{teo1}. In this case, we have
$$\emptyset\neq \mathcal{U}_{k_2}:=\{0,\ldots, \varsigma_1-1\}\times \{0,\ldots, \varsigma_2-1\}\setminus\mathcal{U}_{k_1},$$
and the opening of the sectors in the good covering are of opening slightly larger than $\pi/k_1$.

The function $G_{p_1,p_2}(\epsilon):=(t_1,t_2,z)\mapsto u_{p_1,p_2}(t_1,t_2,z,\epsilon)$ belongs to $\mathcal{O}(\mathcal{E}_{p_1,p_2},\mathbb{F})$. We consider $\{(p_1,p_2),(p'_1,p'_2)\}$ such that $(p_1,p_2)$ and $(p'_1,p'_2)$ belong to $\{0,\ldots,\varsigma_1-1\}\times \{0,\ldots,\varsigma_2-1\}$, and $\mathcal{E}_{p_1,p_2}$ and $\mathcal{E}_{p'_1,p'_2}$ are consecutive sectors in the good covering, so their intersection is not empty. In view of (\ref{exp_small_difference_u_p11}) and (\ref{exp_small_difference_u_p12}), one has that $\Delta_{(p_1,p_2),(p'_1,p'_2)}(\epsilon):=G_{p_1,p_2}(\epsilon)- G_{p'_1,p'_2}(\epsilon)$ satisfies exponentially flat bounds of certain Gevrey order, which is $k_1$ in the case that $\{(p_1,p_2),(p'_1,p'_2)\}\in\mathcal{U}_{k_1}$ and $k_2$ if $\{(p_1,p_2),(p'_1,p'_2)\}\in\mathcal{U}_{k_2}$. Multilevel-RS Theorem guarantees the existence of formal power series $\hat{G}(\epsilon),\hat{G}_1(\epsilon),\hat{G}_2(\epsilon)\in\mathbb{F}[[\epsilon]]$ such that 
$$\hat{G}(\epsilon)=a(\epsilon)+\hat{G}_1(\epsilon)+\hat{G}_{2}(\epsilon),$$ and the splitting
$$G_{p_1,p_2}(\epsilon)=a(\epsilon)+G^1_{p_1,p_2}(\epsilon)+G^2_{p_1,p_2}(\epsilon),$$
for some $a\in\mathbb{F}\{\epsilon\}$, such that for every $(p_1,p_2)\in\{0,\ldots,\varsigma_1-1\}\times \{0,\ldots,\varsigma_2-1\}$, one has that $G^1_{p_1,p_2}(\epsilon)$ admits $\hat{G}_{p_1,p_2}^1(\epsilon)$ as its Gevrey asymptotic expansion of order $k_1$, and $G^2_{p_1,p_2}(\epsilon)$ admits $\hat{G}_{p_1,p_2}^2(\epsilon)$ as its Gevrey asymptotic expansion of order $k_2$. We define
$$\hat{G}(\epsilon)=:\hat{u}(\bt,z,\epsilon)=\sum_{m\ge0}H_m(\bt,z)\frac{\epsilon^{m}}{m!}.$$
Following analogous arguments as in the second part of the proof of Theorem~\ref{teo2}, we conclude that $\hat{u}(\bt,z,\epsilon)$ is a formal solution of (\ref{ICP_main0}). 
\end{proof}

\end{document}